\documentclass[12pt]{amsart}

\usepackage[cp1251]{inputenc}
\usepackage[T2A]{fontenc}
\usepackage[english]{babel}
\usepackage{amsfonts}
\usepackage{amsbsy}
\usepackage{amssymb}
\usepackage{fixltx2e,graphicx}
\usepackage{mathrsfs}
\usepackage{hyperref}
\usepackage{wrapfig}
\usepackage{longtable}
\usepackage{enumitem}

\renewcommand{\abovecaptionskip}{0pt}
\renewcommand{\belowcaptionskip}{6pt}
\makeatletter

\renewcommand{\@makecaption}[2]{
\vspace{\abovecaptionskip}%
\sbox{\@tempboxa}{#1. #2}%
\global\@minipagefalse \hbox to \hsize {{\scshape \hfil #1.
#2\hfil}} \vspace{\belowcaptionskip}}

\makeatother

\newcommand{\CC}{\mathbb C}
\newcommand{\QQ}{\mathbb Q}
\newcommand{\ZZ}{\mathbb Z}
\newcommand{\PP}{\mathbb P}

\renewcommand{\ge}{\geqslant}
\renewcommand{\le}{\leqslant}

\newcommand{\Ker}{\operatorname{Ker}}

\newcommand{\rk}{\operatorname{rk}}

\newcommand{\SL}{\operatorname{SL}}
\newcommand{\GL}{\operatorname{GL}}
\newcommand{\PGL}{\operatorname{PGL}}

\newcommand{\Spin}{\operatorname{Spin}}

\newcommand{\CR}{\operatorname{CR}}
\newcommand{\Der}{\operatorname{Der}}

\newcommand{\Div}{\operatorname{div}}

\DeclareMathOperator{\ad}{\mathrm{ad}}
\DeclareMathOperator{\Supp}{\mathrm{Supp}}
\DeclareMathOperator{\supp}{\mathit{supp}}
\DeclareMathOperator{\Hom}{\mathrm{Hom}}
\DeclareMathOperator{\ord}{\mathrm{ord}}
\DeclareMathOperator{\Quot}{\mathrm{Quot}}
\DeclareMathOperator{\Cl}{\mathrm{Cl}}
\DeclareMathOperator{\Spec}{\mathrm{Spec}}
\DeclareMathOperator{\Aut}{\mathrm{Aut}}

\DeclareMathOperator{\hgt}{ht}

\newtheorem{theorem}{Theorem}

\newtheorem{proposition}[theorem]{Proposition}
\newtheorem{lemma}[theorem]{Lemma}
\newtheorem{corollary}[theorem]{Corollary}
\newtheorem{conjecture}[theorem]{Conjecture}

\theoremstyle{definition}

\newtheorem{definition}[theorem]{Definition}
\newtheorem{example}[theorem]{Example}

\theoremstyle{remark}

\newtheorem{remark}[theorem]{Remark}

\numberwithin{equation}{section}

\makeatletter
\@addtoreset{theorem}{section}

\def\LT@makecaption#1#2#3{%
 \LT@mcol \LT@cols c{\hbox to\z@{\hss\parbox[t]\LTcapwidth{%
 \normalsize\@captionfont\@captionheadfont
 \sbox{\@tempboxa}{#2. #3}%
 \ifdim\wd\@tempboxa>\hsize
 #2. #3%
 \else
 \hbox to\hsize{\hfil\box\@tempboxa\hfil}%
 \fi
 \endgraf\vskip\baselineskip}%
 \hss}}}

\def\fnum@table{\tablename\hskip.33em\thetable}

\makeatother

\begin{document}

\renewcommand{\proofname}{Proof}
\renewcommand{\abstractname}{Abstract}
\renewcommand{\refname}{References}
\renewcommand{\figurename}{Figure}
\renewcommand{\tablename}{Table}

\title[Strongly solvable spherical subgroups]{Strongly solvable spherical subgroups\\and their combinatorial invariants}

\author{Roman Avdeev}

\thanks{Partially supported by the Russian Foundation for Basic Research,
grant no. 12-01-00704. A considerable part of this paper was written
during the author's stays at the University of Bielefeld (Bielefeld,
Germany, June and July, 2012) supported by the Alexander von
Humboldt Foundation and SFB~701 and at the University of Bochum
(Bochum, Germany, November and December, 2012) supported by the DFG
priority program 1388 (Representation theory). The work on this
paper was completed during the author's stay at the University of
Cologne (Cologne, Germany, April and May, 2013) supported by the DFG
priority program 1388 (Representation theory). The author expresses
his gratitude to H.~Abels, P.~Heinzner and S.~Cupit-Foutou, and
P.~Littelmann for the invitations.}

\address{%
{\bf Roman Avdeev} \newline\indent National Research University
``Higher School of Economics'', Moscow, Russia}

\email{suselr@yandex.ru}




\begin{abstract}
A subgroup $H$ of an algebraic group $G$ is said to be strongly
solvable if $H$ is contained in a Borel subgroup of~$G$. This paper
is devoted to establishing relationships between the following three
combinatorial classifications of strongly solvable spherical
subgroups in reductive complex algebraic groups: Luna's general
classification of arbitrary spherical subgroups restricted to the
strongly solvable case, Luna's 1993 classification of strongly
solvable wonderful subgroups, and the author's 2011 classification
of strongly solvable spherical subgroups. We give a detailed
presentation of all the three classifications and exhibit
interrelations between the corresponding combinatorial invariants,
which enables one to pass from one of these classifications to any
other.
\end{abstract}

\maketitle

\sloppy

\section{Introduction} \label{sect_introduction}

Let $G$ be a connected reductive complex algebraic group and let $B$
be a Borel subgroup of~$G$. A closed subgroup~$H \subset G$, as well
as the corresponding homogeneous space $G / H$, is said to be
\textit{spherical} if any one of the following equivalent conditions
is satisfied:
\begin{enumerate}[label=\textup{(S\arabic*)},ref=\textup{S\arabic*}]
\item \label{S1}
the group $B$ has an open orbit in $G / H$;

\item
for every homogeneous line bundle $\mathcal L$ over $G / H$, the
representation of $G$ on the space of regular sections of $\mathcal
L$ is multiplicity free;

\item \label{S3}
for every simple $G$-module~$V$ and every character $\chi$ of~$H$,
the subspace $V^{(H)}_\chi \subset V$ of $H$-semi-invariant vectors
of weight $\chi$ is at most one-dimensional.
\end{enumerate}

Apart from properties~(\ref{S1})--(\ref{S3}), which will play an
important role in this paper, there are many other equivalent
sphericity conditions. An extensive list of them, as well as proofs
of the equivalences, can be found in the monograph by Timashev;
see~\cite[\S\,25]{Tim}.

We recall that, for an arbitrary homogeneous space $G / H$, every
normal irreducible $G$-variety containing $G / H$ as an open orbit
is said to be an \textit{equivariant embedding} (or simply an
\textit{embedding}) of $G / H$. Equivariant embeddings of spherical
homogeneous spaces are said to be \textit{spherical varieties}.
Among arbitrary normal irreducible $G$-varieties, spherical
$G$-varieties are characterized by the existence of an open orbit
for the induced action of~$B$.

Spherical varieties form an extremely interesting class of
$G$-varieties. The most famous representatives of this class are
toric varieties, symmetric varieties, and flag varieties. Countless
papers are devoted to various aspects concerned with spherical
varieties or certain subclasses of them. Some of these aspects are
reflected in~\cite[Chapter~5]{Tim}. In this context we also mention
the recent survey paper~\cite{Per} devoted to the geometry of
spherical varieties.

One of the important problems in the theory of spherical varieties
is the problem of classifying them. The spherical $G$-varieties with
a given open $G$-orbit, that is, the embeddings of a given spherical
homogeneous space, were classified by Luna and Vust in the framework
of their general theory of embeddings of arbitrary homogeneous
spaces developed in~\cite{LV}. Later, in the particular case of
spherical varieties this theory was considerably simplified and
restated in a more transparent form by Knop~\cite{Kn91}. As a
result, the embeddings of a given spherical homogeneous space are
classified in terms of certain objects of convex geometry called
colored fans, which generalize usual fans used for classifying toric
varieties. Thus the classification of spherical varieties reduces to
that of spherical homogeneous spaces, which by definition is
equivalent to the classification (up to conjugacy) of spherical
subgroups in reductive algebraic groups.

In classifying spherical homogeneous spaces $G / H$ (or,
equivalently, spherical subgroups $H \subset G$ up to conjugacy) one
may restrict the consideration to the case where $G$ is semisimple;
see, for instance,~\cite[\S\,I.3.3, Corollary~2]{Vin01}
or~\cite[\S\,10.2]{Tim}. For semisimple~$G$, a complete
classification of \textit{affine} spherical homogeneous spaces $G /
H$ (that is, with reductive $H$) was obtained by Kr\"amer~\cite{Kr}
(in the case of simple~$G$), Mikityuk~\cite{Mi}, Brion~\cite{Br87}
(both treated independently the case of non-simple semisimple~$G$),
and Yakimova~\cite{Yak} (small corrections). We note that this
classification is essentially given by a list of spaces and does not
involve any combinatorial invariants.

At present, there is a complete classification in combinatorial
terms of spherical homogeneous spaces $G / H$, which is a result of
joint decade-long efforts of several researchers. The idea of this
classification was proposed by Luna in 2001~\cite{Lu01}, therefore
we shall refer to it as Luna's general classification. This
classification is carried out in two steps. The first one is to
reduce the classification to the case of so-called wonderful
subgroups. A spherical subgroup $H \subset G$ is said to be
\textit{wonderful} if $G / H$ admits a wonderful completion, which
is a smooth complete embedding with certain additional properties
(see Definitions~\ref{def_wonderful1} and~\ref{def_wonderful2}).
Wonderful completions of spherical homogeneous spaces $G / H$ are
also known as wonderful $G$-varieties. They are generalizations of
complete symmetric varieties considered in~\cite{ConP}. The second
step of the classification is to describe all the wonderful
$G$-varieties. In the paper~\cite{Lu01}, Luna himself performed the
first step in full generality, stated a conjecture for the second
step, and proved this conjecture in the case where $G$ is a product
of simple groups of type~$\mathsf A$. Luna's conjecture claimed that
wonderful $G$-varieties are classified by combinatorial objects
called \textit{spherical systems}. During the following several
years the conjecture was proved in certain other particular cases,
see~\cite{BP1}, \cite{Bra1}, and~\cite{BC}. The uniqueness part of
Luna's conjecture was proved by Losev~\cite{Lo1} in 2009 by a
general argument. The first general proof of the existence part of
this conjecture was obtained by Cupit-Foutou~\cite{Cu} via invariant
Hilbert schemes. Another proof of the existence part, which follows
the lines of the original constructive approach employed by Luna,
was recently suggested by Bravi and Pezzini in the series of
papers~\cite{BP2}, \cite{BP3}, and \cite{BP4}. We also mention the
paper~\cite{BL}, which is an introduction to wonderful varieties
with an emphasis put on the combinatorics of spherical systems.

The combinatorial invariants involved in Luna's general
classification are called the \textit{principal combinatorial
invariants} in this paper. They come from the Luna--Vust theory
mentioned above: the description of all embeddings of a fixed
spherical homogeneous space is given in terms of these invariants. A
detailed discussion of the principal combinatorial invariants,
including definitions and basic properties, can be found in
\S\,\ref{subsec_pci}.

A subgroup $H \subset G$ is said to be \textit{strongly solvable} if
it is contained in a Borel subgroup of~$G$. We note that every
connected solvable subgroup of $G$ is automatically strongly
solvable. An example of a solvable but not strongly solvable
subgroup is given by the normalizer of a maximal torus in $\SL_2$.

Apart from Luna's general classification of arbitrary spherical
subgroups, there are two additional combinatorial classifications in
the case of strongly solvable spherical subgroups. Both of them can
be developed independently of each other and Luna's general
classification. The first one was obtained by Luna in his
unpublished 1993 preprint~\cite{Lu93}. This classification applies
only to the case of wonderful strongly solvable subgroups. In what
follows we shall refer to this classification as Luna's 1993
classification. The second classification was recently obtained by
the author~\cite{Avd_solv} and deals with connected solvable
subgroups. In fact, the connectedness condition turns out to be
inessential in this classification, so that it extends almost
unchanged to the case of arbitrary strongly solvable subgroups.

Both Luna's general classification and Luna's 1993 classification
have a geometric origin: the invariants involved arise from the
geometry of the corresponding homogeneous spaces. As a result, both
classifications provide no simple method for relating the invariants
with an explicit description of the corresponding subgroups, which
leaves the following two problems unsolved:
\begin{enumerate}[label=\textup{(P\arabic*)},ref=\textup{P\arabic*}]
\item \label{P1}
compute the invariants of a given subgroup specified by a Levi
subgroup of it together with the Lie algebra of its unipotent
radical;

\item \label{P2}
determine (in the same sense) the subgroup corresponding to a given
set of invariants.
\end{enumerate}

\noindent We should mention here that a considerable progress in
solving problem (\ref{P2}) for wonderful subgroups has been achieved
by Bravi and Pezzini in~\cite[\S\S\,3,\,5]{BP2} and~\cite{BP3}.
Namely, their approach in fact enables one to construct a wonderful
subgroup $H$ starting from its spherical system, however the
procedure is very indirect and consists in several reduction steps
leading in the end to a list of ``primitive'' cases. From this
procedure, one can extract an explicit description of a Levi
subgroup of~$H$. As for determining the unipotent radical $H_u
\subset H$, they suggested a technique that helps to guess $H_u$ in
every concrete example. Unfortunately, so far there is no proof that
this technique will always work.

In contrast to both Luna's classifications, the author's 2011
classification is much more algebraic: the invariants involved
encode explicitly an embedding of the Lie algebra of $H$ in the Lie
algebra of $G$, so that $H$ can be easily recovered from the
corresponding invariants. For this reason, in what follows the
author's 2011 classification will be referred to as the explicit
classification.

The main goal of this paper is to give a detailed presentation of
all the three above-mentioned combinatorial classifications and
reveal interrelations between the corresponding combinatorial
invariants. Specifically, for every two classifications under
consideration we provide either explicit formulas or an effective
method for computing the invariants involved in the first one
starting from the invariants involved in the second one. Of course,
the latter is possible in the situation where both classifications
apply, that is, for strongly solvable spherical subgroups or for
strongly solvable wonderful subgroups.

This paper is intermediate between a survey and a research paper. On
the one hand, an extensive description of a wide range of known
results concerned with Luna's general classification and some other
aspects makes the paper resemble a survey. On the other hand, our
ultimate goal consists in solving concrete problems, therefore,
unlike a survey, this paper contains a number of new results.
Results that are definitely new are contained in
\S\,\ref{sect_explicit_classification} and concerned with
establishing relationships between the explicit classification and
the other two classifications under consideration. In particular,
these results solve problems (\ref{P1}) and (\ref{P2}) in the case
of strongly solvable spherical subgroups.

A special attention is paid in this paper to Luna's 1993
classification of wonderful strongly solvable subgroups. At the
moment, this classification can be found only in Luna's unpublished
preprint~\cite{Lu93}, which is extremely hard to access. Few
references to this preprint existing in the literature give no idea
on the employed approach and the classification itself. Thus the
classification now seems to be almost forgotten. Being sure that it
does not deserve such a fate, in this paper we make an attempt to
document this classification in full detail, providing complete
proofs for all statements. (We note that Luna's preprint is written
in a rather sketchy style.) Luna's 1993 classification is based on
the description of automorphism groups of smooth complete toric
varieties obtained by Demazure in~1970~\cite{Dem}. To present the
classification, we use a much more transparent version of this
description obtained by Cox in 1995~\cite{Cox} via his realization
of toric varieties as quotients of vector spaces by actions of
diagonalizable groups.

Along with the invariants involved in the three classifications in
question, in this paper we consider one more invariant of arbitrary
spherical homogeneous spaces, called the extended weight semigroup
(see its definition in~\S\,\ref{subsec_ews}). The term ``extended
weight semigroup'' was introduced in the recent paper~\cite{Avd_EWS}
though the semigroup itself appeared implicitly many times in
earlier papers of different authors. This semigroup is closely
related to the principal combinatorial invariants. Namely, it turns
out that, knowing the extended weight semigroup, one can compute all
but one principal combinatorial invariants, and in the strongly
solvable case this semigroup determines all the principal
combinatorial invariants. On the other hand, the extended weight
semigroup is recovered from the principal combinatorial invariants.
A systematic study of the interrelations between these invariants is
undertaken in~\S\,\ref{subsec_ews_pci_interrelations}.

The significance of the extended weight semigroup in this paper
becomes apparent in relating Luna's general classification to the
explicit classification. Thanks to the paper~\cite{AG}, extended
weight semigroups are computed for all spherical homogeneous spaces
with strongly solvable stabilizer in terms of the combinatorial
invariants involved in the explicit classification. In view of the
previous paragraph, this enables one to pass between the explicit
classification and Luna's general classification. In this situation,
the extended weight semigroup plays the role of an intermediate
invariant between the invariants involved in the two
classifications.

This paper is organized as follows.

In \S\,\ref{sect_invariants} we introduce the principal
combinatorial invariants and the extended weight semigroup of a
spherical homogeneous space and then study interrelations between
these invariants.

The main goal of \S\,\ref{sect_Luna_gen} is to present Luna's
general classification of spherical homogeneous spaces. To this end,
we introduce the important notion of a wonderful $G$-variety and
explain how the classification reduces to that of wonderful
$G$-varieties. Apart from precise description of combinatorial
objects involved in Luna's general classification, we also introduce
several related notions needed later in this paper. At last, we
state and prove criteria for spherical or wonderful subgroups to be
strongly solvable in terms of their combinatorial invariants.

In \S\,\ref{sect_Luna_1993} we present Luna's 1993 classification of
strongly solvable wonderful subgroups. We first show how this
classification reduces to that of so-called wonderful
$B^-$-varieties, where $B^-$ is the Borel subgroup of~$G$ opposite
to~$B$ with respect to a fixed maximal torus~$T \subset B$. Next, we
provide a detailed description of connected automorphism groups of
smooth complete toric varieties and then apply it to obtain a
classification of wonderful $B^-$-varieties, which also implies a
classification of strongly solvable wonderful subgroups. Finally, we
find out how the invariants involved in Luna's 1993 classification
are related to the invariants involved in Luna's general
classification.

The explicit classification of strongly solvable spherical subgroups
is presented in~\S\,\ref{sect_explicit_classification}. We begin
with an outline of main ideas employed in this classification and
then state the classification itself. After that, we establish
interrelations between the explicit classification and two Luna's
classifications.

In~\S\,\ref{sect_generalizations} we discuss possible
generalizations of Luna's 1993 classification and the explicit
classification to the case of arbitrary spherical subgroups.

In Appendix~\ref{sect_homogeneous_bundles}, we recall the
construction and main properties of homogeneous bundles, which play
an important role in \S\,\ref{sect_Luna_1993}.

Appendix~\ref{sect_lists} can be considered as an illustration of
relationships between the three classifications under consideration.
Here, we list on the combinatorial level all wonderful strongly
solvable subgroups in all semisimple groups of rank at most~$2$ and
also in simple groups of type~$\mathsf A_3$. For every such
subgroup, we indicate its invariants with respect to each of the
three classifications.

\textbf{Acknowledgements.} The author cordially thanks D.~Luna for
discussions and private notes, which among other things helped the
author to learn much about Luna's general classification of
spherical homogeneous spaces. Thanks are also due to
I.\,V.~Arzhantsev, P.~Bravi, S.~Cupit-Foutou, D.\,A.~Timashev,
E.\,B.~Vinberg, and V.\,S.~Zhgoon for helpful discussions on
particular topics. At last, the author is grateful to the referee
for numerous corrections and valuable suggestions.

\textbf{Some notation and conventions.}

In this paper the base field is the field $\mathbb C$ of complex
numbers. All topological terms relate to the Zariski topology. All
groups and their subgroups are assumed to be algebraic. The Lie
algebras of groups denoted by capital Latin letters are denoted by
the corresponding small Gothic letters. A $K$-variety is an
algebraic variety equipped with a regular action of a group~$K$.

$\ZZ^+$ is the set of non-negative integers;

$\QQ^+$ is the set of non-negative rational numbers;

$\CC^\times$ is the multiplicative group of the field~$\CC$;

$\langle \cdot\,, \cdot \rangle$ is the natural pairing between
$\Hom_\ZZ(L, \QQ)$ and~$L$, where $L$ is a lattice;

$e$ is the identity element of an arbitrary group;

$o$ is the base point of any homogeneous space $L / K$, $o = eK$

$|E|$ is the cardinality of a finite set~$E$;

$V^*$ is the space of linear functions on a vector space~$V$;

$\langle A \rangle$ is the linear span of a subset $A$ of a vector
space~$V$;

$K^0$ is the connected component of the identity of a group~$K$;

$K^\sharp$ is the common kernel of all characters of a group~$K$;

$(K,K)$ is the derived subgroup of a group~$K$;

$\mathfrak X(K)$ is the character group (in additive notation) of a
group~$K$;

$N_L(K)$ is the normalizer of a subgroup $K$ in a group~$L$;

$\CC[X]$ is the algebra of regular functions on an algebraic
variety~$X$;

$\CC(X)$ is the field of rational functions on an algebraic
variety~$X$;

$\Quot A$ is the field of quotients of a commutative algebra $A$
with no zero divisors;

$G$ is a connected reductive algebraic group;

$C \subset G$ is the connected center of~$G$;

$G^{ss} = (G, G)$;

$B \subset G$ is a fixed Borel subgroup of~$G$;

$T \subset B$ is a fixed maximal torus of~$G$;

$U \subset B$ is the unipotent radical of~$B$;

$B^- \subset G$ is the Borel subgroup opposite to $B$ with respect
to~$T$, that is, $B \cap B^- = T$;

$B^{ss} = B \cap G^{ss}$ is the Borel subgroup of $G^{ss}$ contained
in~$B$;

$T^{ss} = T \cap G^{ss}$ is the maximal torus of $G^{ss}$ contained
in~$T$;

$(\cdot\,, \cdot)$ is a fixed inner product on~$\mathfrak X(T)
\otimes_\ZZ \QQ$ invariant with respect to the Weyl group~$N_G(T) /
T$;

$\Delta \subset \mathfrak X(T)$ is the root system of $G$ with
respect to~$T$;

$\Delta^+ \subset \Delta$ is the subset of positive roots with
respect to~$B$;

$\Pi \subset \Delta^+$ is the set of simple roots;

$\mathfrak g_\alpha \subset \mathfrak g$ is the root subspace
corresponding to a root~$\alpha \in \Delta$;

$e_\alpha \in \mathfrak g_\alpha$ is a fixed nonzero element;

$\alpha^\vee \in \Hom_\ZZ(\ZZ\Delta, \ZZ)$ is the dual root
corresponding to a root $\alpha \in \Delta$;

$\mathfrak X_+(B) \subset \mathfrak X(B)$ is the set of dominant
weights of~$B$;

$\mathfrak X_+(B^{ss}) \subset \mathfrak X_+(B)$ is the set of
dominant weights of~$B^{ss}$;

$V(\lambda)$ is the simple $G$-module with highest weight $\lambda
\in \mathfrak X_+(B)$;

$v_\lambda$ is a highest-weight vector of $V(\lambda)$ with respect
to~$B$;

$w_\lambda$ is a lowest-weight vector of $V(\lambda)$ with respect
to~$B$ (that is, the line $\langle w_\lambda \rangle$ is
$B^-$-stable);

$\lambda^*$ is the highest weight of the simple $G$-module
$V(\lambda)^*$, so that $V(\lambda^*) \simeq V(\lambda)^*$;

$\varpi_\alpha \in \mathfrak X(T) \otimes_\ZZ \QQ$ is the
fundamental weight associated with a simple root~$\alpha$.

The lattices $\mathfrak X(B)$ and $\mathfrak X(T)$ are identified
via the restriction of characters from~$B$ to~$T$.

For every element $\gamma = \sum \limits_{\alpha \in \Pi}k_\alpha
\alpha$, where $k_\alpha \in \QQ^+$ for all $\alpha \in \Pi$, we set
$\Supp \gamma = \lbrace \alpha \mid \nobreak k_\alpha > \nobreak 0
\rbrace$. If moreover $\gamma \in \Delta^+$, then we set $\hgt
\gamma = \sum \limits_{\alpha \in \Pi} k_\alpha$.

For every weight $\lambda \in \mathfrak X_+(B^{ss})$, one has
$\lambda = \sum \limits_{\alpha \in \Pi} l_\alpha \varpi_\alpha$ for
some non-negative integers~$l_\alpha$. We set $\supp \lambda =
\lbrace \varpi_\alpha \mid l_\alpha > 0 \rbrace$.

For every subset $\Pi' \subset \Pi$, $P_{\Pi'} \supset B$ is the
parabolic subgroup of~$G$ whose Lie algebra is the direct sum of
$\mathfrak b$ and all root subspaces $\mathfrak g_{-\alpha}$ with
$\alpha \in \Delta^+$ and $\Supp \alpha \subset \Pi'$.

If $V$ is a vector space equipped with an action of a group~$K$,
then the notation $V^K$ stands for the subspace of $K$-invariant
vectors and, for every $\chi \in \mathfrak X(K)$, the notation
$V^{(K)}_\chi$ stands for the subspace of $K$-semi-invariant vectors
of weight~$\chi$.

The actions of $G$ on itself by left translation ($(g,x) \mapsto
gx$) and right translation ($(g,x) \mapsto xg^{-1}$) induce its
actions on $\mathbb C[G]$ and $\CC(G)$ by the formulas $(gf)(x)
=f(g^{-1}x)$ and $(gf)(x) =f(xg)$, respectively. For brevity, we
refer to these actions as the action \textit{on the left} and
\textit{on the right}, respectively. Unless otherwise specified, for
every subgroup $K \subset G$ the notation $\CC[G]^K$ (resp.
$\CC[G]^{(K)}_\chi$) stands for the $K$-invariants (resp.
$K$-semi-invariants of weight~$\chi$) with respect to the action
of~$K$ on the right.

Let $K$ be a group and let $K_1, K_2$ be subgroups of it. We write
$K = K_1 \rightthreetimes K_2$ if $K$ is a semidirect product of
$K_1, K_2$ with $K_2$ being a normal subgroup of~$K$.

For connected Dynkin diagrams, the numbering of nodes (that is,
simple roots) is the same as in the book~\cite{OV}.

For the notion of a geometric quotient used in this paper we refer
to~\cite[\S\,4.2]{PV}.

Let $Q$ be a finite-dimensional vector space over~$\QQ$.

A \textit{cone} in $Q$ is a subset $\mathcal C \subset Q$ that is
invariant under addition and multiplication by elements in $\QQ^+$,
that is, $q_1 x_1 + q_2 x_2 \in \mathcal C$ whenever $x_1, x_2 \in
\mathcal C$ and $q_1, q_2 \in \QQ^+$.

A cone $\mathcal C \subset Q$ is said to be \textit{finitely
generated} if there are finitely many elements $q_1, \ldots, q_s \in
Q$ with $\mathcal C = \QQ^+ q_1 + \ldots + \QQ^+ q_s$. All cones
considered in this paper are finitely generated.

A cone $\mathcal C \subset Q$ is said to be \textit{strictly convex}
if $\mathcal C \cap (- \mathcal C) = \lbrace 0 \rbrace$.

The \textit{dimension} of a cone is the dimension of its linear
span.

The \textit{dual cone} of a cone $\mathcal C \subset Q$ is the cone
$$
\mathcal C^\vee = \lbrace \xi \in Q^* \mid \xi(q) \ge 0 \text{\;for
all\;} q \in \mathcal C \rbrace.
$$
One always has $(\mathcal C^\vee)^\vee = \mathcal C$.

A \textit{face} of a cone $\mathcal C \subset Q$ is a subset of the
form $\mathcal C \cap \lbrace q \in Q \mid \xi(q) = 0 \rbrace$ for
some $\xi \in \mathcal C^\vee$.

A \textit{facet} of a cone $\mathcal C \subset Q$ is a face of
codimension~$1$.

The \textit{relative interior} $\mathcal C^\circ$ of a cone
$\mathcal C$ is $\mathcal C$ with all proper faces removed.

A \textit{fan} in $Q$ is a collection $\mathcal F$ of cones in $Q$
satisfying the two axioms below:

(1) if $\mathcal C \in \mathcal F$, then every face of $\mathcal C$
also belongs to~$\mathcal F$;

(2) if $\mathcal C_1, \mathcal C_2 \in \mathcal F$, then $\mathcal
C_1 \cap \mathcal C_2$ is a face of both $\mathcal C_1$
and~$\mathcal C_2$.

A fan $\mathcal F$ in $Q$ is said to be \textit{strictly convex} if
every cone in $\mathcal F$ is strictly convex.

A fan $\mathcal F$ in $Q$ is said to be \textit{complete} if $Q =
\bigcup \limits_{\mathcal C \in \mathcal F} \mathcal C$.

Let $L \subset Q$ be a fixed lattice of maximal rank, so that $Q = L
\otimes_\ZZ \QQ$. For every cone $\mathcal C \subset Q$, let
$\mathcal C^1$ denote the set of primitive elements $q$ of $L$ such
that $\QQ^+ q$ is a face of~$\mathcal C$. Similarly, for every fan
$\mathcal F$ in~$Q$, let $\mathcal F^1$ denote the set of primitive
elements $q$ of $L$ such that $\QQ^+ q$ is a cone in~$\mathcal F$.
Clearly, $\mathcal F^1 = \bigcup \limits_{\mathcal C \in \mathcal F}
\mathcal C^1$. In this paper we shall often find ourselves in the
situation where $Q = \Hom_\ZZ(M, \QQ)$ and $L = \Hom_\ZZ(M, \ZZ)$
for some lattice~$M$.

Under the assumptions of the previous paragraph, a cone $\mathcal C
\subset Q$ is said to be \textit{regular} if $\mathcal C$ is
strictly convex and the set $\mathcal C^1$ is a part of a basis
of~$L$. A fan $\mathcal F$ in $Q$ is said to be \textit{regular} if
every cone in $\mathcal F$ is regular. We note that every regular
fan is strictly convex.

\section{Combinatorial invariants of spherical homogeneous
spaces} \label{sect_invariants}

\subsection{The principal combinatorial invariants} \label{subsec_pci}

In this subsection we work with a fixed spherical homogeneous space
$G/H$.

Let $P = P_{G/H}$ be the stabilizer of the open $B$-orbit in $G/H$.
Evidently, $P$ is a parabolic subgroup of $G$ containing~$B$, so
that $P = P_{\Pi^p}$ for some subset $\Pi^p \subset \Pi$. The set
$\Pi^p = \Pi^p_{G/H}$ is the first invariant associated with $G/H$.

The second invariant is the \textit{weight lattice} $\Lambda =
\Lambda_{G/H}$. It is defined to be the lattice of weights of
$B$-semi-invariant rational functions on $G/H$:
$$
\Lambda = \lbrace \mu \in \mathfrak X(T) \mid \CC(G/H)^{(B)}_\mu \ne
\lbrace 0 \rbrace \rbrace.
$$
(Semi-invariants are taken with respect to the action of~$B$ on the
left.) The rank of $\Lambda$ is said to be the \textit{rank}
of~$G/H$. We also introduce the rational vector space $\mathcal Q =
\mathcal Q_{G/H} = \Hom_\ZZ(\Lambda, \QQ)$.

For every $\mu \in \Lambda$, one has $\dim \CC(G/H)^{(B)}_\mu = 1$
since there is an open $B$-orbit in $G/H$. Fix a nonzero function
$f_\mu$ in each of the subspaces $\CC(G/H)^{(B)}_\mu$.

Let $\mathcal V = \mathcal V_{G/H}$ denote the set of all discrete
$\QQ$-valued $G$-invariant valuations of the field $\CC(G/H)$
vanishing on $\CC^\times$. We define a map $\phi \colon \mathcal V
\to \mathcal Q$ by the formula
$$
\langle \phi(v), \mu \rangle = v(f_\mu),
$$
where $v \in \mathcal V$, $\mu \in \Lambda$. It is known that the
map $\phi$ is injective (see~\cite[\S\,7.4]{LV}
or~\cite[Corollary~1.8]{Kn91}) and its image is a finitely generated
cone containing the image in~$\mathcal Q$ of the antidominant Weyl
chamber (see~\cite[\S\,4.1, Corollary, i)]{BriP}
or~\cite[Corollary~5.3]{Kn91}). This cone is called the
\textit{valuation cone} of~$G/H$. Later on, we identify it
with~$\mathcal V$. We note that $\mathcal V$ spans $\mathcal Q$ as a
vector space. Brion proved that $\mathcal V$ is a fundamental
chamber of a finite subgroup in $\GL(\mathcal Q)$ generated by
reflections; see~\cite[\S\,3]{Bri90}.

Let $\Sigma = \Sigma_{G/H}$ be the set of primitive elements
$\sigma$ of $\Lambda$ with the following properties:
\begin{enumerate}[label=\textup{(\arabic*)},ref=\textup{\arabic*}]

\item
$\Ker \sigma \subset \mathcal Q$ contains a facet of~$\mathcal V$;

\item
$\langle \sigma, \mathcal V\rangle \le 0$.
\end{enumerate}

\noindent The elements in $\Sigma$ are called the \textit{spherical
roots} of~$G/H$. In particular, the set $\Sigma$ is linearly
independent. We note that the valuation cone $\mathcal V$ and the
set of spherical roots $\Sigma$ uniquely determine each other
whenever the weight lattice $\Lambda$ is known. The set $\Sigma$ is
the third invariant associated with $G/H$.

The fourth invariant is the set $\mathcal D = \mathcal D_{G/H}$ of
$B$-stable prime divisors in~$G/H$. The elements in $\mathcal D$ are
called the \textit{colors} of $G/H$. The set $\mathcal D$ is
considered together with a map $\varkappa = \varkappa_{G/H} \colon
\mathcal D \to \Hom_\ZZ(\Lambda, \ZZ) \subset \mathcal Q$ defined in
the following way. For a color $D \in \mathcal D$, one has $\langle
\varkappa(D), \mu \rangle = \ord_D(f_\mu)$ for all $\mu \in
\Lambda$, where $\ord_D(f_\mu)$ is the order of $f_\mu$ along~$D$.

For every $\alpha \in \Pi$, we consider the corresponding minimal
parabolic subgroup $P_{\lbrace \alpha \rbrace} \supset B$ and put
$\mathcal D(\alpha)$ to be the set of $P_{\lbrace \alpha
\rbrace}$-unstable colors.

\begin{proposition}[{\cite[\S\S\,2.7,~3.4]{Lu97}}]
\label{prop_alternative} The following assertions hold:
\begin{enumerate}[label=\textup{(\alph*)},ref=\textup{\alph*}]
\item
$\mathcal D = \bigcup \limits_{\alpha \in \Pi} \mathcal D(\alpha)$;

\item
every $\alpha \in \Pi$ belongs to exactly one of the following
types:
\begin{enumerate}
\item[type]$p$:
$\mathcal D(\alpha) = \varnothing$;

\item[type]$a$:
$\alpha \in \Sigma$, $\mathcal D(\alpha) = \lbrace D^+, D^-
\rbrace$, and $\varkappa(D^+) + \varkappa(D^-) = \left. \alpha^\vee
\right|_{\Lambda}$;

\item[type]$a'$:
$2\alpha \in \Sigma$, $\mathcal D(\alpha) = \lbrace D \rbrace$, and
$\varkappa(D) = \frac12\left. \alpha^\vee \right|_{\Lambda}$;

\item[type]$b$:
$\QQ \alpha \cap \Sigma = \varnothing$, $\mathcal D(\alpha) =
\lbrace D \rbrace$, and $\varkappa(D) = \left. \alpha^\vee
\right|_{\Lambda}$.
\end{enumerate}
\end{enumerate}
\end{proposition}

According to Proposition~\ref{prop_alternative}, the set $\Pi$
splits into the disjoint union
\begin{equation} \label{eqn_division_of_Pi}
\Pi = \Pi^p \cup \Pi^a \cup \Pi^{a'} \cup \Pi^b,
\end{equation}
where each superscript signifies the corresponding type of simple
roots. We note that the set $\Pi^p$ appearing
in~(\ref{eqn_division_of_Pi}) is nothing else than the set $\Pi^p$
defined in the beginning of this subsection. We denote by $\mathcal
D^a$ (resp. $\mathcal D^{a'}$, $\mathcal D^b$) the union of the sets
$\mathcal D(\alpha)$ where $\alpha$ runs over~$\Pi^a$
(resp.~$\Pi^{a'}, \Pi^b$). It turns out (see~\cite[\S\,2.7]{Lu97}
or~\cite[\S\,30.10]{Tim}) that there is the following disjoint
union:
\begin{equation} \label{eqn_disjoint_union_colors}
\mathcal D = \mathcal D^a \cup \mathcal D^{a'} \cup \mathcal D^b.
\end{equation}

For two spherical homogeneous spaces $G/H$ and $G/H'$, we write
$\mathcal D_{G/H} = \mathcal D_{G/H'}$ if there exists a bijection
$i \colon \mathcal D_{G/H} \to \mathcal D_{G/H'}$ such that
$\varkappa_{G/H} = \varkappa_{G/H'} \circ i$.

The following theorem is just a reformulation
of~\cite[Theorem~1]{Lo1}.

\begin{theorem} \label{thm_uniqueness}
The spherical homogeneous space $G / H$ is uniquely determined by
the quadruple $(\Lambda, \Pi^p, \Sigma, \mathcal D)$. In other
words, if $G / H'$ is another spherical homogeneous space such that
$\Lambda_{G/H} = \Lambda_{G/H'}$, $\Pi^p_{G/H} = \Pi^p_{G/H'}$,
$\Sigma_{G/H} = \Sigma_{G/H'}$, and $\mathcal D_{G/H} = \mathcal
D_{G/H'}$, then the subgroups $H$ and $H'$ are conjugate in~$G$.
\end{theorem}

Within the framework of Luna's general classification, in
\S\,\ref{subsec_classification} we shall explain which quadruples
$(\Lambda, \Pi^p, \Sigma, \mathcal D)$ arise from spherical
homogeneous spaces.

\subsection{The extended weight semigroup} \label{subsec_ews}

Basic references for this subsection
are~\cite[\S\S\,1.2,~1.3]{Avd_EWS} and~\cite[\S\S\,1.2,~1.3]{AG}.
(Both sources deal with the case of semisimple~$G$.)

First assume $H \subset G$ to be an arbitrary subgroup. We recall
that characters of $H$ are in one-to-one correspondence with
homogeneous line bundles over $G/H$ (see~\cite[Theorem~4]{Pop}).
Namely, for every character $\chi \in \mathfrak X(H)$ the
corresponding homogeneous line bundle, denoted by $G *_H \CC_\chi$,
is defined to be the quotient $(G \times \CC_\chi) / H$, where $H$
acts on $G$ by right translation and on $\CC_\chi \simeq \CC$ via
the character~$\chi$ (see Appendix~\ref{sect_homogeneous_bundles}
for more details). Let $\Gamma(G *_H \CC_\chi)$ denote the space of
regular sections of $G *_H \CC_\chi$. For every $\chi \in \mathfrak
X(H)$ there is a natural isomorphism
$$
\Gamma(G *_H \CC_{-\chi}) \simeq \CC[G]^{(H)}_\chi.
$$
For every $\lambda \in \mathfrak X_+(B)$, let $\CC[G]^{(B \times
H)}_{(\lambda, \chi)} \subset \CC[G]^{(H)}_\chi$ be the subspace of
$(B \times H)$-semi-invariant functions in $\CC[G]$ of
weight~$(\lambda, \chi)$, where the semi-invariants are taken with
respect to the action of $B$ on the left and the action of $H$ on
the right. Every nonzero function $f \in \CC[G]^{(B \times
H)}_{(\lambda, \chi)}$ is a highest-weight vector of a simple
$G$-submodule of $\CC[G]^{(H)}_\chi$ with highest weight~$\lambda$,
and vice versa. Let $\widehat \Lambda^+ = \widehat \Lambda^+_{G /
H}$ denote the set of all pairs $(\lambda, \chi)$, where $\lambda
\in \mathfrak X_+(B)$ and $\chi \in \mathfrak X(H)$, such that $\dim
\CC[G]^{(B \times H)}_{(\lambda, \chi)} \ne 0$, that is, the
$G$-module $\CC[G]^{(H)}_\chi \simeq \Gamma(G *_H \CC_{-\chi})$
contains the simple $G$-module $V(\lambda)$. The set $\widehat
\Lambda^+$ is a semigroup with zero (that is, a monoid) called the
\textit{extended weight semigroup} of $G/H$.

We recall the following well-known isomorphism of $(G \times
G)$-modules (see, for instance,~\cite[Theorem~2.15]{Tim}):
\begin{equation} \label{eqn_GtimesG}
\CC[G] \simeq \bigoplus \limits_{\lambda \in \mathfrak X_+(B)}
V(\lambda) \otimes V(\lambda^*),
\end{equation}
where in the left-hand side $G \times G$ acts on the left and on the
right, and in the right-hand side the first (resp. second) factor of
$G \times G$ acts on the first (resp. second) tensor factor. For a
fixed $\lambda \in \mathfrak X_+(B)$, the embedding $V(\lambda)
\otimes V(\lambda^*) \hookrightarrow \CC[G]$ is defined as follows.
For $u \in V(\lambda)$ and $v \in V(\lambda^*)$, $u \otimes v$ maps
to the function whose value at a point $g \in G$ is $\langle u, gv
\rangle$, where $\langle \cdot \,, \cdot \rangle$ is the natural
pairing between $V(\lambda)$ and $V(\lambda^*)$. Under
isomorphism~(\ref{eqn_GtimesG}), the subspace $\CC[G]^{(B \times
H)}_{(\lambda, \chi)} \subset \CC[G]$ corresponds to the subspace
$V(\lambda)^{(B)}_{\lambda} \otimes V(\lambda^*)^{(H)}_\chi =
\langle v_\lambda \rangle \otimes V(\lambda^*)^{(H)}_\chi \subset
V(\lambda) \otimes V(\lambda^*)$, where $v_\lambda$ is a
highest-weight vector in~$V(\lambda)$. Hence $\dim \CC[G]^{(B \times
H)}_{(\lambda, \chi)} = \dim V(\lambda^*)^{(H)}_\chi$ and $(\lambda,
\chi) \in \widehat \Lambda^+$ if and only if
$V(\lambda^*)^{(H)}_\chi \ne 0$.

Until the end of this subsection we assume that $H$ is spherical
in~$G$.

It is well known (see~\cite[Theorem~1]{VK}) that the sphericity of
$H$ is equivalent to the property that the representation of $G$ on
$\Gamma(G *_H \CC_{-\chi}) \simeq \CC[G]^{(H)}_\chi$ is multiplicity
free for every $\chi \in \mathfrak X(H)$. Hence $(\lambda, \chi) \in
\widehat \Lambda^+$ if and only if $\dim \CC[G]^{(B \times
H)}_{(\lambda, \chi)} = \dim V(\lambda^*)^{(H)}_\chi = 1$.

If $G$ is semisimple and simply connected, then the semigroup
$\widehat \Lambda^+$ is free and isomorphic to the semigroup of
effective $B$-stable divisors in $G/H$, which is freely generated by
the colors of~$G/H$ (see~\cite[Theorem~2]{AG}). Under this
isomorphism, a color $D$ of $G/H$ corresponds to an indecomposable
element $(\lambda_D, \chi_D)$ of $\widehat \Lambda^+$ such that $D$
is the divisor of zeros of a (unique up to proportionality) regular
section $s_D$ of $G *_H \CC_{-\chi_D}$ where $s_D$ is
$B$-semi-invariant of weight~$\lambda_D$.

Now assume that $G = C \times G^{ss}$ and $G^{ss}$ is simply
connected. Then $\mathfrak X(G) \simeq \mathfrak X(C)$, $\mathfrak
X(B) \simeq \mathfrak X(C) \oplus \nobreak \mathfrak X(B^{ss})$, and
$\mathfrak X_+(B) \simeq \mathfrak X(C) \oplus \mathfrak
X_+(B^{ss})$. For every subgroup $H \subset G$, let $H^{ss}$ denote
the projection of~$H$ to~$G^{ss}$. It is easily deduced from
condition (\ref{S1}) that $H$ is spherical in~$G$ if and only if
$H^{ss}$ is spherical in~$G^{ss}$. For every $\nu \in \mathfrak
X(C)$, let $\nu_H \in \mathfrak X(H)$ be the restriction of $\nu$
to~$H$.

Since every simple $G$-module is isomorphic to the tensor product of
a simple $G^{ss}$-module with a one-dimensional $C$-module, the
following simple result holds.

\begin{proposition} \label{prop_ews_in_general}
Modulo the natural inclusions $\mathfrak X_+(B^{ss}) \subset
\mathfrak X_+(B)$ and $\mathfrak X(H^{ss}) \subset \mathfrak X(H)$,
one has
\begin{equation} \label{eqn_ews_reductive_case}
\widehat \Lambda^+_{G/H} = \widehat \Lambda^+_{G^{ss}/H^{ss}} \oplus
\lbrace (\nu, -\nu_H) \mid \nu \in \mathfrak X(C) \rbrace.
\end{equation}
\end{proposition}

Consider the natural surjective morphism
$$
\phi \colon G / H \to G^{ss} / H^{ss} \simeq G / CH.
$$
It follows from~\cite[Theorem~4.2]{PV} that $\phi$ is a geometric
quotient for the action of~$C$ on $G / H$ on the left. Therefore,
given a color $D$ of $G / H$, $\phi(D)$ is a color of $G^{ss} /
H^{ss}$ and $D = \phi^{-1}(\phi(D))$, so that there is a natural
bijection between $\mathcal D_{G / H}$ and $\mathcal D_{G^{ss} /
H^{ss}}$. For every $D \in \mathcal D_{G / H}$, we put $\lambda_D =
\lambda_{\phi(D)}$ and define $\chi_D$ to be the image in $\mathfrak
X(H)$ of the character $\chi_{\phi(D)} \in \mathfrak X(H^{ss})$.
Clearly, the map $D \mapsto (\lambda_D, \chi_D)$ yields a bijection
between the colors of $G / H$ and the indecomposable elements of the
semigroup $\widehat \Lambda^+_{G^{ss} / H^{ss}}$ regarded as a
subsemigroup of $\widehat \Lambda^+_{G / H}$.

For every $D \in \mathcal D_{G/H}$, let $s_D$ be the image of the
section $s_{\phi(D)}$ defining $\phi(D)$ under the chain
$$
\Gamma(G^{ss} *_{H^{ss}} \CC_{-\chi_{\phi(D)}}) \simeq
\CC[G^{ss}]^{(H^{ss})}_{\chi_{\phi(D)}} \hookrightarrow
\CC[G]^{(H)}_{\chi_D} \simeq \Gamma(G *_H \CC_{-\chi_D}).
$$
Then $D$ is the divisor of zeros of the section~$s_D$.

We also note that for every $\nu \in \mathfrak X(C)$ the subspace
$\CC[G]^{(B \times H)}_{(\nu, - \nu_H)}$ is spanned by the character
$\nu^{-1}$, which nowhere vanishes.

\subsection{Relations between the principal combinatorial invariants
and the extended weight semigroup}
\label{subsec_ews_pci_interrelations}

In this subsection we begin with some auxiliary considerations.

For every group~$K$, let $K^\sharp$ denote the subgroup of $K$
defined to be the common kernel of all characters of~$K$. Then
$K/K^\sharp$ is a diagonalizable group and there is a natural
isomorphism $\mathfrak X(K) \simeq \mathfrak X(K/K^\sharp)$.

We recall that a subgroup $L_0$ of a group $L$ is said to be
\textit{observable} in $L$ if the homogeneous space $L/L_0$ is
quasi-affine. The following lemma is well known, but for convenience
of the reader we provide a proof of it.

\begin{lemma} \label{lemma_observable}
For arbitrary groups $K \subset L$, the subgroup $K^\sharp$ is
observable in~$L$.
\end{lemma}

\begin{proof}
By a theorem of Chevalley (see~\cite[\S\,11.2]{Hum}) there exist a
finite-dimensional $L$-module $V$ and a nonzero vector $v \in V$
such that $K$ is the stabilizer of the line $\langle v \rangle$. Let
$K_v \subset K$ be the stabilizer of~$v$. Evidently, $K_v \supset
K^\sharp$ and $K_v$ is observable in~$L$. Let $\chi_1, \ldots,
\chi_k$ be generators of the group $\mathfrak X(K)$. For $i = 1,
\ldots, k$ denote by $V_i$ the one-dimensional $K$-module on which
$K$ acts via the character~$\chi_i$. In each of the spaces $V_i$ fix
a nonzero vector~$v_i$. Then $K^\sharp$ is the stabilizer of the
vector $v_1 + \ldots + v_k$ in the $K$-module $V_1 \oplus \ldots
\oplus V_k$, which is also a $K_v$-module. Hence $K^\sharp$ is
observable in~$K_v$. As $K_v$ is observable in~$L$, it follows that
$K^\sharp$ is observable in~$L$ (see~\cite[\S\,5]{BHM}
or~\cite[Corollary~2.3]{Gro}).
\end{proof}

We now turn to the situation where $G$ is a connected reductive
group and $H \subset G$ is an arbitrary subgroup.

\begin{lemma} \label{lemma_H^sharp_com_stab_modules}
The group $H^\sharp$ is the common stabilizer in $G$ of all
$H^\sharp$-fixed vectors in all simple $G$-modules.
\end{lemma}

\begin{proof}
In view of Lemma~\ref{lemma_observable} the group $H^\sharp$ is
observable in~$G$, so by~\cite[Theorem~2.1]{Gro} there exist a
finite-dimensional $G$-module $V$ and a vector $v \in V$ such that
$H^\sharp$ is the stabilizer of~$v$ in~$G$. Hence $H^\sharp$ is the
common stabilizer in~$G$ of all $H^\sharp$-fixed points in~$V$.
Since $G$ is reductive, the $G$-module $V$ is completely reducible,
which implies that $H^\sharp$ is the common stabilizer in $G$ of all
$H^\sharp$-fixed points in all simple $G$-modules contained in~$V$,
whence the required result.
\end{proof}

\begin{corollary} \label{crl_H^sharp}
The group $H^\sharp$ is the common stabilizer in $G$ of all elements
in $\CC[G]^{H^\sharp}$ with respect to the action of $G$ on the
right.
\end{corollary}

\begin{proof}
This follows from Lemma~\ref{lemma_H^sharp_com_stab_modules} and
isomorphism~(\ref{eqn_GtimesG}).
\end{proof}

We now proceed to establishing relations between the principal
combinatorial invariants and the extended weight semigroup of a
spherical homogeneous space $G / H$.

Let $\widehat \Lambda = \widehat \Lambda_{G/H}$ denote the
sublattice in $\mathfrak X(B) \oplus \mathfrak X(H)$ generated by
$\widehat \Lambda^+$.

\begin{proposition} \label{prop_weight_lattice}
The map $\mu \mapsto (\mu, 0)$ induces an isomorphism
$$
\Lambda \simeq \widehat \Lambda \cap \mathfrak X(B) = \lbrace
(\lambda, \chi) \in \widehat \Lambda \mid \chi = 0 \rbrace.
$$
\end{proposition}

\begin{proof}
The inclusion $\Lambda \supset \widehat \Lambda \cap \mathfrak X(B)$
is obvious. To prove the converse inclusion, let $\mu \in \Lambda$
and regard the corresponding element $f_\mu \in \CC(G/H)$ as a $(B
\times H/H^\sharp)$-semi-invariant rational function on $G/H^\sharp$
of weight $(\mu, 0)$. Lemma~\ref{lemma_observable} yields that $G /
H^\sharp$ is quasi-affine, therefore $\CC(G/H^\sharp) = \Quot \CC[G
/ H^\sharp]$ and by~\cite[Theorem~3.3(a)]{PV} one has $f_\mu = F_1 /
F_2$ for some $B$-semi-invariant functions $F_1, F_2 \in \CC[G /
H^\sharp]$. As $H / H^\sharp$ is diagonalizable, we may also assume
$F_1, F_2$ to be $H / H^\sharp$-semi-invariant (of the same weight).
Let $(\lambda_1, \chi), (\lambda_2, \chi) \in \mathfrak X_+(B)
\oplus \mathfrak X(H/H^\sharp)$ be the ($B \times
H/H^\sharp$)-weights of $F_1, F_2$, respectively. Then, modulo the
isomorphism $\CC[G / H^\sharp] \simeq \CC[G]^{H^\sharp}$, one has
$F_1 \in \CC[G]^{(B \times H)}_{(\lambda_1, \chi)}$ and $F_2 \in
\CC[G]^{(B \times H)}_{(\lambda_2, \chi)}$, whence $(\lambda_1,
\chi), (\lambda_2, \chi) \in \widehat \Lambda^+$ and $(\mu, 0) =
(\lambda_1, \chi) - (\lambda_2, \chi)$.
\end{proof}

Until the end of this subsection we assume that $G = C \times
G^{ss}$ and $G^{ss}$ is simply connected.

By Proposition~\ref{prop_weight_lattice}, for every $\mu \in
\Lambda$ there is a unique expression of the form
\begin{equation} \label{eqn_(mu, 0)}
(\mu, 0) = \sum\limits_{D \in \mathcal D} c(D, \mu)(\lambda_D,
\chi_D) + (\nu, - \nu_H),
\end{equation}
where $\nu \in \mathfrak X(C)$. Taking into account the information
in the last three paragraphs of \S\,\ref{subsec_ews}, we immediately
obtain the following result.

\begin{proposition} \label{prop_value_of_color}
$\langle \varkappa(D), \mu \rangle = c(D, \mu)$ for every $D \in
\mathcal D$.
\end{proposition}

\begin{lemma} \label{lemma_fwls}
Suppose that $\alpha \in \Pi$ and $D \in \mathcal D$. Then $D \in
\mathcal D(\alpha)$ if and only if $\varpi_\alpha \in \supp
\lambda_D$.
\end{lemma}

\begin{proof}
Recall from~\S\,\ref{subsec_ews} that $D$ is the divisor of zeros of
a (unique up to proportionality) regular section $s_D \in \Gamma(G
*_H \CC_{-\chi_D})$ that is $B$-semi-invariant of weight
$\lambda_D$. Clearly, $D$ is $P_\alpha$-unstable if and only if its
preimage $\widetilde D \subset G$ under the map $G \to G / H$ is so.
Note that $\widetilde D$ is the divisor of zeros of $s_D$ considered
as an element of $\CC[G]$. It follows from \cite[Theorem~3.1]{PV}
that $\widetilde D$ is $P_\alpha$-unstable if and only if $s_D$ is
not $P_\alpha$-semi-invariant. The latter is equivalent to
$\varpi_\alpha \in \supp \lambda_D$.
\end{proof}

Lemma~\ref{lemma_fwls} and Proposition~\ref{prop_alternative} imply
the proposition below.

\begin{proposition} \label{prop_Sigma_via_supports}
Suppose that $\alpha \in \Pi$. Then:
\begin{enumerate}[label=\textup{(\alph*)},ref=\textup{\alph*}]
\item
$\alpha \in \Pi^p$ if and only if there is no free generator
$(\lambda, \chi)$ of\, $\widehat \Lambda^+_{G^{ss} / H^{ss}}$ such
that $\varpi_\alpha \in \supp \lambda$;

\item
$\alpha \in \Sigma$ if and only if there are exactly two different
free generators $(\lambda_1, \chi_1)$, $(\lambda_2, \chi_2)$ of
$\widehat \Lambda^+_{G^{ss} / H^{ss}}$ such that $\varpi_\alpha \in
\supp \lambda_1$ and $\varpi_\alpha \in \supp \lambda_2$.
\end{enumerate}
\end{proposition}

A combination of
Propositions~\ref{prop_weight_lattice},~\ref{prop_value_of_color},
and~\ref{prop_Sigma_via_supports} yields the following result.

\begin{proposition} \label{prop_EWS_determines_almost_everything}
The pair $(\mathfrak X(H), \widehat \Lambda^+)$ uniquely determines
$\Lambda$, $\Pi^p$, $\mathcal D$, and $\Sigma \cap \Pi$.
\end{proposition}

A slightly stronger form of
Proposition~\ref{prop_EWS_determines_almost_everything} will be
given by Theorem~\ref{thm_EWS_determines_almost_everything}.

Taking into account Theorem~\ref{thm_uniqueness} we get the
corollary below.

\begin{corollary}
The spherical homogeneous space $G / H$ is uniquely determined by
the triple $(\mathfrak X(H), \widehat \Lambda^+, \Sigma)$.
\end{corollary}

\begin{example}
Table~1 in~\cite{Avd_EWS} contains the free generators of the
semigroups $\widehat \Lambda^+$ for certain affine spherical
homogeneous spaces $G / H$ with non-simple semisimple~$G$, which by
Proposition~\ref{prop_EWS_determines_almost_everything} enables one
to compute the invariants $\Lambda$, $\Pi^p$, $\mathcal D$, and
$\Sigma \cap \Pi$ for these spaces. In particular, for spherical
homogeneous spaces $(\SL_n \times \SL_{n+1}) / (\SL_n \times
\CC^\times)$ with $n \ge 2$ and $(\Spin_n \times \Spin_{n+1}) /
\Spin_n$ with $n \ge 3$ (see items~1,~2 in~\cite[Table~1]{Avd_EWS})
it turns out that $\Pi \cap \Sigma = \Pi$, in which case $\Sigma =
\Pi$ and the entire quadruple $(\Lambda, \Pi^p, \Sigma, \mathcal D)$
is uniquely determined by the pair $(\mathfrak X(H), \widehat
\Lambda^+)$.
\end{example}

\begin{example}
A spherical homogeneous space $G / H$ is said to be \textit{model}
if it is quasi-affine and there is a $G$-module isomorphism
$$
\CC[G / H] \simeq \bigoplus \limits_{\lambda \in \mathfrak X_+(B)}
V(\lambda),
$$
that is, for every $\lambda \in \mathfrak X_+(B)$ the $G$-module
$\CC[G / H]$ contains a simple $G$-submodule isomorphic
to~$V(\lambda)$. An example of such a space is given by $G / U$. All
model spherical homogeneous spaces of (not necessarily simply
connected) semisimple groups were classified by Luna in~\cite{Lu07}.
If $G$ is simply connected, then for every such space $G / H$ one
has $\mathfrak X(H) = 0$ and $\widehat \Lambda^+ = \lbrace (\lambda,
0) \mid \lambda \in \mathfrak X_+(B) \rbrace$. Therefore, $G / H$ is
uniquely determined by the set~$\Sigma$, which may be any subset of
a certain finite set $\Sigma_G^{\mathrm{mod}} \subset \Delta^+$. We
note that $\Sigma_G^{\mathrm{mod}} \ne \varnothing$ whenever $G \ne
\SL_2$. This example shows that in general the set $\Sigma$ is not
determined by the pair $(\mathfrak X(H), \widehat \Lambda^+)$.
\end{example}

We now turn to the problem of expressing $\mathfrak X(H)$ and
$\widehat \Lambda^+$ in terms of the quadruple $(\Lambda, \Pi^p,
\Sigma, \mathcal D)$.

\begin{proposition} \label{prop_gen_char}
The group $\mathfrak X(H)$ is generated by the characters $\chi_D$,
$D \in \mathcal D$, and the characters $\nu_H$, $\nu \in \mathfrak
X(C)$.
\end{proposition}

\begin{proof}

Let $H_0 \subset H$ be the common kernel of all the
characters~$\chi_D$, $D \in \mathcal D$, and all the characters
$\nu_D$, $\nu \in \mathfrak X(C)$. Clearly, $H^\sharp \subset H_0$.
On the other hand, since $\widehat \Lambda^+_{G/H}$ is generated by
all the elements $(\lambda_D, \chi_D)$, $D \in \mathcal D$, and all
the elements $(\nu, - \nu_H)$, $\nu \in \mathfrak X(C)$, one has
$\CC[G]^{H^\sharp} = \CC[G]^{H_0}$. Corollary~\ref{crl_H^sharp}
yields $H_0 \subset H^\sharp$, whence $H_0 = H^\sharp$.
\end{proof}

Let $\ZZ^{\mathcal D}$ be the free Abelian group consisting of
integer linear combinations of elements in~$\mathcal D$.
Proposition~\ref{prop_gen_char} yields a surjective homomorphism
$$
\psi \colon \ZZ^{\mathcal D} \oplus \mathfrak X(C) \to \mathfrak
X(H)
$$
given by $D \mapsto\nobreak \chi_D$ for every $D \in \mathcal D$ and
$\nu \mapsto \nu_H$ for every $\nu \in \mathfrak X(C)$, hence
$\mathfrak X(H) \simeq \ZZ^{\mathcal D} / \Ker \psi$.

Every element $\mu \in \Lambda$ admits a unique expression of the
form $\mu = \mu^{ss} + \mu^C$, where $\mu^{ss} \in \mathfrak
X_+(B^{ss})$ and $\mu^C \in \mathfrak X(C)$. Combining
Propositions~\ref{prop_weight_lattice} and~\ref{prop_value_of_color}
together with formula~(\ref{eqn_(mu, 0)}), we obtain the following
result.

\begin{proposition} \label{prop_kernel}
The kernel of $\psi$ is generated by the elements $\sum \limits_{D
\in \mathcal D} \langle \varkappa(D), \mu \rangle D - \mu^C$, where
$\mu$ runs over a basis of~$\Lambda$.
\end{proposition}

For every $D \in \mathcal D$, we consider the expression $\lambda_D
= \sum \limits_{\alpha \in \Pi} n_{\alpha, D} \varpi_\alpha$.

\begin{proposition}[{\cite[\S\,2.2, Theorem~2.2]{Fos}, see also~\cite[Lemma~30.24]{Tim}}]
The numbers $n_{\alpha, D}$ are determined as follows:
$$
n_{\alpha, D} =
\begin{cases}
0 & \text{if }\, D \notin \mathcal D(\alpha); \\
1 & \text{if }\, D \in \mathcal D(\alpha) \text{ and } 2\alpha
\notin
\Sigma; \\
2 & \text{if }\, D \in \mathcal D(\alpha) \text{ and } 2\alpha \in
\Sigma.
\end{cases}
$$
\end{proposition}

\begin{corollary}[{\cite[\S\,2.1.2]{Cu},~\cite[Lemma~30.24]{Tim}}]
\label{crl_lambda_D} Depending on the type of~$D$, the weight
$\lambda_D$ is determined as follows:
$$
\lambda_D =
\begin{cases}
\sum \limits_{\alpha \in \Pi : D \in \mathcal D(\alpha)}
\varpi_\alpha & \text{if }\, D \in \mathcal D^a \text{ or } D
\in \mathcal D^b ;\\
\qquad \qquad 2\varpi_\alpha & \text{if }\, D \in \mathcal D^{a'}
\text{ and } D \in \mathcal D(\alpha).
\end{cases}
$$
\end{corollary}

Thus the principal combinatorial invariants of $G / H$ uniquely
determine the pair $(\mathfrak X(H), \widehat \Lambda^+)$, where
$\mathfrak X(H)$ is regarded as an abstract group.

Let us mention the following result implied by
Corollary~\ref{crl_lambda_D}.

\begin{corollary} \label{crl_2alpha}
Suppose that $\alpha \in \Pi$. Then $2\alpha \in \Sigma$ if and only
if there is a free generator of $\widehat \Lambda^+_{G^{ss} /
H^{ss}}$ of the form $(2\varpi_\alpha, \chi)$ for some $\chi \in
\mathfrak X(H)$.
\end{corollary}

Combining Proposition~\ref{prop_EWS_determines_almost_everything}
with Corollary~\ref{crl_2alpha}, we obtain the following result.

\begin{theorem} \label{thm_EWS_determines_almost_everything}
The pair $(\mathfrak X(H), \widehat \Lambda^+)$ uniquely determines
$\Lambda$, $\Pi^p$, $\mathcal D$, and $\Sigma \cap (\Pi \cup 2\Pi)$.
\end{theorem}

\section{Luna's general classification of spherical homogeneous spaces}
\label{sect_Luna_gen}

\subsection{Simple embeddings of spherical homogeneous spaces}
\label{subsec_simple_embeddings}

Let $G / H$ be a spherical homogeneous space. We retain all the
notation introduced in~\S\,\ref{sect_invariants}.

\begin{definition}
An embedding $X$ of $G / H$ is said to be \textit{simple} if $X$
contains exactly one closed $G$-orbit.
\end{definition}

Simple embeddings are classified by strictly convex colored cones.

\begin{definition}[{see \cite[\S\,3]{Kn91}}]
A \textit{colored cone} is a pair $(\mathcal C, \mathcal A)$ with
$\mathcal C \subset \mathcal Q$ and $\mathcal A \subset \mathcal D$
having the following properties:
\begin{enumerate}[label=\textup{(CC\arabic*)},ref=\textup{CC\arabic*}]
\item
$\mathcal C$ is a cone generated by $\varkappa(\mathcal A)$ and
finitely many elements of $\mathcal V$;

\item
$\mathcal C^\circ \cap \mathcal V \ne \varnothing$.
\end{enumerate}

A colored cone is said to be \textit{strictly convex} if the
following property holds:
\begin{enumerate}[label=\textup{(SCC)},ref=\textup{SCC}]
\item
$\mathcal C$ is strictly convex and $0 \notin \varkappa(\mathcal
A)$.
\end{enumerate}
\end{definition}

\begin{definition}[{see \cite[\S\,4]{Kn91}}] \label{def_colored_subspace}
A colored cone $(\mathcal C, \mathcal A)$ is said to be a
\textit{colored subspace} if $\mathcal C$ is a vector subspace of
$\mathcal Q$.
\end{definition}

Let $X$ be a simple embedding of $G / H$ and let $Y$ be its closed
$G$-orbit. We consider all $B$-stable prime divisors in~$X$
containing~$Y$. These can be divided into two parts. The first part,
denoted by~$\mathcal G(X)$, consists of divisors that are
$G$-stable. Divisors in the second part are closures of colors. Let
$\mathcal A(X)$ denote the set of colors arising in this way.

Let $\mathcal C(X)$ be the cone in $\mathcal Q$ generated by
$\varkappa(\mathcal A(X))$ and the images of $G$-invariant
valuations associated with elements in~$\mathcal G(X)$.

\begin{proposition}[{\cite[\S\,8.10, Proposition]{LV},
\cite[Theorem~3.1]{Kn91}}] \label{prop_simple_embeddings} The map $X
\mapsto (\mathcal C(X), \mathcal A(X))$ is a bijection between
simple embeddings of $G / H$ \textup(considered up to
$G$-equivariant isomorphism\textup) and strictly convex colored
cones in~$\mathcal Q$.
\end{proposition}

\subsection{Standard completions and wonderful $G$-varieties}
\label{subsec_standard_completions}

In this subsection we retain all the notation introduced
in~\S\,\ref{subsec_simple_embeddings}.

\begin{definition}
An embedding $X$ of $G / H$ is said to be \textit{toroidal} if no
color contains a $G$-orbit in its closure.
\end{definition}

In other words, $X$ is toroidal if every irreducible $B$-stable
closed subvariety containing a closed $G$-orbit is actually
$G$-stable.

\begin{definition}
A complete simple toroidal embedding of $G / H$ is said to be a
\textit{standard completion}\footnote{The term ``standard
embedding'' seems to be more common in this situation, however we
avoid using this term since it will appear later in this paper in
another context.} of $G / H$.
\end{definition}

\begin{proposition} \label{prop_standard_completion}
Suppose that $X$ is a simple embedding of $G / H$. Then $X$ is a
standard completion if and only if $\mathcal A(X) = \varnothing$ and
$\mathcal C(X) = \mathcal V$. In particular, a standard completion
is unique if exists.
\end{proposition}

\begin{proof}
It follows from the definition that $X$ is toroidal if and only if
$\mathcal A(X) = \varnothing$ and $\mathcal C(X) \subset \mathcal
V$. By~\cite[Theorem~4.2]{Kn91}, $X$ is complete if and only if
$\mathcal C(X) \supset \mathcal V$.
\end{proof}

\begin{definition}[{see~\cite{Lu97}}]
The subgroup $H$ is said to be \textit{sober} if the group $N_G(H) /
H$ is finite.
\end{definition}

\begin{corollary}
A standard completion of\, $G / H$ exists if and only if the group
$H$ is sober.
\end{corollary}

\begin{proof}
It follows from Propositions~\ref{prop_simple_embeddings}
and~\ref{prop_standard_completion} that a standard completion exists
if and only if the cone $\mathcal V$ is strictly convex.
By~\cite[\S\,5.3, Corollary]{BriP} the latter is equivalent to $H$
being sober.
\end{proof}

Until the end of this subsection we assume $H$ to be sober. Let $X$
be the standard completion of $G / H$. We put $X_B = X \backslash
\bigcup \limits_{D \in \mathcal D} \overline D$.
By~\cite[Theorem~2.1]{Kn91} the set $X_B$ is $B$-stable, affine, and
open. We call it the \textit{canonical $B$-chart} of~$X$. We note
that $X_B$ is nothing else than the union of $B$-orbits in $X$ whose
closure contains the closed $G$-orbit. Let $\mathcal C_B$ denote the
cone in $\Lambda \otimes_\ZZ \QQ$ generated by the weights of
$B$-semi-invariant regular functions on~$X_B$.

\begin{proposition} \label{prop_B-chart_spherical_roots}
Under the above assumptions, $\mathcal C_B^1 = \lbrace - \sigma \mid
\sigma \in \Sigma \rbrace$.
\end{proposition}

\begin{proof}
By~\cite[Theorem~2.5(a)]{Kn91}, $\mathcal C_B = \mathcal C(X)^\vee =
\mathcal V^\vee$, hence the required result is implied by the
definition of spherical roots.
\end{proof}

\begin{definition} \label{def_wonderful1}
A smooth standard completion of $G / H$ is said to be
\textit{wonderful}.
\end{definition}

\begin{definition} \label{def_wonderful2}
The subgroup $H$ is said to be \textit{wonderful} if $G / H$ admits
a wonderful completion.
\end{definition}

Wonderful subgroups $H \subset G$ are characterized by the following
property.

\begin{proposition} \label{prop_Sigma_generates_Lambda}
A spherical subgroup $H \subset G$ is wonderful if and only if
$\Lambda = \ZZ \Sigma$.
\end{proposition}

This fact seems to have been first observed by Knop in \cite{Kn96}.
For a proof see also~\cite[\S\,30.1]{Tim}.

\begin{definition}[{\cite{Lu96}}] \label{def_wond_G-var}
A smooth complete irreducible $G$-variety $X$ is said to be
\textit{wonderful} of rank~$r$ if the three conditions below are
satisfied:
\begin{enumerate}[label=\textup{(WG\arabic*)},ref=\textup{WG\arabic*}]
\item
$X$ contains an open $G$-orbit whose complement is a union of
$G$-stable prime divisors $D_1, \ldots, D_r$;

\item
$D_1, \ldots, D_r$ are smooth and have a non-empty transversal
intersection;

\item
for every two points $x, x' \in X$, $Gx = Gx'$ if and only if
$\lbrace i \mid x \in D_i \rbrace = \lbrace j \mid x' \in D_j
\rbrace$.
\end{enumerate}
\end{definition}

In~1996 Luna~\cite{Lu96} proved that every wonderful $G$-variety is
spherical, which implies the following result (see
also~\cite[Theorem~30.15]{Tim}).

\begin{theorem} \label{thm_criterion_wonderful}
A $G$-variety $X$ is wonderful if and only if $X$ is a wonderful
completion of a spherical homogeneous space $G / H$.
\end{theorem}

\begin{remark} \label{rem_rank}
The rank of a wonderful $G$-variety $X$ equals $|\Sigma_{G / H}|$,
where $G / H \subset X$ is the open $G$-orbit. Indeed,
by~\cite[Lemma~2.4]{Kn91} the number of $G$-stable prime divisors
in~$X$ equals $|\mathcal V^1_{G / H}|$. Since the cone $\mathcal
V_{G / H}$ is simplicial and generates $\mathcal Q_{G / H}$ as a
vector space, one has $|\mathcal V^1_{G / H}| = |\Sigma_{G / H}|$.
\end{remark}

\begin{remark}
It follows from Remark~\ref{rem_rank} and
Proposition~\ref{prop_Sigma_generates_Lambda} that the rank of a
wonderful $G$-variety~$X$ coincides with the rank of $X$ as a
spherical $G$-variety.
\end{remark}

It was noted in~\cite{Lu96} that every wonderful $G$-variety $X$ is
projective and the connected center of $G$ acts trivially on~$X$.
This follows from the proposition below, provided together with a
proof for convenience of the reader.

\begin{proposition}
Suppose that $X$ is a complete normal $G$-variety containing a
unique closed orbit. Then $X$ is projective and $C$ acts trivially
on~$X$.
\end{proposition}

\begin{proof}
By~\cite[Lemma~8]{Sum}, $X$ is covered by a finite number of
$G$-stable quasi-projective open subsets. Every open subset
containing the closed $G$-orbit then coincides with~$X$, since
otherwise its complement would contain a closed $G$-orbit. Hence $X$
is quasi-projective. The completeness of $X$ implies that $X$ is
projective. Next, by~\cite[Theorem~1]{Sum}, there are a
finite-dimensional vector space $V$, a homomorphism $G \to \PGL(V)$,
and a $G$-equivariant closed embedding $X \hookrightarrow \PP(V)$.
In what follows we assume $X \subset \PP(V)$. Replacing $G$ by a
suitable finite covering of its image in $\PGL(V)$, we may assume
that $V$ is a $G$-module.

Assume that $C$ acts non-trivially on~$X$. Then there is a
one-parameter subgroup $q \colon \CC^\times \hookrightarrow C$
acting non-trivially on~$X$. Let $C_q \subset C$ be the image
of~$q$. Consider the decomposition $V = V_1 \oplus \ldots \oplus
V_s$ into a direct sum of weight subspaces with respect to~$C_q$.
Let $n_1, \ldots, n_s$ be the integers such that $q(t)v = t^{n_i}v$
for every $t \in \CC^\times$ and $v \in V_i$, $i = 1, \ldots, s$.
Without loss of generality we may assume $n_1 < \ldots < n_s$.
Choose any vector $v \in V \backslash \lbrace 0 \rbrace$ such that
the line $\langle v \rangle \in \PP(V)$ belongs to $X$ and is
$C_q$-unstable. Consider the expression $v = v_1 + \ldots + v_s$,
where $v_i \in V_i$ for all $i = 1, \ldots, s$. Let $a$ (resp.~$b$)
be the minimal (resp. maximal) value of~$i$ such that $v_i \ne 0$.
Since $\langle v \rangle$ is $C_q$-unstable, one has $a \ne b$. Then
$\lim \limits_{t \to 0} q(t)\langle v \rangle = \langle v_a \rangle
\in X$ and $\lim \limits_{t \to \infty} q(t)\langle v \rangle =
\langle v_b \rangle \in X$. As $C_q$ is a central subgroup of~$G$,
each of the subspaces $V_1, \ldots, V_s$ is $G$-stable. It follows
that the sets $\overline{G \langle v_a \rangle}, \overline{G \langle
v_b \rangle} \subset X$, each of them containing a closed $G$-orbit,
do not intersect, a contradiction.
\end{proof}

\subsection{The spherical closure of a spherical subgroup}
\label{subsec_spherical_closure}

Let $H \subset G$ be a spherical subgroup. We consider the action of
$N_G(H)$ on $G / H$ given by $(n, gH) \mapsto gn^{-1}H$, where $n
\in N_G(H)$, $g \in G$. This action commutes with the action of $G$
on $G / H$ by left translation and therefore induces an action of
$N_G(H)$ on the set of colors~$\mathcal D$.

\begin{definition}
The kernel of the above action of $N_G(H)$ on~$\mathcal D$ is said
to be the \textit{spherical closure} of~$H$.
\end{definition}

We denote the spherical closure of $H$ by~$\overline H$. It follows
directly from the definition that $\overline H \supset N_G(H)^0$ and
$\overline H$ contains the center of~$G$.

\begin{remark}
By~\cite[\S\,5.2, Corollary]{BriP}, for every spherical subgroup $H
\subset G$ the group $N_G(H) / H$ is diagonalizable. Hence
$\overline H / H$ is also diagonalizable and $(\overline H)^\sharp
\subset H \subset \overline H$.
\end{remark}

\begin{definition}
The subgroup $H$ is said to be \textit{spherically closed} if
$\overline H = H$.
\end{definition}

Assume that $G = C \times G^{ss}$ and $G^{ss}$ is simply connected.
For every $D \in \mathcal D$, we consider the corresponding element
$(\lambda_D, \chi_D) \in \widehat \Lambda^+$ and let $v_D$ be a
nonzero element of the one-dimensional space
$V(\lambda_D^*)^{(H)}_{\chi_D}$. For every $\nu \in \mathfrak X(C)$,
we fix a nonzero vector $v_\nu$ in the one-dimensional
space~$V(\nu)$.

\begin{proposition}
\label{prop_spherical_closure} The spherical closure $\overline H$
of $H$ is the common stabilizer in $G$ of all the lines $\langle v_D
\rangle$, where $D$ runs over~$\mathcal D$.
\end{proposition}

In the proof of this proposition we shall need the following lemma.

\begin{lemma} \label{lemma_H^sharp}
For a spherical subgroup $H \subset G$, the group $H^\sharp$ is the
common stabilizer in~$G$ of all the vectors~$v_D$, where $D$ runs
over~$\mathcal D$, and all the vectors $v_\nu$, where $\nu$ runs
over a basis of~$\mathfrak X(C)$.
\end{lemma}

\begin{proof}
As the group $H/H^\sharp$ is diagonalizable,
Lemma~\ref{lemma_H^sharp_com_stab_modules} implies that $H^\sharp$
is the common stabilizer in~$G$ of all vectors in the spaces
$V(\lambda^*)^{(H)}_{\chi}$, where $(\lambda, \chi)$ runs over the
whole semigroup $\widehat \Lambda^+$. In view of
isomorphism~(\ref{eqn_GtimesG}), for every $(\lambda_1, \chi_1),
(\lambda_2, \chi_2) \in \widehat \Lambda^+$ and every $v_1 \in
V(\lambda_1^*)^{(H)}_{\chi_1}$, $v_2 \in
V(\lambda_2^*)^{(H)}_{\chi_2}$, $v_3 \in V(\lambda_1^* +
\lambda_2^*)^{(H)}_{\chi_1 + \chi_2}$, the common stabilizer of
$v_1$ and $v_2$ stabilizes~$v_3$. Using
equality~(\ref{eqn_ews_reductive_case}) concludes the proof.
\end{proof}

\begin{proof}
[Proof of Proposition~\textup{\ref{prop_spherical_closure}}]

Let $\widetilde H$ be the common stabilizer in $G$ of all the lines
$\langle v_D \rangle$, $D \in \mathcal D$. We note that $\widetilde
H$ automatically stabilizes all the lines $\langle v_\nu \rangle =
V(\nu)$, $\nu \in \mathfrak X(C)$. Clearly, $H^\sharp \subset H
\subset \widetilde H$. We first show that $\overline H = \widetilde
H \cap N_G(H)$. Indeed, an argument similar to the proof of
Lemma~\ref{lemma_fwls} shows that an element $n \in N_G(H)$ fixes a
color $D$ if and only if $n$ fixes the line
$$
\CC[G]^{(B \times H)}_{(\lambda_D, \chi_D)} \subset
\CC[G]^{(H)}_{\chi_D} \simeq \Gamma(G *_H \CC_{-\chi_D}).
$$
Taking into account isomorphism~(\ref{eqn_GtimesG}), we find that
$n$ fixes $D$ if and only if $n$ fixes the line
$V(\lambda^*_D)^{(H)}_{\chi_D} = \langle v_D \rangle$.

To complete the proof, it suffices to show that $\widetilde H
\subset N_G(H)$. By the definition of $\widetilde H$, we have the
natural diagonalizable action of $\widetilde H$ on the vector space
$\bigoplus \limits_{D \in \mathcal D} \langle v_D \rangle \oplus
\bigoplus \limits_\nu \langle v_\nu \rangle$ (in the latter sum
$\nu$ runs over a basis of~$\mathfrak X(C)$). By
Lemma~\ref{lemma_H^sharp}, the kernel of this action is~$H^\sharp$,
so $H^\sharp$ is a normal subgroup of $\widetilde H$ and the
quotient $\widetilde H / H^\sharp$ is commutative. The latter
implies that for every $h \in H$ and $\widetilde h \in \widetilde H$
one has $\widetilde h h \widetilde h^{-1} \in hH^\sharp \subset H$,
therefore $\widetilde H \subset N_G(H)$ and $\widetilde H =
\overline H$.
\end{proof}

Until the end of this subsection we assume that $G$ is an arbitrary
connected reductive group. Fix a finite covering group $\widetilde
G$ of $G$ that is a direct product of a torus with a simply
connected semisimple group. For every simple $\widetilde
G$-module~$V$, the corresponding projective space $\PP(V)$ has the
natural structure of a $G$-variety. Every $G$-variety arising in
this way is said to be a \textit{simple projective $G$-space}.

\begin{corollary}[{see~\cite[\S\,2.4.2, Lemma]{BL}}]
\label{crl_simple_projective_spaces}
For every spherical subgroup $H \subset G$, its spherical closure
$\overline H$ is the common stabilizer in $G$ of all $H$-fixed
points in all simple projective $G$-spaces.
\end{corollary}

\begin{proof}
Passing from $G$ to $\widetilde G$, one can easily reduce the
problem to the case where $G = C \times G^{ss}$ and $G^{ss}$ is
simply connected. In the latter case the assertion is a direct
consequence of Proposition~\ref{prop_spherical_closure}.
\end{proof}

The following well-known result is implied by
Corollary~\ref{crl_simple_projective_spaces}.

\begin{corollary}
For every spherical subgroup $H \subset G$, the group $\overline H$
is spherically closed.
\end{corollary}

\begin{proposition} \label{prop_sph_clos_in_parabolic}
Let $H \subset G$ be a spherical subgroup. Suppose that $H \subset
P$ for some parabolic subgroup $P \subset G$. Then $\overline H
\subset P$.
\end{proposition}

\begin{proof}
Without loss of generality we may assume that $G = C \times G^{ss}$,
$G^{ss}$ is simply connected, and $P = P_{\Pi'}$ for some subset
$\Pi' \subset\nobreak \Pi$. For every $\alpha \in \Pi \backslash
\Pi'$, the point $\langle v_{\varpi_\alpha} \rangle \in
\PP(V(\varpi_\alpha))$ is fixed by~$P$. Moreover, it is well known
that the stabilizer in~$G$ of this point is~$P_{\Pi \backslash
\lbrace \alpha \rbrace}$. Applying
Corollary~\ref{crl_simple_projective_spaces} we obtain
$$
\overline H \subset \bigcap \limits_{\alpha \in \Pi \backslash \Pi'}
P_{\Pi \backslash \lbrace \alpha \rbrace} = P,
$$
which completes the proof.
\end{proof}

\begin{corollary} \label{crl_sph_clos_str_solv}
Let $H \subset G$ be a spherical subgroup. If $H$ is strongly
solvable, then so is~$\overline H$.
\end{corollary}

The following theorem is a crucial point in reducing the
classification of spherical subgroups to that of wonderful
varieties.

\begin{theorem}[{\cite[\S\S\,7.6, 7.2]{Kn96}}]
\label{thm_sph_closed_wonderful} Let $H \subset G$ be a spherical
subgroup. If $H$ is spherically closed, then $H$ is wonderful. In
particular, $H$ is wonderful whenever $N_G(H) = H$.
\end{theorem}

\subsection{Classification of spherical homogeneous spaces and
wonderful $G$-varieties} \label{subsec_classification}

In this subsection we present Luna's general classification of
spherical homogeneous spaces and wonderful $G$-varieties. The idea
of this classification was proposed in the paper~\cite{Lu01} and
consists in performing two steps. At the first step, one classifies
all spherical subgroups $H \subset G$ with a given spherical
closure~$\overline H$. By Theorem~\ref{thm_sph_closed_wonderful} the
classification then reduces to that of wonderful $G$-varieties,
which is performed at the second step. As was mentioned in
Introduction, Luna himself accomplished the first step, proposed a
conjecture for the second step, and managed to prove this conjecture
in the case where $G$ is a product of simple groups of type~$\mathsf
A$. More details about the history of the proof of Luna's conjecture
can be found in \S\,\ref{sect_introduction}.

In the classification of wonderful $G$-varieties an important role
is played by wonderful $G$-varieties of small rank. As follows from
Definition~\ref{def_wond_G-var}, wonderful $G$-varieties of rank
zero are just complete homogeneous $G$-varieties, which are well
known to have the form $G / P$ for some parabolic subgroup $P
\subset G$. Rank-one wonderful varieties were classified by
Akhiezer~\cite{Akh83} and, by another method, Brion~\cite{Bri89}.
Wonderful varieties of rank two were classified by
Wasserman~\cite{Was}.

Elements of $\mathfrak X(T)$ appearing as spherical roots of
rank-one wonderful $G$-varieties are said to be \textit{spherical
roots of~$G$}. Let $\Sigma_G$ denote the set of all spherical roots
of~$G$. It is a finite set easily obtained from the classification
of rank-one wonderful $G$-varieties. Spherical roots are
non-negative linear combinations of simple roots of $G$ with
coefficients in $\frac12 \ZZ$. Spherical roots $\sigma$ that belong
to the root lattice of $G$ are listed in
Table~\ref{table_spherical_roots}. An element $\mu \in \mathfrak
X(T) \backslash \ZZ \Delta$ is a spherical root of $G$ if and only
if $\sigma = 2\mu$ appears in Table~\ref{table_spherical_roots} and
its number is marked by an asterisk. (In
Table~\ref{table_spherical_roots}, the notation $\alpha_i$ stands
for the $i$th simple root of the set $\Supp \sigma$ whenever the
Dynkin diagram of $\Supp \sigma$ is connected. If $\Supp \sigma$ is
of type $\mathsf A_1 \times \mathsf A_1$, then $\alpha, \beta$ are
the two distinct roots in $\Supp \sigma$.)

A pair $(\Pi^p, \sigma)$ with $\Pi^p \subset \Pi$ and $\sigma \in
\Sigma_G$ is said to be \textit{compatible} if there exists a
rank-one wonderful variety $X$ such that $\Pi^p_X = \Pi^p$ and
$\Sigma_X = \lbrace \sigma \rbrace$. Based on the classification of
rank-one wonderful $G$-varieties, the compatibility condition can be
reformulated in purely combinatorial terms (see, for
instance,~\cite[\S\,1.1.6]{BL}). Namely, the pair $(\Pi^p, \sigma)$
is compatible if and only if
$$
\Pi^{pp}(\sigma) \subset \Pi^p \subset \Pi^p(\sigma),
$$
where $\Pi^p(\sigma) = \lbrace \alpha \in \Pi \mid \langle
\alpha^\vee, \sigma \rangle = 0 \rbrace$ and the set
$\Pi^{pp}(\sigma) \subset \Pi$ is determined as follows:
$$
\Pi^{pp}(\sigma) =
\begin{cases}
\Supp \sigma \cap \Pi^p(\sigma) \backslash \lbrace \alpha_r \rbrace
& \text{if } \sigma = \alpha_1 + \alpha_2 + \ldots + \alpha_r
\text{ with support of type } \mathsf B_r; \\
\Supp \sigma \cap \Pi^p(\sigma) \backslash \lbrace \alpha_1 \rbrace
& \text{if } \sigma \text{ has support of type } \mathsf C_r; \\
\Supp \sigma \cap \Pi^p(\sigma) & \text{otherwise}.
\end{cases}
$$
For the reader's convenience, in the column ``$\Pi^{pp}(\sigma)$''
of Table~\ref{table_spherical_roots} we listed all roots in the set
$\Pi^{pp}(\sigma)$ for every spherical root~$\sigma \in \ZZ \Delta$.
If $\mu \in \Sigma_G \backslash \ZZ \Delta$, then $\Pi^{pp}(\mu) =
\Pi^{pp}(2\mu)$.

\begin{table}[h]

\caption{Spherical roots} \label{table_spherical_roots}

\begin{tabular}{|c|c|c|c|c|}
\hline

No. & \begin{tabular}{c} Type of \\ $\Supp \sigma$\end{tabular} &
$\sigma$ & $\Pi^{pp}(\sigma)$ & Note\\

\hline

$1$ & $\mathsf A_1$ & $\alpha_1$ & $\varnothing$ & \\

\hline

$2$ & $\mathsf A_1$ & $2\alpha_1$ & $\varnothing$ & \\

\hline

$3\lefteqn{^*}$ & $\mathsf A_1 \times \mathsf A_1$ & $\alpha +
\beta$ & $\varnothing$ & \\

\hline

$4$ & $\mathsf A_r$ & $\alpha_1 + \alpha_2 + \ldots +
\alpha_r$ & $\alpha_2, \alpha_3, \ldots, \alpha_{r-1}$ & $r \ge 2$\\

\hline

$5\lefteqn{^*}$ & $\mathsf A_3$ & $\alpha_1 + 2\alpha_2 + \alpha_3$
& $\alpha_1, \alpha_3$ & \\

\hline

$6$ & $\mathsf B_r$ & $\alpha_1 + \alpha_2 + \ldots +
\alpha_r$ & $\alpha_2, \alpha_3, \ldots, \alpha_{r-1}$ & $r \ge 2$\\

\hline

$7$ & $\mathsf B_r$ & $2\alpha_1 + 2\alpha_2 + \ldots +
2\alpha_r$ & $\alpha_2, \alpha_3, \ldots, \alpha_r$ & $r \ge 2$\\

\hline

$8\lefteqn{^*}$ & $\mathsf B_3$ & $\alpha_1 + 2\alpha_2 + 3\alpha_3$
& $\alpha_1, \alpha_2$ & \\

\hline

$9$ & $\mathsf C_r$ & $\alpha_1 + 2\alpha_2 + 2\alpha_3 + \ldots +
2\alpha_{r-1} + \alpha_r$ & $\alpha_3, \alpha_4, \ldots,
\alpha_r$ & $r \ge 3$\\

\hline

$10\lefteqn{^*}$ & $\mathsf D_r$ & $2\alpha_1 + 2\alpha_2 + \ldots +
2\alpha_{r-2} + \alpha_{r-1} + \alpha_r$ & $\alpha_2, \alpha_3,
\ldots, \alpha_r$ & $r \ge 4$ \\

\hline

$11$ & $\mathsf F_4$ & $2\alpha_1 + 3\alpha_2 + 2\alpha_3 +
\alpha_4$ &
$\alpha_2, \alpha_3, \alpha_4$ & \\

\hline

$12$ & $\mathsf G_2$ & $\alpha_1 + \alpha_2$ & $\varnothing$ & \\

\hline

$13$ & $\mathsf G_2$ & $2\alpha_1 + \alpha_2$ & $\alpha_2$ & \\

\hline

$14$ & $\mathsf G_2$ & $4\alpha_1 + 2\alpha_2$ & $\alpha_2$ & \\

\hline

\end{tabular}

\end{table}

Let $H \subset G$ be a spherical subgroup.

\begin{proposition} \label{prop_colors_only_a}
The following assertions hold:
\begin{enumerate}[label=\textup{(\alph*)},ref=\textup{\alph*}]
\item \label{prop_colors_only_a_a}
the quadruple $(\Lambda_{G/H}, \Pi^p_{G/H}, \Sigma_{G/H}, \mathcal
D_{G/H})$ amounts to the quadruple $(\Lambda_{G/H}, \Pi^p_{G/H},
\Sigma_{G/H}, \mathcal D^a_{G/H})$;

\item \label{prop_colors_only_a_b}
if $H$ is wonderful, then the quadruple $(\Lambda_{G/H},
\Pi^p_{G/H}, \Sigma_{G/H}, \mathcal D_{G/H})$ amounts to the triple
$(\Pi^p_{G/H}, \Sigma_{G/H}, \mathcal D^a_{G/H})$.
\end{enumerate}
\end{proposition}

\begin{proof}
(\ref{prop_colors_only_a_a}) In view of the disjoint
union~(\ref{eqn_disjoint_union_colors}) we need to show that the
sets $\mathcal D^{a'}_{G/H}$ and $\mathcal D^{b}_{G/H}$ are uniquely
determined by the other combinatorial invariants of $G / H$.
Clearly, the sets $\Pi^a_{G / H}$ and $\Pi^{a'}_{G / H}$ are
determined by $\Sigma_{G / H}$ and $\Pi^{b}_{G / H} = \Pi \backslash
(\Pi^p_{G / H} \cup \Pi^a_{G / H} \cup \Pi^{a'}_{G / H})$,
see~(\ref{eqn_division_of_Pi}). Proposition~\ref{prop_alternative}
yields surjective maps $\Pi^{a'}_{G / H} \to \mathcal D^{a'}_{G /
H}$ and $\Pi^b_{G / H} \to \mathcal D^b_{G / H}$ sending $\alpha$ to
the unique color $D_\alpha$ in $G / H$ moved by~$P_{\lbrace \alpha
\rbrace}$. We note that in both cases the element
$\varkappa(D_\alpha)$ is uniquely determined by~$\alpha$. One easily
checks that the map $\Pi^{a'}_{G / H} \to \mathcal D^{a'}_{G / H}$
is in fact bijective. It turns out that two different roots $\alpha,
\beta \in \Pi^b_{G / H}$ are taken to the same color if and only if
$\alpha \perp \beta$ and $\alpha + \beta \in \Sigma_{G/H} \cup
2\Sigma_{G/H}$; see~\cite[\S\,2.7]{Lu97}, or~\cite[\S\,2.3]{Lu01},
or~\cite[\S\,30.10]{Tim}.

Part~(\ref{prop_colors_only_a_b}) follows directly from
part~(\ref{prop_colors_only_a_a}) and
Proposition~\ref{prop_Sigma_generates_Lambda}.
\end{proof}

The set $\mathscr H_{G / H} = (\Lambda_{G/H}, \Pi^p_{G/H},
\Sigma_{G/H}, \mathcal D^a_{G/H})$ is said to be the
\textit{homogeneous spherical datum} of $G/H$. If $H$ is wonderful,
then the set $\mathscr S_{G / H} = (\Pi^p_{G/H}, \Sigma_{G/H},
\mathcal D^a_{G/H})$ is said to be the \textit{spherical system} of
$G/H$. These two combinatorial objects satisfy certain axioms, which
are listed in the definition below. These axioms trace back to
Proposition~\ref{prop_alternative} and Wasserman's classification of
rank-two wonderful varieties~\cite{Was}. Knop \cite{Kn14} has
recently deduced these axioms using only the classification of
rank-one wonderful varieties.

\begin{definition}[{\cite[\S\,2]{Lu01}}] \label{def_HSD&SS}
Suppose that $\Lambda$ is a sublattice in~$\mathfrak X(T)$, $\Pi^p$
is a subset of~$\Pi$, $\Sigma \subset \Sigma_G \cap \Lambda$ is a
linearly independent set consisting of indivisible elements
in~$\Lambda$, and $\mathcal D^a$ is a finite set equipped with a map
$\varkappa \colon \mathcal D^a \to \Hom_\ZZ(\Lambda, \ZZ)$. For
every $\alpha \in \Pi \cap \Sigma$, put $\mathcal D(\alpha) =
\lbrace D \in \mathcal D^a \mid \langle \varkappa(D), \alpha \rangle
= 1 \rbrace$.

The quadruple $(\Lambda, \Pi^p, \Sigma, \mathcal D^a)$ is said to be
a \textit{homogeneous spherical datum} if it satisfies the following
axioms:
\begin{enumerate}
\renewcommand{\labelenumi}{(A\arabic{enumi})}
\renewcommand{\theenumi}{A\arabic{enumi}}

\item \label{A1}
$\langle \varkappa(D), \sigma \rangle \le 1$ for all $D \in \mathcal
D^a$ and $\sigma \in \Sigma$, and the equality is attained if and
only if $\sigma = \alpha \in \Pi \cap \Sigma$ and $D \in \mathcal
D(\alpha)$;

\item \label{A2}
for every $\alpha \in \Pi \cap \Sigma$, the set $\mathcal D(\alpha)$
contains exactly two elements $D_\alpha^+$ and $D_\alpha^-$ such
that $\langle \varkappa(D_\alpha^+), \lambda \rangle + \langle
\varkappa(D_\alpha^-), \lambda \rangle = \langle \alpha^\vee,
\lambda \rangle$ for all $\lambda \in \Lambda$;

\item \label{A3}
the set $\mathcal D^a$ is the union of the sets $\mathcal D(\alpha)$
over all $\alpha \in \Pi \cap \Sigma$;

\setcounter{enumi}{0}
\renewcommand{\labelenumi}{($\Sigma$\arabic{enumi})}
\renewcommand{\theenumi}{$\Sigma$\arabic{enumi}}

\item \label{Sigma1}
if $\alpha \in \Pi \cap \frac12 \Sigma$, then $\langle \alpha^\vee,
\Lambda \rangle \subset 2\ZZ$ and $\langle \alpha^\vee, \sigma
\rangle \le 0$ for all $\sigma \in \Sigma \backslash \lbrace 2\alpha
\rbrace$;

\item \label{Sigma2}
if $\alpha, \beta \in \Pi$, $\alpha \perp \beta$, and $\alpha +
\beta \in \Sigma \cup 2\Sigma$, then $\langle \alpha^\vee, \lambda
\rangle = \langle \beta^\vee, \lambda \rangle$ for all $\lambda \in
\Lambda$;

\renewcommand{\labelenumi}{(S)}
\renewcommand{\theenumi}{S}

\item \label{label_S}
$\langle \alpha^\vee, \lambda \rangle = 0$ for all $\alpha \in
\Pi^p$ and $\lambda \in \Lambda$, and for every $\sigma \in \Sigma$
the pair $(\Pi^p, \sigma)$ is compatible.
\end{enumerate}

The triple $(\Pi^p, \Sigma, \mathcal D^a)$ is said to be a
\textit{spherical system} if it satisfies the above axioms with
$\Lambda = \ZZ\Sigma$.
\end{definition}

\begin{remark}
For every homogeneous spherical datum $(\Lambda, \Pi^p, \Sigma,
\mathcal D^a)$, the triple $(\Pi^p, \Sigma, \mathcal D^a)$ with
$\varkappa$ restricted to $\ZZ \Sigma$ is a spherical system.
\end{remark}

Let $X$ be a wonderful $G$-variety and let $G / H$ be the open
$G$-orbit in~$X$. By definition, we put $\Pi^p_X = \Pi^p_{G / H}$,
$\Sigma_X = \Sigma_{G / H}$, and $\mathcal D^a_X = \lbrace \overline
D \mid D \in \mathcal D^a_{G / H} \rbrace$. The triple $\mathscr S_X
= (\Pi^p_X, \Sigma_X, \mathcal D^a_X)$ is said to be the
\textit{spherical system} of~$X$. We note that the whole set of
colors of~$X$, which consists of the closures in $X$ of colors of~$G
/ H$, is usually defined as the set of $B$-stable prime divisors
of~$X$ that are not $G$-stable.

Luna's general classification is given by the following theorem.

\begin{theorem} \label{thm_classification}
The following assertions hold:
\begin{enumerate}[label=\textup{(\alph*)},ref=\textup{\alph*}]
\item
\textup(Luna's conjecture\textup) the map $X \mapsto (\Pi^p_X,
\Sigma_X, \mathcal D^a_X)$ is a bijection between wonderful
$G$-varieties \textup(considered up to $G$-equivariant
isomorphism\textup) and spherical systems for~$G$;

\item
the map $G / H \mapsto (\Lambda_{G / H}, \Pi^p_{G / H}, \Sigma_{G /
H}, \mathcal D^a_{G / H})$ is a bijection between spherical
homogeneous spaces of~$G$ \textup(considered up to $G$-equivariant
isomorphism\textup) and homogeneous spherical data for~$G$;
\end{enumerate}
\end{theorem}

\begin{example} \label{example_hom_wond}
Homogeneous wonderful $G$-varieties (that is, rank zero wonderful
$G$-varieties) are characterized by the condition $\Sigma =
\varnothing$, which immediately implies $\mathcal D^a =
\varnothing$. It is well known that, for every subset $\Pi' \subset
\Pi$, the stabilizer of the open $B$-orbit in $X = G / P_{\Pi'}$
coincides with $P_{\Pi'}$, whence the spherical system of $X$ is
$(\Pi', \varnothing, \varnothing)$. In particular, the spherical
system of $G / B$ is $(\varnothing, \varnothing, \varnothing)$.
\end{example}

Suppose we are given a homogeneous spherical datum $\mathscr H =
(\Lambda, \Pi^p, \Sigma, \mathcal D^a)$ or a spherical system
$\mathscr S = (\Pi^p, \Sigma, \mathcal D^a)$. Set $\Pi^a = \Pi \cap
\Sigma$, $\Pi^{a'} = \Pi \cap \frac12 \Sigma$, and $\Pi^b = \Pi
\backslash (\Pi^p \cup \Pi^a \cup \Pi^{a'})$. Axiom~(\ref{Sigma2})
implies that for every $\alpha \in \Pi^b$ there is at most one root
$\beta \in \Pi^b$ such that $\alpha \perp \beta$ and $\alpha + \beta
\in \Sigma \cup 2\Sigma$. Therefore we can introduce an equivalence
relation on $\Pi^b$ as follows. For $\alpha, \beta \in \Pi^b$ we
write $\alpha \sim \beta$ if and only if $\alpha = \beta$ or $\alpha
\perp \beta$ and $\alpha + \beta \in \Sigma \cup 2\Sigma$. We set
$\mathcal D^{a'} = \Pi^{a'}$, $\mathcal D^b = \Pi^b / \sim$ and let
$\mathcal D$ denote the disjoint union $\mathcal D^a \cup \mathcal
D^{a'} \cup \mathcal D^b$. Recall from Definition~\ref{def_HSD&SS}
that every $\alpha \in \Pi^a$ is associated with a subset $\mathcal
D(\alpha) \subset \mathcal D^a$. We extend the definition of
$\mathcal D(\alpha) \subset \mathcal D$ from $\Pi^a$ to $\Pi$ in the
following way:
$$
\mathcal D(\alpha) = \begin{cases}%
\varnothing & \text{ if } \alpha \in \Pi^p;\\
\alpha \text{ (as an element of } \mathcal D^{a'}) & \text{ if } \alpha \in \Pi^{a'};\\
\text{the equivalence class of } \alpha \text{ (as an element of }
\mathcal D^b) & \text{ if } \alpha \in \Pi^b.
\end{cases}
$$
We extend the map $\varkappa \colon \mathcal D^a \to
\Hom_\ZZ(\Lambda, \ZZ)$ from $\mathcal D^a$ to $\mathcal D$ as
follows: every $\alpha \in \Pi^{a'}$ maps to $\left.\frac12
\alpha^\vee \right|_\Lambda$ and for every $\alpha \in \Pi^b$ its
equivalence class maps to $\left. \alpha^\vee \right|_\Lambda$. The
set $\mathcal D$ equipped with the map $\varkappa \colon \mathcal D
\to \Hom_\ZZ(\Lambda, \ZZ)$ is said to be the \textit{set of colors
associated with $\mathscr H$ or~$\mathscr S$}. Let $G / H$ be a
spherical homogeneous space such that $\mathscr H_{G / H} = \mathscr
H$ or $\mathscr S_{G / H} = \mathscr S$. As follows from the proof
of Proposition~\ref{prop_colors_only_a}, there is a natural
bijection $i \colon \mathcal D_{G / H} \to \mathcal D$ such that
$\varkappa_{G / H} = \varkappa \circ i$, $\mathcal D^{a'} =
i(\mathcal D^{a'}_{G / H})$, $\mathcal D^b = i(\mathcal D^b_{G /
H})$, and $\mathcal D(\alpha) = i(\mathcal D_{G / H}(\alpha))$ for
every $\alpha \in \Pi \backslash \Pi^p$.

Let $H \subset G$ be a spherical subgroup and let $\overline H$ be
the spherical closure of~$H$. For every $\sigma \in \Sigma_{G/H}$,
we introduce the spherical root $\overline \sigma \in \Sigma_G$ in
the following way.
$$
\overline \sigma = \begin{cases} 2\sigma & \text{if } \sigma \notin
\Pi, 2\sigma \in \Sigma_G, \text{ and the pair } (\Pi^p_{G/H},
2\sigma) \text{ is compatible}; \\ \sigma & \text{
otherwise}.\end{cases}
$$
We set $\overline \Sigma_{G/H} = \lbrace \overline \sigma \mid
\sigma \in \Sigma_{G/H} \rbrace$.

\begin{proposition}[{\cite[Lemma~7.1]{Lu01}}] %
\label{prop_sph_syst_of_sph_closure}
The spherical system of $G / \overline H$ is determined as follows:
\begin{enumerate}[label=\textup{(\alph*)},ref=\textup{\alph*}]
\item
$\Pi^p_{G / \overline H} = \Pi^p_{G / H}$;

\item
$\Sigma_{G / \overline H} = \overline \Sigma_{G / H}$;

\item
the natural map $G / H \to G / \overline H$ induces a bijection
$\mathcal D_{G / H} \to \mathcal D_{G / \overline H}$, and the map
$\varkappa_{G / \overline H}$ is the restriction of the map
$\varkappa_{G / H}$ to $\Lambda_{G / \overline H}$.
\end{enumerate}
\end{proposition}

\subsection{Distinguished subsets of colors and quotient systems}

\label{subsec_dist_subsets}

Let $(\Pi^p, \Sigma, \mathcal D^a)$ be a spherical system and let
$\mathcal D$ be the associated set of colors. Set $\Lambda = \ZZ
\Sigma$ and, for every subset ${\mathcal D' \subset \mathcal D}$,
let $\mathcal C_{\mathcal D'}$ denote the cone in $\mathcal Q =
\Hom_\ZZ(\Lambda, \QQ)$ generated by the set $\varkappa(\mathcal
D')$. We set $\mathcal V = \lbrace q \in \mathcal Q \mid \langle q,
\sigma \rangle \le 0 \text{ for all } \sigma \in \Sigma \rbrace$.

\begin{definition}[{see~\cite[\S\,3.3]{Lu01}}] %
\label{def_distinguished}
A subset $\mathcal D' \subset \mathcal D$ is said to be
\textit{distinguished} if either of the two equivalent conditions
below holds:
\begin{itemize}
\item the set $\mathcal C_{\mathcal D'}^\circ$ meets $-\mathcal V$;

\item there exists an element $\delta = \sum \limits_{D \in \mathcal
D'} n_D \varkappa(D)$, where $n_D > 0$ for all $D \in \mathcal D'$,
such that $\langle \delta, \sigma \rangle \ge 0$ for all $\sigma \in
\Sigma$.
\end{itemize}
\end{definition}

Let $\mathcal D' \subset \mathcal D$ be a distinguished subset of
colors. We put $\Sigma_{\mathcal D'} \subset \Sigma$ to be the set
of spherical roots $\sigma$ such that $\langle \delta, \sigma
\rangle = 0$ for all $\delta \in \mathcal C^\circ_{\mathcal D'} \cap
(-\mathcal V)$. Let $\mathcal V_{\mathcal D'}$ be the largest face
of $\mathcal V$ such that $\mathcal C^\circ_{\mathcal D'} \cap
(-\mathcal V_{\mathcal D'}^\circ) \ne \varnothing$. Then $\mathcal
V_{\mathcal D'} = \lbrace q \in \mathcal V \mid \langle q, \sigma
\rangle = 0 \text{ for all } \sigma \in \Sigma_{\mathcal D'}
\rbrace$. Let $V_{\mathcal D'}$ be the vector subspace of $\mathcal
Q$ generated by the set $\varkappa(\mathcal D') \cup \mathcal
V_{\mathcal D'}$. Clearly, $(V_{\mathcal D'}, \mathcal D')$ is a
colored subspace (see Definition~\ref{def_colored_subspace}) and
$V_{\mathcal D'} \cap \mathcal V = \mathcal V_{\mathcal D'}$.

Given a distinguished subset of colors $\mathcal D' \subset \mathcal
D$, one defines the \textit{quotient system} $(\Pi^p, \Sigma,
\mathcal D^a) / \mathcal D' = (\Pi^p / \mathcal D', \Sigma /
\mathcal D', \mathcal D^a / \mathcal D')$. To do that, we first
introduce the lattice
$$
\Lambda / \mathcal D' = \lbrace \lambda \in \Lambda \mid \langle q,
\lambda \rangle = 0 \text{ for all } q \in V_{\mathcal D'} \rbrace =
\lbrace \lambda \in \ZZ \Sigma_{\mathcal D'} \mid \langle
\varkappa(D), \lambda \rangle = 0 \text{ for all } D \in \mathcal D'
\rbrace.
$$
Then the elements of the quotient system are defined as follows:
\begin{itemize}
\item $\Pi^p / \mathcal D' = \lbrace \alpha \in \Pi \mid \mathcal
D(\alpha) \subset \mathcal D' \rbrace$;

\item $\Sigma / \mathcal D'$ is the set of indecomposable elements
of the semigroup $ \ZZ^+ \Sigma \cap \Lambda / \mathcal D'$;

\item $\mathcal D^a / \mathcal D'$ is the union of the sets $\mathcal
D(\alpha)$ over all $\alpha \in \Pi \cap (\Sigma / \mathcal D')$,
and the map $\varkappa / \mathcal D'$ is the restriction of
$\varkappa$ to $\mathcal D^a / \mathcal D'$ followed by the
projection $\Hom_\ZZ(\Lambda, \ZZ) \to \Hom_\ZZ(\Lambda / \mathcal
D', \ZZ)$.
\end{itemize}

\begin{remark}
Using the condition $\mathcal C^\circ_{\mathcal D'} \cap (-\mathcal
V^\circ_{\mathcal D'}) \ne \varnothing$, one can show that the
semigroup $\ZZ^+ \Sigma \cap \Lambda / \mathcal D'$ can be expressed
as
$$
\ZZ^+ \Sigma \cap \Lambda / \mathcal D' = \lbrace \lambda \in \ZZ^+
\Sigma \mid \langle \varkappa(D), \lambda \rangle = 0 \text{ for all
} D \in \mathcal D' \rbrace,
$$
which can be useful in computations.
\end{remark}

Recently Bravi proved that the semigroup $\ZZ^+\Sigma \cap \Lambda /
\mathcal D'$ is always free and generates the lattice $\Lambda /
\mathcal D'$; see~\cite[Theorem~3.1]{Bra2}. So $\Sigma / \mathcal
D'$ is a basis of~$\Lambda / \mathcal D'$.

A $G$-equivariant morphism $X \to X'$ between two wonderful
$G$-varieties is said to be \textit{wonderful} if it is dominant
(and thereby surjective) and has connected fibers. We note that, in
case where $X' = G / P$ for a parabolic subgroup~$P \subset G$,
every $G$-equivariant morphism $X \to G / P$ is wonderful.

Let $\phi \colon X \to X'$ be a wonderful morphism between two
wonderful varieties. Let $\mathcal D'_X(\phi) \subset\nobreak
\mathcal D_X$ denote the set of colors that map dominantly (and
thereby surjectively) onto~$X'$. Bravi's result together
with~\cite[Proposition~3.3.2]{Lu01} imply the following proposition.

\begin{proposition} \label{prop_wonderful_morphisms}
The following assertions hold:
\begin{enumerate}[label=\textup{(\alph*)},ref=\textup{\alph*}]
\item
the map $\phi \mapsto \mathcal D'_X(\phi)$ is a bijection between
wonderful morphisms $\phi \colon X \to X'$ and distinguished subsets
of\,~$\mathcal D_X$;

\item
for the wonderful morphism $X \to X'$ corresponding to a
distinguished subset $\mathcal D' \subset\nobreak \mathcal D_X$, the
spherical system of $X'$ is given by $(\Pi^p_X, \Sigma_X, \mathcal
D^a_X) / \mathcal D'$. In particular, $(\Pi^p_X, \Sigma_X, \mathcal
D^a_X) / \mathcal D'$ is a spherical system.
\end{enumerate}
\end{proposition}

\subsection{Characterization of strongly solvable spherical
subgroups} \label{subsec_charact_str_solv}

The main goal of this subsection is to obtain a characterization of
strongly solvable spherical (resp. wonderful) subgroups in~$G$ in
terms of their homogeneous spherical data (resp. spherical systems).

We retain all the notation introduced
in~\S\,\ref{subsec_dist_subsets}.

\begin{definition}
A homogeneous spherical datum (resp. spherical system) is said to be
\textit{strongly solvable} if the corresponding spherical (resp.
wonderful) subgroup of~$G$ is strongly solvable.
\end{definition}

\begin{proposition}[{see~\cite[\S\,1]{Lu93}}]
\label{prop_strongly_solvable} A spherical system $\mathscr S =
(\Pi^p, \Sigma, \mathcal D^a)$ is strongly solvable if and only if
there exists a subset $\mathcal D' \subset \mathcal D^a$ having the
following properties:
\begin{enumerate}[label=\textup{(\arabic*)},ref=\textup{\arabic*}]
\item \label{prop_strongly_solvable1}
the cone generated by the set $\mathcal V \cup \varkappa(\mathcal
D')$ coincides with $\mathcal Q$ or, equivalently, the set $\mathcal
C^\circ_{\mathcal D'}$ meets $-\mathcal V^\circ$ or, equivalently,
there exists an element $\delta = \sum \limits_{D \in \mathcal D'}
n_D \varkappa(D)$, where $n_D > 0$ for all $D \in \mathcal D'$, such
that $\langle \delta, \sigma \rangle > 0$ for all $\sigma \in
\Sigma$;

\item \label{prop_strongly_solvable2}
$|\mathcal D \backslash \mathcal D'| = |\Pi|$, where $\mathcal D$ is
the set of colors associated with~$\mathscr S$.
\end{enumerate}
\end{proposition}

\begin{remark}
Condition~(\ref{prop_strongly_solvable1}) guarantees that the set
$\mathcal D'$ is distinguished.
\end{remark}

\begin{proof}[Proof of Proposition~\textup{\ref{prop_strongly_solvable}}]
Let $H \subset G$ be a wonderful subgroup with $\mathscr S_{G / H} =
\mathscr S$ and let $X$ be the wonderful embedding of $G / H$. In
what follows, we identify $\mathcal D$ with the set of colors
of~$X$. Consider a wonderful morphism $X \to X'$ and denote by
$\mathcal D'$ the corresponding distinguished subset of colors. By
the definition of a quotient system,
condition~(\ref{prop_strongly_solvable1}) holds for $\mathcal D'$ if
and only if $\Sigma / \mathcal D' = \varnothing$, which in turn is
equivalent to $X'$ being homogeneous (see
Example~\ref{example_hom_wond}), that is, $X' = G / P$ for a
parabolic subgroup $P \subset G$. Now assume $X'$ to be homogeneous
so that condition~(\ref{prop_strongly_solvable1}) holds
for~$\mathcal D'$. Then the colors in $\mathcal D \backslash
\mathcal D'$ are exactly the preimages of colors in~$X'$. So
condition~(\ref{prop_strongly_solvable2}) holds for $\mathcal D'$ if
and only if $X'$ contains exactly $|\Pi|$ colors. It is well known
that the latter is equivalent to $X' = G / B$. In this case, for
every $\alpha \in \Pi$ one has $|(\mathcal D \backslash \mathcal D')
\cap \mathcal D(\alpha)| = 1$, whence $\mathcal D' \subset \mathcal
D^a$ by Proposition~\ref{prop_alternative}.

Now let us prove the assertion. The subgroup $H$ is strongly
solvable if and only if there exists a $G$-equivariant morphism $G /
H \to G / B$. By~\cite[Theorem~4.1]{Kn91}, such a morphism always
extends to a $G$-equivariant morphism $X \to G / B$, and we may
apply the above reasoning.
\end{proof}

\begin{corollary} \label{crl_strongly_solvable_SS}
If $\mathscr S = (\Pi^p, \Sigma, \mathcal D^a)$ is a strongly
solvable spherical system, then $\Pi^p = \varnothing$ and $\Sigma
\subset \Pi$.
\end{corollary}

\begin{proof}
Let $\mathcal D$ be the set of colors associated with~$\mathscr S$
and let $\mathcal D' \subset \mathcal D^a$ be a subset satisfying
the conditions of Proposition~\ref{prop_strongly_solvable}. In the
proof of the proposition it was shown that $|(\mathcal D \backslash
\mathcal D') \cap \mathcal D(\alpha)| = 1$ for every $\alpha \in
\Pi$, whence $\Pi^p = \varnothing$. Assume that $\sigma \in \Sigma
\backslash \Pi$. Then axiom~(\ref{A1}) yields $\langle \varkappa(D),
\sigma \rangle \le 0$ for all $D \in \mathcal D^a$, hence the cone
generated by $\mathcal V \cup \mathcal D'$ is contained in the
half-space $\lbrace q \in \mathcal Q \mid \langle q, \sigma \rangle
\le 0 \rbrace$ of~$\mathcal Q$. The latter contradicts
condition~(\ref{prop_strongly_solvable1}) of
Proposition~\ref{prop_strongly_solvable}, thus $\Sigma \subset \Pi$.
\end{proof}

\begin{corollary} \label{crl_str_solv_wond_sph_clos}
Every strongly solvable wonderful subgroup of $G$ is spherically
closed.
\end{corollary}

\begin{proof}
Let $H \subset G$ be a strongly solvable wonderful subgroup and let
$\overline H$ be the spherical closure of~$H$. Since $\Sigma_{G/H}
\subset \Pi$, by Proposition~\ref{prop_sph_syst_of_sph_closure} the
spherical systems of $G / H$ and $G / \overline H$ coincide. The
description of spherical subgroups of~$G$ with a given spherical
closure (see~\cite[Proposition~6.3]{Lu01}) yields $H = \overline H$.
(Of course, here one may also refer to Theorems~\ref{thm_uniqueness}
or~\ref{thm_classification} as more general results.)
\end{proof}

\begin{corollary} \label{crl_HSD_str_solv}
A homogeneous spherical datum $\mathscr H = (\Lambda, \Pi^p, \Sigma,
\mathcal D^a)$ is strongly solvable if and only if the spherical
system $\mathscr S = (\Pi^p, \Sigma, \mathcal D^a)$ \textup(where
$\varkappa$ is restricted to~$\ZZ \Sigma$\textup) is so.
\end{corollary}

\begin{proof}
Let $H$ be a spherical subgroup with $\mathscr H_{G / H} = \mathscr
H$ and let $\overline H$ be the spherical closure of~$H$. If $H$ is
strongly solvable, then $\overline H$ is also strongly solvable by
Corollary~\ref{crl_sph_clos_str_solv}. Therefore $\Sigma_{G /
\overline H} \subset \Pi$ and
Proposition~\ref{prop_sph_syst_of_sph_closure} implies $\mathscr S =
\mathscr S_{G / \overline H}$. Conversely, if $\mathscr S$ is
strongly solvable, then $\Sigma \subset \Pi$ and $\mathscr S_{G /
\overline H} = \mathscr S$ by
Proposition~\ref{prop_sph_syst_of_sph_closure}. Hence $\mathscr H$
is strongly solvable.
\end{proof}

\begin{corollary} \label{crl_strongly_solvable_HSD}
If $\mathscr H = (\Lambda, \Pi^p, \Sigma, \mathcal D^a)$ is a
strongly solvable homogeneous spherical datum, then $\Pi^p =
\varnothing$ and $\Sigma \subset \Pi$.
\end{corollary}

\begin{remark}
The converse to Corollary~\ref{crl_strongly_solvable_SS} is not
true. For example, consider the group $G = \SL_2 \times \SL_2 \times
\SL_2$, take the subgroup $H_0 \simeq \SL_2$ diagonally embedded
in~$G$, and set $H = N_G(H_0) = Z(G)H_0$. It is well known that
$H_0$ is spherical in~$G$, hence so is~$H$. Clearly, $N_G(H) = H$,
therefore $H$ is a wonderful subgroup of~$G$ by
Theorem~\ref{thm_sph_closed_wonderful}. Let $\varpi_i$ be the
fundamental weight of the $i$th factor of~$G$, $i = 1,2,3$. One has
$\mathfrak X(H) \simeq (\ZZ / 2\ZZ) \oplus (\ZZ / 2\ZZ)$, and the
semigroup $\widehat \Lambda^+_{G/H}$ is freely generated by the
elements $(\varpi_1 + \varpi_2, a)$, $(\varpi_2 + \varpi_3, b)$,
$(\varpi_3 + \varpi_1, c)$, where $a,b,c$ are some pairwise
different order~$2$ elements in $\mathfrak X(H)$. By
Proposition~\ref{prop_Sigma_via_supports} one has $\Pi^p_{G/H} =
\varnothing$ and $\Sigma_{G/H} = \Pi$, however $H$ is not strongly
solvable in~$G$.
\end{remark}

\section{Luna's 1993 approach for classifying\\
strongly solvable wonderful subgroups} \label{sect_Luna_1993}

\subsection{Spherical and wonderful $B^-$-varieties}
\label{subsec_spher&wond_B-var} In this subsection we introduce
spherical and wonderful $B^-$-varieties and reduce the
classification of wonderful strongly solvable subgroups of~$G$ to
that of wonderful $B^-$-varieties.

\begin{definition} \label{def_spher_B-var}
A normal irreducible $B^-$-variety $Z$ is said to be
\textit{spherical} if $T$ has an open orbit in~$Z$.
\end{definition}

\begin{definition} \label{def_wond_B-var}
A spherical $B^-$-variety $Z$ is said to be \textit{wonderful} if it
possesses the following properties:
\begin{enumerate}[label=\textup{(WB\arabic*)},ref=\textup{WB\arabic*}]
\item \label{def_wond_B-var1}
$Z$ is smooth and complete;

\item \label{def_wond_B-var2}
$Z$ contains exactly one closed $B^-$-orbit (which is necessarily a
fixed point~$z_0$);

\item \label{def_wond_B-var3}
every irreducible $T$-stable closed subvariety $Z' \subset Z$
containing $z_0$ is actually $B^-$-stable.
\end{enumerate}
\end{definition}

\begin{proposition} \label{prop_wonderful_B&G}
Let $Z$ be a $B^-$-variety and consider the $G$-variety $X = G
*_{B^-} Z$.
\begin{enumerate}[label=\textup{(\alph*)},ref=\textup{\alph*}]
\item \label{prop_wonderful_B&G_a}
$Z$ is a spherical $B^-$-variety if and only if $X$ is a spherical
$G$-variety.

\item \label{prop_wonderful_B&G_b}
$Z$ is a wonderful $B^-$-variety if and only if $X$ is a wonderful
$G$-variety.
\end{enumerate}
\end{proposition}

To prove this proposition, we need an additional consideration.
Namely, we consider the natural $G$-equivariant morphism $\phi
\colon X \to G / B^-$. Note that $Z$ is naturally identified with
the subset $\phi^{-1}(o) \subset X$. Since the subset $B o = U o
\subset G / B^-$ is open, the subset $X_0 = \phi^{-1}(B o) \subset
X$ is open and $B$-stable. Applying
Proposition~\ref{prop_morphism_to_L/K} to the $B$-equivariant
morphism $X_0 \to B o \simeq B / T$, we get $X_0 \simeq B *_T Z
\simeq U \times Z$, where the latter variety is acted on by $B$ by
the formula
\begin{equation} \label{eqn_action_on_X0}
tv \cdot (u, z) = (tvut^{-1}, tz) \quad (t \in T, v,u \in U, z \in
Z).
\end{equation}

\begin{proof}[Proof of Proposition~\textup{\ref{prop_wonderful_B&G}}]
(\ref{prop_wonderful_B&G_a}) Evidently, $X$ is spherical if and only
if $B$ has an open orbit in~$X_0$. Formula~(\ref{eqn_action_on_X0})
shows that the latter holds if and only if $T$ has an open orbit
in~$Z$.

(\ref{prop_wonderful_B&G_b}) We shall use the interpretation of
wonderful $G$-varieties as wonderful completions of spherical
homogeneous spaces, see Theorem~\ref{thm_criterion_wonderful}.

First suppose that $X$ is a wonderful $G$-variety and let $X^c
\subset X$ be the closed $G$-orbit. Since $X$ is smooth, $Z$ is also
smooth by Proposition~\ref{prop_smoothness}. Clearly, $Z$ is
complete. Next, for every $G$-orbit $O \subset X$, the intersection
$O \cap Z$ is a $B^-$-orbit, and $O$ is closed if and only if $O
\cap Z$ is so. This proves that $X^c \cap Z$ is a unique closed
$B^-$-orbit in~$Z$, which is necessarily a fixed point~$z_0$. Let
$Z' \subset Z$ be an irreducible $T$-stable closed subvariety
containing~$z_0$. Then $BZ'$ is a closed subvariety in~$X_0$
containing $X^c \cap X_0$. Let $X'$ be the closure of $BZ'$ in~$X$.
Clearly, $X'$ is an irreducible $B$-stable closed subvariety
containing~$X^c$. As $X$ is toroidal, we obtain that $X'$ is
$G$-stable, whence $Z' = X' \cap Z$ is $B^-$-stable, so that $Z$ is
a wonderful $B^-$-variety.

Now suppose that $Z$ is a wonderful $B^-$-variety. Then $X$ is
smooth and complete by Propositions~\ref{prop_smoothness}
and~\ref{prop_completeness}. If $z_0$ is the (unique) point in $Z$
fixed by~$B^-$, then $X^c = Gz_0$ is a unique closed $G$-orbit
in~$X$. Let $X' \subset X$ be an irreducible $B$-stable closed
subvariety containing $X^c$. Then $Z' = X' \cap Z$ is an irreducible
$T$-stable closed subvariety containing~$z_0$, whence $Z'$ is
$B^-$-stable. The latter implies that $X' = GZ'$. Thus $X$ is a
wonderful $G$-variety.
\end{proof}

For every wonderful $B^-$-variety~$Z$, let $H_Z$ be the stabilizer
of a point of the open $B^-$-orbit in~$Z$.

\begin{theorem} \label{thm_bij_wond_B-var_SSWS}
The map $Z \mapsto H_Z$ is a bijection between wonderful
$B^-$-varieties \textup(considered up to $B^-$-equivariant
isomorphism\textup) and conjugacy classes in $B^-$ of strongly
solvable wonderful subgroups of $G$ contained in~$B^-$.
\end{theorem}

\begin{proof}
Let $Z$ be a wonderful $B^-$-variety.
Proposition~\ref{prop_wonderful_B&G}(\ref{prop_wonderful_B&G_b})
yields that $X = G *_{B^-} Z$ is a wonderful $G$-variety and its
open $G$-orbit is isomorphic to $G / H_Z$, so that $H_Z$ is
wonderful in~$G$.

Conversely, let $H \subset G$ be a wonderful subgroup contained
in~$B^-$ and let $X$ be the wonderful embedding of the homogeneous
space $G / H$. By~\cite[Theorem~4.1]{Kn91} the natural morphism $G /
H \to G / B^-$ extends to a $G$-equivariant morphism $\phi \colon X
\to G / B^-$. Therefore $X = G *_{B^-} Z$, where $Z = \phi^{-1}(o)$
(see Proposition~\ref{prop_morphism_to_L/K}).
Proposition~\ref{prop_wonderful_B&G}(\ref{prop_wonderful_B&G_b})
implies that $Z$ is a wonderful $B^-$-variety.
\end{proof}

\subsection{Connected automorphism groups of smooth complete toric
varieties} \label{subsec_tor_var_aut}

In this subsection we present a description of the connected
automorphism group of a smooth complete toric variety, which goes
back to Demazure~\cite{Dem} (see also~\cite[\S\,3.4]{Oda}). In our
exposition we follow a more modern viewpoint on this description,
which is due to Cox~\cite{Cox}.

In this paper, we adopt the following definition of a toric variety.

\begin{definition} \label{def_toric_var}
A normal irreducible $T$-variety $Z$ is said to be \textit{toric} if
it possesses an open $T$-orbit.
\end{definition}

We note that the toric $T$-varieties (in the sense of
Definition~\ref{def_toric_var}) are exactly the spherical
$T$-varieties.

Let $Z$ be a toric $T$-variety and let $T_0$ denote the quotient of
$T$ by the kernel of its action on~$Z$. Following the notation
widely used in the theory of toric varieties, we denote by~$M$ the
weight lattice of $Z$ and put $N = \Hom_\ZZ (M, \ZZ)$. Clearly, $M
\simeq \mathfrak X(T_0)$. Put also $N_\QQ = N \otimes_\ZZ \QQ \simeq
\Hom_{\ZZ}(M, \QQ)$. It is well known that $Z$ determines a strictly
convex fan $\mathcal F$ in $N_\QQ$ and the map $Z \mapsto (M,
\mathcal F)$ is a bijection between toric $T$-varieties (considered
up to $T$-equivariant isomorphism) and all pairs $(M, \mathcal F)$
with $M$ a sublattice of $\mathfrak X(T)$ and $\mathcal F$ a
strictly convex fan in~$\Hom_\ZZ (M, \QQ)$.

Let $Z$ be a toric $T$-variety and let $\mathcal F$ be the
corresponding fan in~$N_\QQ$. The following two well-known facts
will be of particular importance for us:

\begin{enumerate}[label=\textup{(F\arabic*)},ref=\textup{F\arabic*}]
\item
$Z$ is complete if and only if $\mathcal F$ is complete;

\item
$Z$ is smooth if and only if $\mathcal F$ is regular.
\end{enumerate}

We recall that the set $\mathcal F^1$ is in bijection with the set
of $T$-stable prime divisors in~$Z$. For every $\varrho \in \mathcal
F^1$, let $D_\varrho$ denote the corresponding $T$-stable prime
divisor in~$Z$.

To each $\varrho \in \mathcal F^1$ we assign a variable~$x_\varrho$.
Let $\CR = \CR(Z)$ denote the polynomial ring in
variables~$x_\varrho$, where $\varrho$ runs over the set~$\mathcal
F^1$. Put $f = |\mathcal F^1|$.

\begin{definition}
The ring $\CR$ is said to be the \textit{Cox ring} of the toric
variety~$Z$.
\end{definition}

The ring $\CR$ is naturally acted on by a torus $\mathbf T \simeq
(\CC^\times)^f$. For every $\varrho \in \mathcal F^1$, let
$\chi_\varrho$ be the character by which $\mathbf T$ acts
on~$x_\varrho$. Clearly, the map $\chi_\varrho \mapsto D_\varrho$
extends to an isomorphism between $\mathfrak X(\mathbf T)$ and the
group of $T$-stable Weil divisors in~$Z$. We denote the latter group
by~$\ZZ^f$.

Every monomial $\prod \limits_{\varrho \in \mathcal F^1}
x_\varrho^{a_\varrho}$ determines a $T$-stable Weil divisor $D =
\sum \limits_{\varrho \in \mathcal F^1} a_\varrho D_\varrho$. We
shall also write this monomial as~$x^D$. The ring $\CR$ has a
natural grading by the divisor class group $\Cl Z$. By definition,
the degree of a monomial $x^D$ is $[D] \in \Cl Z$. For every $c \in
\Cl Z$, let $\CR_c$ denote the linear span of all monomials $x^D$
with $[D] = c$, so that $\CR = \bigoplus \limits_{c \in \Cl Z}
\CR_c$. We note that for complete $Z$ each of the subspaces $\CR_c$
is finite-dimensional (see~\cite[Corollary~1.2(i)]{Cox}).

It is known that the group $\Cl Z$ is finitely generated, whence it
may be identified with the character group of a uniquely determined
quasitorus $\mathbf S$. Fix an isomorphism $\Cl Z \to \mathfrak
X(\mathbf S)$, $c \mapsto \chi_c$. The group $\mathbf S$ acts
naturally on $\CR$ preserving the grading: each component $\CR_c$ is
multiplied by the character~$\chi_c$.

Set $\widetilde Z = \Spec \CR$.

For every cone $\mathcal C \in \mathcal F$, we define the monomial
$r(\mathcal C) = \prod \limits_{\varrho \in \mathcal F^1 \backslash
\mathcal C} x_\varrho$ and let $I$ be the ideal of $\CR$ generated
by all the monomials $r(\mathcal C)$, $\mathcal C \in \mathcal F$.
Let $E$ be the subvariety of $\widetilde Z$ defined by the vanishing
of all polynomials in~$I$. As $I$ is $\mathbf S$-stable, $E \subset
\widetilde Z$ is $\mathbf S$-stable as well. Besides,
by~\cite[Lemma~1.4]{Cox} the set $E$ has codimension at least two
in~$\widetilde Z$.

\begin{theorem}[{see~\cite[Theorem~2.1(iii)]{Cox}}]
\label{thm_tor_var_geom_quot} If $Z$ is smooth, then $Z$ is
naturally isomorphic to the geometric quotient of $\widetilde Z
\backslash E$ by the action of~$\mathbf S$.
\end{theorem}

We note that for every $\varrho \in \mathcal F^1$ the preimage of
the divisor $D_\varrho$ under the morphism $\widetilde Z \backslash
E \to Z$ is defined by the equation $x_\varrho = 0$.

For every $\alpha \in M$, let $f_\alpha \in \mathbb C(Z)$ be a
$T$-semi-invariant rational function of weight~$\alpha$, which is
unique up to proportionality. Its Weil divisor is given by $\Div
f_\alpha = \sum \limits_{\varrho \in \mathcal F^1} \langle \varrho,
\alpha \rangle D_\varrho$. By~\cite[\S\,3.4]{Ful}, the map $M \to
\ZZ^f$ defined by $\alpha \mapsto \Div f_\alpha$ is included into
the following exact sequence:
$$
0 \to M \to \ZZ^f \to \Cl Z \to 0.
$$
The identifications $M \simeq \mathfrak X(T_0)$, $\ZZ^f \simeq
\mathfrak X(\mathbf T)$, and $\Cl Z \simeq \mathfrak X(\mathbf S)$
yield the exact sequence
$$
0 \to \mathfrak X(T_0) \to \mathfrak X(\mathbf T) \to \mathfrak
X(\mathbf S) \to 0,
$$
where the map $\mathfrak X(T_0) \to \mathfrak X(\mathbf T)$ is
defined by $\alpha \mapsto \sum \limits_{\varrho \in \mathcal F^1}
\langle \varrho, \alpha \rangle \chi_\varrho$. In particular, $T_0
\simeq \mathbf T / \mathbf S$.

We now turn to the problem of determining the connected automorphism
group of~$Z$. In what follows, $Z$ is assumed to be smooth and
complete.

Let $\Aut_g(\CR)$ be the group of grading-preserving $\CC$-algebra
automorphisms of~$\CR$. By~\cite[Proposition~4.3(i)]{Cox},
$\Aut_g(\CR)$ is a connected affine algebraic group. Clearly,
$\mathbf T \subset \Aut_g(\CR)$ and the subgroup $\mathbf S$ is
identified with a central subgroup of $\Aut_g(\CR)$. Therefore every
element of $\Aut_g(\CR)$, regarded as an automorphism of $\widetilde
Z$, preserves $\mathbf S$-orbits and by
Theorem~\ref{thm_tor_var_geom_quot} descends to an automorphism
of~$Z$, so that there is a homomorphism
$$
d \colon \Aut_g(\CR) \to \Aut Z.
$$

\begin{definition}[{see~{\cite[\S\,4.5]{Dem}}}] \label{def_root}
An element $\alpha \in M$ is said to be a \textit{root} of the fan
$\mathcal F$ if there exists an element $\varrho_\alpha \in \mathcal
F^1$ with $\langle \varrho_\alpha, \alpha \rangle = 1$ and $\langle
\varrho, \alpha \rangle \le 0$ for all $\varrho \in \mathcal F^1
\backslash \lbrace \varrho_\alpha \rbrace$.
\end{definition}

We note that the element $\varrho_\alpha$ is unique if exists.

Let $R(\mathcal F)$ denote the set of roots of the fan~$\mathcal F$.
Since $\mathcal F$ is complete, it can be easily shown that
$R(\mathcal F)$ is finite.

For every $\alpha \in R(\mathcal F)$, let $y_{-\alpha}$ be the
derivation of the ring $\CR$ defined on the generators as follows:
\begin{equation*}
\begin{split}
y_{-\alpha} (x_{\varrho_\alpha}) &= \prod \limits_{\varrho \in
\mathcal F^1 \backslash \lbrace \varrho_\alpha \rbrace}
x_\varrho^{- \langle \varrho, \alpha \rangle}; \\
y_{-\alpha} (x_\varrho) &= 0 \text{\;\;for\;\;} \varrho \in \mathcal
F^1 \backslash \lbrace \varrho_\alpha \rbrace.
\end{split}
\end{equation*}
For every $\zeta \in \CC$, we put $Y_{-\alpha}(\zeta) = \exp(\zeta
y_{-\alpha})$ and introduce the one-parameter subgroup
$$
Y_{-\alpha} = \lbrace Y_{-\alpha}(\zeta) \mid \zeta \in \CC \rbrace
\subset \Aut (\CR).
$$
The action of $Y_{-\alpha}$ on $\CR$ is described as follows:
\begin{equation} \label{eqn_group_y_alpha}
\begin{split}
[Y_{-\alpha}(\zeta)] x_{\varrho_\alpha} &= x_{\varrho_\alpha} +
\zeta \prod \limits_{\varrho \in \mathcal F^1 \backslash \lbrace
\varrho_\alpha \rbrace} x_\varrho^{- \langle
\varrho, \alpha \rangle}; \\
[Y_{-\alpha}(\zeta)] x_\varrho &= x_\varrho \text{ for } \varrho \in
\mathcal F^1 \backslash \lbrace \varrho_\alpha \rbrace.
\end{split}
\end{equation}
Clearly, $y_{-\alpha}$ preserves the grading, and so $Y_{-\alpha}
\subset \Aut_g(\CR)$. Moreover, $Y_{-\alpha}$ maps isomorphically
onto~$d(Y_{-\alpha})$.

The minus sign in the notation $y_{-\alpha}$ and $Y_{-\alpha}$ is
justified by the following proposition.

\begin{proposition} \label{prop_normalized-alpha}
The subgroup $d(Y_{-\alpha}) \subset \Aut Z$ is normalized by~$T$
with weight $-\alpha$.
\end{proposition}

\begin{proof}
Let $t \in \mathbf T$ and $\zeta \in \CC$. Clearly, $[t \cdot
Y_{-\alpha}(\zeta) \cdot t^{-1}]x_\varrho = x_\varrho$ for all
$\varrho \in \mathcal F^1 \backslash \lbrace \varrho_\alpha
\rbrace$. Further,
\begin{multline}
[t \cdot Y_{-\alpha}(\zeta) \cdot t^{-1}]x_{\varrho_\alpha} =%
[t \cdot Y_{-\alpha}(\zeta)] \chi_{\varrho_\alpha}(t)^{-1}
x_{\varrho_\alpha} =\\
t(\chi_{\varrho_\alpha}(t)^{-1} x_{\varrho_\alpha} +
\chi_{\varrho_\alpha}(t)^{-1} \zeta \prod \limits_{\varrho \in
\mathcal F^1 \backslash \lbrace \varrho_\alpha \rbrace}
x_\varrho^{- \langle \varrho, \alpha \rangle}) =\\
x_{\varrho_\alpha} + \zeta \prod \limits_{\varrho \in \mathcal F^1}
\chi_\varrho(t)^{-\langle \varrho, \alpha \rangle} \prod
\limits_{\varrho \in \mathcal F^1 \backslash \lbrace \varrho_\alpha
\rbrace} x_\varrho^{- \langle \varrho, \alpha \rangle}.
\end{multline}
We have obtained the equality $t \cdot Y_{-\alpha}(\zeta) \cdot
t^{-1} = Y_{-\alpha}(\chi(t)\zeta)$ in the group~$\Aut_g(\CR)$,
where $\chi = -\sum \limits_{\varrho \in \mathcal F^1} \langle
\varrho, \alpha \rangle \chi_\varrho$. Since $\chi$ is nothing else
than the image of $-\alpha$ in $\mathfrak X(\mathbf T)$, the group
$d(Y_{-\alpha})$ is normalized by $d(\mathbf T) \simeq T_0$ with
weight $-\alpha$.
\end{proof}

\begin{theorem}[{see~\cite[Corollary~4.7]{Cox}}]
\label{thm_tor_var_aut_group} The following assertions hold:
\begin{enumerate}[label=\textup{(\alph*)},ref=\textup{\alph*}]
\item \label{thm_tor_var_aut_group_a}
the group $\Aut Z$ is an affine algebraic group, and the group $T_0
= d(\mathbf T)$ is a maximal torus of $\Aut Z$;

\item \label{thm_tor_var_aut_group_b}
the group $(\Aut Z)^0$ is generated by $T_0$ and the groups
$d(Y_{-\alpha})$ for all $\alpha \in R(\mathcal F)$;

\item \label{thm_tor_var_aut_group_c}
the sequence of homomorphisms
$$
1 \to \mathbf S \to \Aut_g(\CR) \xrightarrow{d} (\Aut Z)^0 \to 1
$$
is exact; in particular, $(\Aut Z)^0 \simeq \Aut_g(\CR) / \mathbf
S$.
\end{enumerate}
\end{theorem}

In the remaining part of this subsection we state and prove several
lemmas that will be needed in the following subsections.

\begin{lemma} \label{lemma_pair_of_roots}
Let $\alpha, \beta \in R(\mathcal F)$.
\begin{enumerate}[label=\textup{(\alph*)},ref=\textup{\alph*}]
\item \label{lemma_pair_of_roots_a}
If $\langle \varrho_\beta, \alpha \rangle < 0$ and $\langle
\varrho_\alpha, \beta \rangle < 0$, then $\alpha + \beta = 0$.

\item \label{lemma_pair_of_roots_b}
If $\langle \varrho_\beta, \alpha \rangle = 0$ and $\langle
\varrho_\alpha, \beta \rangle = - p < 0$, then $\alpha + \beta \in
R(\mathcal F)$, $\varrho_{\alpha + \beta} = \varrho_\beta$, and
$\langle \varrho_\alpha, \alpha + \beta \rangle = - p + 1$.
\end{enumerate}
\end{lemma}

\begin{proof}
(\ref{lemma_pair_of_roots_a}) The hypothesis implies that $\langle
\varrho, \alpha + \beta \rangle \le 0$ for all $\varrho \in \mathcal
F^1$. Since the fan $\mathcal F$ is complete, it follows that
$\alpha + \beta = 0$.

(\ref{lemma_pair_of_roots_b}) Obvious.
\end{proof}

\begin{lemma} \label{lemma_commutators}
Suppose that $\alpha, \beta \in R(\mathcal F)$, $\langle
\varrho_\beta, \alpha \rangle = 0$, and $\langle \varrho_\alpha,
\beta \rangle = - p \le 0$. Then:
\begin{enumerate}[label=\textup{(\alph*)},ref=\textup{\alph*}]
\item \label{lemma_commutators_a}
$[y_{-\alpha}, y_{-\beta}] = - p y_{-\alpha - \beta}$ for $p > 0$
and $[y_{-\alpha}, y_{-\beta}] = 0$ for $p = 0$;

\item \label{lemma_commutators_b}
$(\ad y_{-\alpha})^{q}y_{-\beta} \ne 0$ for $0 \le q \le p$ and
$(\ad y_{-\alpha})^{q}y_{-\beta} = 0$ for $q \ge p + 1$.
\end{enumerate}
\end{lemma}

\begin{proof}
Part~(\ref{lemma_commutators_a}) is obtained by a direct
computation, part~(\ref{lemma_commutators_b}) is a consequence
of~(\ref{lemma_commutators_a}) and
Lemma~\textup{\ref{lemma_pair_of_roots}(\ref{lemma_pair_of_roots_b})}.
\end{proof}

Fix a root $\alpha \in R(\mathcal F)$ and an element $\varrho \in
\mathcal F^1$.

\begin{lemma} \label{lemma_Y-unstable_divisor}
The divisor $D_\varrho$ is $d(Y_{-\alpha})$-unstable if and only if
$\varrho = \varrho_\alpha$.
\end{lemma}

\begin{proof}
The divisor $D_\varrho$ is $d(Y_{-\alpha})$-unstable if and only if
the element $x_\varrho \in \CR$ is not $Y_{-\alpha}$-invariant.
By~(\ref{eqn_group_y_alpha}) the latter holds if and only if
$\varrho = \varrho_\alpha$.
\end{proof}

Fix a root $\alpha \in R(\mathcal F)$ and a maximal cone $\mathcal C
\in \mathcal F$. Let $z \in Z$ be the $T$-fixed point corresponding
to the cone~$\mathcal C$.

\begin{lemma} \label{lemma_Y-unstable_point}
The point $z$ is $d(Y_{-\alpha})$-unstable if and only if
$\varrho_\alpha \in \mathcal C^1$ and $\langle \varrho, \alpha
\rangle = 0$ for all $\varrho \in \mathcal C^1 \backslash \lbrace
\varrho_\alpha \rbrace$.
\end{lemma}

\begin{proof}
Since $\lbrace z \rbrace = \bigcap \limits_{\varrho \in \mathcal
C^1} D_\varrho$, the preimage of $z$ under the morphism $\widetilde
Z \backslash E \to Z$ is the set of zeros of the ideal $I_z \subset
\CR$ generated by all variables $x_\varrho$ with $\varrho \in
\mathcal C^1$. So $z$ is $d(Y_{-\alpha})$-unstable if and only if
$I_z$ is $Y_{-\alpha}$-unstable. By
Lemma~\ref{lemma_Y-unstable_divisor}, the condition $\varrho_\alpha
\in \mathcal C^1$ is necessary for $z$ to be
$d(Y_{-\alpha})$-unstable. Under this condition,
from~(\ref{eqn_group_y_alpha}) we see that $I_z$ is
$Y_{-\alpha}$-unstable if and only if $\prod \limits_{\varrho \in
\mathcal F^1 \backslash \lbrace \varrho_\alpha \rbrace} x_\varrho^{-
\langle \varrho, \alpha \rangle} \notin I_z$. Evidently, the latter
holds if and only if $\langle \varrho, \alpha \rangle = 0$ for all
$\varrho \in \mathcal C^1 \backslash \lbrace \varrho_\alpha
\rbrace$.
\end{proof}

\subsection{Classification of smooth complete spherical $B^-$-varieties}
\label{subsec_smooth_compl_sph_B-var}

Smooth complete spherical $B^-$-varieties are classified by
combinatorial objects called Enriques's \hbox{$B^-$-}systems.

\begin{definition}[{compare with \cite[\S\,2]{Lu93}}]
\label{def_B-system} An \textit{Enriques's $B^-$-system} is a triple
$(\mathfrak X, \mathcal F, \rho)$ consisting of the following
elements:

\begin{enumerate}
\item
$\mathfrak X$ is a sublattice of $\mathfrak X(T)$;

\item
$\mathcal F$ is a regular complete fan in~$Q =
\Hom_\ZZ(\mathfrak X, \QQ)$;

\item
$\rho \colon \Pi \to \mathcal F^1 \cup \{0\}$ is a map
satisfying the following conditions:

\begin{enumerate}
\renewcommand{\labelenumi}{}
\renewcommand{\theenumi}{}
\renewcommand{\labelenumii}{(\alph{enumii})}
\renewcommand{\theenumii}{\alph{enumii}}

\item \label{item_a}
if $\rho(\alpha) \ne 0$, then $\alpha \in \mathfrak X$, $\langle
\rho(\alpha), \alpha \rangle = 1$, and $\langle \varrho, \alpha
\rangle \le 0$ for every $\varrho \in \mathcal F^1 \backslash
\{\rho(\alpha)\}$;

\item \label{item_b}
$\langle \rho(\alpha), \beta \rangle \ge \langle \alpha^\vee, \beta
\rangle$ for every $\alpha, \beta \in \Pi$ with $\alpha \ne \beta$
and $\rho(\beta) \ne 0$.

\end{enumerate}
\end{enumerate}
\end{definition}

The following lemma asserts a property that was initially included
in the definition of an Enriques's $B^-$-system;
see~\cite[\S\,2]{Lu93}.

\begin{lemma} \label{lemma_EBS_property}
Let $(\mathfrak X, \mathcal F, \rho)$ be an Enriques's $B^-$-system.
If $\langle \rho(\alpha), \beta \rangle < 0$ for some $\alpha,\beta
\in \Pi \cap \mathfrak X$, then $\langle \rho(\beta), \alpha \rangle
= 0$.
\end{lemma}

\begin{proof}
It suffices to assume $\rho(\beta) \ne 0$, so that $\alpha, \beta
\in\nobreak R(\mathcal F)$. Then $\rho(\alpha) \ne \rho(\beta)$,
whence $\langle \rho(\beta), \alpha \rangle \le 0$. By
Lemma~\ref{lemma_pair_of_roots}(\ref{lemma_pair_of_roots_a}) the
inequality $\langle \rho(\beta), \alpha \rangle < \nobreak 0$ would
imply $\alpha + \beta = 0$, which is not the case. Hence $\langle
\rho(\beta), \alpha \rangle = 0$.
\end{proof}

Let $Z$ be a smooth complete spherical $B^-$-variety and let $\theta
\colon B^- \to (\Aut Z)^0$ be the natural homomorphism. By
definition of a spherical $B^-$-variety, $Z$ is a toric $T$-variety.
In this subsection we shall use all the notation associated with $Z$
in~\S\,\ref{subsec_tor_var_aut}. Let $\mathfrak X_Z \subset
\mathfrak X(T)$ be the weight lattice of~$Z$. Note that $\mathfrak
X_Z$ is naturally identified with the character lattice of the torus
$\theta(T)$. We put $Q_Z = \Hom_\ZZ (\mathfrak X_Z, \QQ)$. The
structure of a toric $T$-variety on~$Z$ determines a regular
complete fan $\mathcal F_Z$ in~$Q_Z$. For every $\alpha \in \Pi$,
let $U_{-\alpha}$ denote the corresponding one-dimensional unipotent
subgroup of~$U^-$. If $U_{-\alpha} \subset \Ker \theta$, then we put
$\rho_Z(\alpha) = 0$. Otherwise the group $\theta(U_{-\alpha})$ is a
unipotent subgroup of $\Aut Z$ normalized by~$\theta(T)$.
Theorem~\ref{thm_tor_var_aut_group}(\ref{thm_tor_var_aut_group_b})
then implies that $\alpha \in R(\mathcal F_Z)$ and
$\theta(U_{-\alpha})$ coincides with~$d(Y_{-\alpha})$. We take
$\rho_Z(\alpha) \in Q_Z$ to be the element $\varrho_\alpha$
associated with $\alpha$ as a root of the fan $\mathcal F_Z$ (see
Definition~\ref{def_root}).

\begin{proposition} \label{prop_spherical_B-var}
The map $Z \mapsto (\mathfrak X_Z, \mathcal F_Z, \rho_Z)$ is a
bijection between smooth complete spherical $B^-$-varieties
\textup(considered up to $B^-$-equivariant isomorphism\textup) and
Enriques's $B^-$-systems.
\end{proposition}

\begin{proof}
Let $Z$ be a smooth complete spherical $B^-$-variety and retain the
above notation. To show that $(\mathfrak X_Z, \mathcal F_Z, \rho_Z)$
is an Enriques's $B^-$-system, it remains to establish properties
(\ref{item_a}),~(\ref{item_b}) of Definition~\ref{def_B-system} for
the map~$\rho_Z$. Property~(\ref{item_a}) holds by construction and
Definition~\ref{def_root}. To prove~(\ref{item_b}), let $\alpha,
\beta \in \Pi$ be such that $\alpha \ne \beta$ and $\rho_Z(\beta)
\ne 0$. Since $\langle \alpha^\vee, \beta \rangle \le 0$, the
required inequality holds automatically whenever $\langle
\rho_Z(\alpha), \beta \rangle \ge 0$. Therefore in what follows we
assume $\langle \rho_Z(\alpha), \beta \rangle < 0$ (in particular,
$\rho_Z(\alpha) \ne 0$). Clearly, the homomorphism
$$
\left.\theta\right|_{U^-} \colon U^- \to (\Aut Z)^0
$$
lifts to a unique homomorphism
$$
\widetilde \theta \colon U^- \to \Aut_g(\CR(Z)).
$$
Since $\rho_Z(\alpha) \ne 0$ and $\rho_Z(\beta) \ne 0$, one has
$\widetilde \theta (U_{-\alpha}) = Y_{-\alpha}$ and $\widetilde
\theta (U_{-\beta}) = Y_{-\beta}$. The corresponding Lie algebra
homomorphism
$$
d\widetilde \theta \colon \mathfrak u^- \to \Der_g (\CR(Z)),
$$
where $\Der_g (\CR(Z))$ is the Lie algebra of grading-preserving
derivations of the ring~$\CR(Z)$, sends $e_{-\alpha}$ to a nonzero
multiple of $y_{-\alpha}$ and $e_{-\beta}$ to a nonzero multiple of
$y_{-\beta}$. One of the Serre relations (see~\cite[Chapter~VI,
Theorem~6(c)]{Ser} says that
$$
(\ad e_{-\alpha})^{1 - \langle \alpha^\vee, \beta \rangle}
e_{-\beta} = 0,
$$
hence $(\ad y_{-\alpha})^{1 - \langle \alpha^\vee, \beta \rangle}
y_{-\beta} = 0$. By
Lemma~\ref{lemma_commutators}(\ref{lemma_commutators_b}) we have $1
- \langle \alpha^\vee, \beta \rangle \ge 1 - \langle \rho_Z(\alpha),
\beta \rangle$, whence $\langle \rho_Z(\alpha), \beta \rangle \ge
\langle \alpha^\vee, \beta \rangle$.

Conversely, let $(\mathfrak X, \mathcal F, \rho)$ be an Enriques's
$B^-$-system and let $Z$ be the smooth complete toric $T$-variety
associated with the pair $(\mathfrak X, \mathcal F)$. Our goal is to
extend the action of~$T$ on~$Z$ to an action of~$B^-$. We set
$$
\Pi_0 = \lbrace \alpha \in \Pi \mid \rho(\alpha) \ne 0 \rbrace.
$$
Property~(\ref{item_a}) of the map $\rho$ yields $\Pi_0 \subset
\mathfrak X$ and $\Pi_0 \subset R(\mathcal F)$. We recall
(see~\cite[Chapter~VI, Theorem~7(i)]{Ser}) that the Lie algebra
$\mathfrak u^-$ is the quotient of the free Lie algebra
$\widehat{\mathfrak u}^-$ generated by the set $\lbrace e_{-\alpha}
\mid \alpha \in \Pi \rbrace$ modulo the ideal $J$ generated by the
set
$$
\lbrace (\ad e_{-\alpha})^{1 - \langle \alpha^\vee, \beta \rangle}
e_{-\beta} \mid \alpha, \beta \in \Pi, \alpha \ne \beta \rbrace
\quad\qquad \text{ (the Serre relations).}
$$
Consider the Lie algebra homomorphism $\iota \colon
\widehat{\mathfrak u}^- \to \Der_g (\CR(Z))$ given by
$$
\iota(e_{-\alpha})
= \begin{cases} y_{-\alpha} & \text{if } \alpha \in \Pi_0;\\
0 & \text{if } \alpha \notin \Pi_0.\end{cases}
$$
Property~(\ref{item_b}) combined with
Lemma~\ref{lemma_commutators}(\ref{lemma_commutators_b}) implies
that the ideal $J$ is contained in $\Ker \iota$, which gives rise to
a homomorphism $\mathfrak u^- \to \Der_g (\CR(Z))$ and in turn to a
homomorphism $\widetilde \theta \colon U^- \to \Aut_g(\CR(Z))$. Then
the homomorphism $\theta = d \circ \widetilde \theta \colon U^- \to
\Aut Z$ defines an action of $U^-$ on~$Z$. Clearly,
$\theta(U_{-\alpha}) = d(Y_{-\alpha})$ for all $\alpha \in \Pi_0$
and $\theta(U_{-\alpha})$ is trivial for all $\alpha \in \Pi
\backslash \Pi_0$. Using Proposition~\ref{prop_normalized-alpha} and
the fact that the group $U^-$ is generated by all the subgroups
$U_{-\alpha}$ with $\alpha \in \Pi$, we deduce that the actions
on~$Z$ of $T$ and $U^-$ extend to an action of~$B^-$, so that $Z$
becomes a $B^-$-variety. The equalities $\mathfrak X_Z = \mathfrak
X$, $\mathcal F_Z = \mathcal F$, and $\rho_Z = \rho$ hold by
construction. The preceding argument also shows that an action of
$B^-$ on $Z$ extending the initial action of~$T$ is unique up to
conjugation by an element of~$T$, which completes the proof.
\end{proof}

The following two lemmas provide some properties of a smooth
complete spherical $B^-$-variety~$Z$. We retain all the notation
introduced above.

\begin{lemma} \label{lemma_B-stable_divisor}
Let $D \subset Z$ be a $T$-stable prime divisor and let $\varrho \in
\mathcal F^1$ be the corresponding element. The divisor $D$ is
$B^-$-stable if and only if $\langle \varrho, \alpha \rangle \le 0$
for all $\alpha \in \Pi$ with $\rho_Z(\alpha) \ne 0$.
\end{lemma}

\begin{proof}
The divisor $D$ is $B^-$-stable if and only if $D$ is
$U_{-\alpha}$-stable for all $\alpha \in \Pi$ with $\rho_Z(\alpha)
\ne\nobreak 0$. Since $\theta(U_{-\alpha}) = d(Y_{-\alpha})$ for all
such~$\alpha$, the assertion follows from
Lemma~\ref{lemma_Y-unstable_divisor}.
\end{proof}

\begin{lemma} \label{lemma_B-unstable_point}
Let $z \in Z$ be a $T$-fixed point and let $\mathcal C \in \mathcal
F$ be the corresponding maximal cone. The point $z$ is
$B^-$-unstable if and only if there exists a root $\alpha \in \Pi$
such that $\rho_Z(\alpha) \in \mathcal C^1$ and $\langle \varrho,
\alpha \rangle = 0$ for all $\varrho \in \mathcal C^1 \backslash
\lbrace \rho_Z(\alpha) \rbrace$.
\end{lemma}

\begin{proof}
Clearly, $z$ is $B^-$-unstable if and only if there exists a root
$\alpha \in \Pi$ such that $z$ is $U_{-\alpha}$-unstable. Since
$\theta(U_{-\alpha}) = d(Y_{-\alpha})$ for all $\alpha \in \Pi$ with
$\rho_Z(\alpha) \ne 0$, the assertion follows from
Lemma~\ref{lemma_Y-unstable_point}.
\end{proof}

\subsection{Classification of wonderful $B^-$-varieties}
\label{subsec_wond_B-var}

Wonderful $B^-$-varieties are classified by so-called admissible
maps; see Proposition~\ref{prop_wonderful_B-var} below.

\begin{definition}[{see~\cite[\S\,2]{Lu93}}]
A map $\eta \colon \Pi \times \Pi \to \{-3,-2,-1,0,1\}$ is said to
be \textit{admissible} if it satisfies the following five
conditions:

\begin{enumerate}[label=\textup{(AM\arabic*)},ref=\textup{AM\arabic*}]
\item \label{AM1}
$\eta(\alpha, \alpha) \in \{0,1\}$;

\item \label{AM2}
if $\eta(\alpha, \alpha) = 0$ then $\eta(\alpha, \beta) =
\eta(\beta, \alpha) = 0$ for every $\beta \in \Pi$;

\item \label{AM3}
if $\eta(\alpha, \beta) = 1$ then $\eta(\alpha, \gamma) =
\eta(\beta, \gamma)$ for every $\gamma \in \Pi$;

\item \label{AM4}
if $\eta(\alpha, \beta) < 0$ then $\eta(\beta, \alpha) = 0$;

\item \label{AM5}
$\eta(\alpha, \beta) \ge \langle \alpha^\vee, \beta \rangle$
whenever $\alpha \ne \beta$.
\end{enumerate}
\end{definition}

Let $\eta$ be an admissible map. Our immediate goal is to associate
an Enriques's $B^-$-system with~$\eta$. First, we set $\Pi_\eta =
\lbrace \alpha \in \Pi \mid \eta(\alpha, \alpha) = 1 \rbrace$ and
let $\mathfrak X_\eta$ denote the sublattice in $\mathfrak X(T)$
generated by~$\Pi_\eta$. Second, we set $Q_\eta = \Hom_\ZZ(\mathfrak
X_\eta, \QQ)$ and let $\lbrace \breve \alpha \mid \alpha \in
\Pi_\eta \rbrace$ be the basis of $Q_\eta$ dual to~$\Pi_\eta$. We
introduce the map $\rho_\eta \colon \Pi \to Q_\eta$ by the formula
\begin{equation} \label{eqn_rho_alpha}
\rho_\eta(\alpha) = \sum \limits_{\gamma \in \Pi_\eta} \eta(\alpha,
\gamma) \breve \gamma.
\end{equation}
Next, let $\widetilde \Pi_\eta$ denote the collection of subsets
$\Pi' \subset \Pi_\eta$ such that the restriction of $\rho_\eta$ to
$\Pi'$ is injective. For every $\Pi' \in \widetilde \Pi_\eta$, let
$\mathcal C_{\Pi'}$ be the cone generated by the set
$$
S(\Pi') = \lbrace \rho_\eta(\alpha) \mid \alpha \in \Pi' \rbrace
\cup \lbrace - \breve \alpha \mid \alpha \in \Pi_\eta \backslash
\Pi' \rbrace.
$$
At last, let $\mathcal F_\eta$ denote the set formed by all the
cones $\mathcal C_{\Pi'}$ ($\Pi' \in \widetilde \Pi_\eta$) and their
faces.

\begin{lemma} \label{lemma_properties_Pi'}
Suppose that $\Pi' \in \widetilde \Pi_\eta$ and $\alpha, \beta \in
\Pi'$ are different roots. Then:
\begin{enumerate}[label=\textup{(\alph*)},ref=\textup{\alph*}]
\item \label{lemma_properties_Pi'_a}
$\eta(\alpha, \beta) \le 0$;

\item \label{lemma_properties_Pi'_b}
if $(\alpha, \beta) = 0$ then $\eta(\alpha, \beta) = \eta(\beta,
\alpha) = 0$.
\end{enumerate}
\end{lemma}

\begin{proof}
To prove part~(\ref{lemma_properties_Pi'_a}), we note that
by~(\ref{AM3}) the condition $\eta(\alpha, \beta) = 1$ would imply
$\rho_\eta(\alpha) = \rho_\eta(\beta)$ contradicting the definition
of~$\widetilde \Pi_\eta$. Part~(\ref{lemma_properties_Pi'_b}) is a
direct consequence of~(\ref{lemma_properties_Pi'_a})
and~(\ref{AM5}).
\end{proof}

\begin{lemma} \label{lemma_cone_is_regular}
Let $\Pi' \in \widetilde \Pi_\eta$.
\begin{enumerate}[label=\textup{(\alph*)},ref=\textup{\alph*}]
\item \label{lemma_cone_is_regular_a}
The set $S(\Pi')$ is linearly independent. In particular, $S(\Pi') =
\mathcal C_{\Pi'}^1$ and the cone $\mathcal C_{\Pi'}$ is simplicial.

\item \label{lemma_cone_is_regular_b}
The cone $\mathcal C_{\Pi'}$ is regular.
\end{enumerate}
\end{lemma}

\begin{proof}
To prove (\ref{lemma_cone_is_regular_a}) it suffices to show that
the set
$$
\lbrace \sum \limits_{\gamma \in \Pi'} \eta(\alpha, \gamma) \breve
\gamma \mid \alpha \in \Pi'\rbrace
$$
is linearly independent. (Note that the sum is taken over the set
$\Pi'$ instead of~$\Pi_\eta$.) Let $\alpha' \in \Pi'$ be such that
the corresponding node of the Dynkin diagram of $\Pi'$ is incident
to at most one edge. By
Lemma~\ref{lemma_properties_Pi'}(\ref{lemma_properties_Pi'_b}) this
means that there is at most one root $\gamma' \in \Pi' \backslash
\lbrace \alpha' \rbrace$ with $\eta(\alpha', \gamma') \ne 0$ (which
implies $\eta(\alpha', \gamma') < 0$ by
Lemma~\ref{lemma_properties_Pi'}(\ref{lemma_properties_Pi'_a})).
Condition~(\ref{AM4}) yields that either $\eta(\alpha', \gamma) = 0$
for all $\gamma \in \Pi' \backslash \lbrace \alpha' \rbrace$ or
$\eta(\gamma, \alpha') = 0$ for all $\gamma \in \Pi' \backslash
\lbrace \alpha' \rbrace$. In any case, the problem reduces to the
linear independence of the set
$$
\lbrace \sum \limits_{\gamma \in \Pi' \backslash \lbrace \alpha'
\rbrace} \eta(\alpha, \gamma) \breve \gamma \mid \alpha \in \Pi'
\backslash \lbrace \alpha' \rbrace \rbrace.
$$
The proof of~(\ref{lemma_cone_is_regular_a}) is completed by
induction. From the above argument it also follows that the
determinant of the transformation matrix from $\Pi'$ to $S(\Pi')$
equals $\pm 1$, which implies~(\ref{lemma_cone_is_regular_b}).
\end{proof}

\begin{lemma} \label{lemma_eta_nonnegative}
For every subset $\Pi' \subset \Pi_\eta$, there is a root $\beta \in
\Pi'$ such that $\eta(\beta, \gamma) \ge 0$ for all $\gamma \in
\Pi'$.
\end{lemma}

\begin{proof}
Assume the converse. Then there exists an infinite sequence
$\beta_1, \beta_2, \beta_3, \ldots$ of roots in $\Pi'$ with
$\eta(\beta_i, \beta_{i+1}) < 0$ for all~$i$. Applying (\ref{AM5})
we obtain $(\beta_i, \beta_{i+1}) < 0$. Next, condition~(\ref{AM4})
implies $\eta(\beta_{i+1}, \beta_i) = 0$, whence $\beta_{i+2} \ne
\beta_i$ for all~$i$. Consequently, the Dynkin diagram of $\Pi'$
contains a cycle, a contradiction.
\end{proof}

According to Lemma~\ref{lemma_eta_nonnegative}, we choose
successively subsets $\Pi_1, \ldots, \Pi_s \subset \Pi_\eta$, where
$s = s(\eta) = |\rho_\eta(\Pi_\eta)|$, in the following way. At
first, we take a root $\beta_1 \in \Pi_\eta$ such that
${\eta(\beta_1, \gamma) \ge 0}$ for all $\gamma \in \Pi_\eta$ and
set $\Pi_1 = \Pi_1(\eta) = \lbrace \gamma \in \Pi_\eta \mid
\eta(\beta_1, \gamma) = 1 \rbrace$. Next, assume that $i \in \lbrace
2, \ldots, s \rbrace$ and the sets $\Pi_1, \ldots, \Pi_{i-1}$ have
been chosen. We put $\overline \Pi_i = \Pi_\eta \backslash (\Pi_1
\cup \ldots \cup \Pi_{i-1})$ and choose a root $\beta_i \in
\overline \Pi_i$ such that $\eta(\beta_i, \gamma) \ge 0$ for all
$\gamma \in \overline \Pi_i$. We set
$$
\Pi_i = \Pi_i(\eta) = \lbrace \gamma \in \overline \Pi_i \mid
\eta(\beta_i, \gamma) = 1 \rbrace.
$$
Note that $\beta_i \in \Pi_i$. By~(\ref{AM3}), for every $i = 1,
\ldots, s$ the set $\rho_\eta(\Pi_i)$ contains exactly one element;
we denote it by $\varrho_i = \varrho_i(\eta)$.

\begin{remark} \label{rem_number_of_max_cones}
For every $\Pi' \in \widetilde \Pi_\eta$ and every $i = 1, \ldots,
s$, there is exactly one element $\sigma_i$ in the set $ \lbrace -
\breve \alpha \mid \alpha \in \Pi_i \rbrace \cup \lbrace \varrho_i
\rbrace $ that is not contained in~$S(\Pi')$. Moreover, all possible
$s$-tuples $(\sigma_1, \ldots, \sigma_s)$ are obtained in this way,
so that
$$
|\widetilde \Pi_\eta| = (|\Pi_1| + 1) \cdot \ldots \cdot (|\Pi_s| +
1).
$$
\end{remark}

For every $i = 1, \ldots, s$, let $Q_i$ be the subspace of $Q_\eta$
generated by the set
$$
\lbrace \breve \beta \mid \beta \in \Pi_1 \cup \ldots \cup \Pi_s
\rbrace.
$$
Set also $Q_0 = \lbrace 0 \rbrace$. Then by construction one has
$\varrho_i \in Q_i$ for every $i = 1, \ldots, s$.

For a fixed $i \in \lbrace 1, \ldots, s \rbrace$, let $q \mapsto
q^\diamond$ be the natural epimorphism from $Q_i$ to $Q_i /
Q_{i-1}$. Then $\varrho_i^\diamond = \sum \limits_{\gamma \in \Pi_i}
\breve \gamma^\diamond$. A key observation is that the cones
generated by all proper subsets of the set $\lbrace -\breve
\gamma^\diamond \mid \gamma \in \Pi_i \rbrace \cup \lbrace
\varrho_i^\diamond \rbrace$ form a complete fan in $Q_i / Q_{i-1}$;
we denote this fan by $\mathcal F_i$.

\begin{lemma} \label{lemma_complete_fan}
The set $\mathcal F_\eta$ is a regular complete fan in $\mathfrak
X_\eta$ and
$$
\mathcal F_\eta^1 = \rho_\eta(\Pi_\eta) \cup \lbrace - \breve \alpha
\mid \alpha \in \Pi_\eta\rbrace.
$$
\end{lemma}

\begin{proof}
The second assertion becomes obvious as soon as the first one has
been proved.

We first show that $\mathcal F_\eta$ is a fan. To this end, it
suffices to check that the intersection of any two cones in
$\mathcal F_\eta$ is a face of each. Let $\mathcal C', C'' \in
\mathcal F_\eta$. Choose subsets $\Pi', \Pi'' \in \widetilde
\Pi_\eta$ such that $\mathcal C'$ (resp.~$\mathcal C''$) is a face
of the cone $\mathcal C_{\Pi'}$ (resp.~$\mathcal C_{\Pi''}$). For
every $x \in \mathcal C' \cap \mathcal C''$, one has
$$
x = \sum \limits_{\alpha \in \Pi'} a_\alpha \rho_\eta(\alpha) - \sum
\limits_{\alpha \in \Pi_\eta \backslash \Pi'} b_\alpha \breve \alpha
= \sum \limits_{\alpha \in \Pi''} k_\alpha \rho_\eta(\alpha) - \sum
\limits_{\alpha \in \Pi_\eta \backslash \Pi''} l_\alpha \breve
\alpha,
$$
where $a_\alpha, b_\alpha, k_\alpha, l_\alpha \in \QQ^+$. We note
that $a_\alpha = 0$ for all $\alpha \in \Pi'$ with
$\rho_\eta(\alpha) \notin \mathcal C'$ and $b_\alpha = 0$ for all
$\alpha \in \Pi_\eta \backslash \Pi'$ with $- \breve \alpha \notin
\mathcal C'$; similarly, $k_\alpha = 0$ for all $\alpha \in \Pi''$
with $\rho_\eta(\alpha) \notin \mathcal C''$ and $l_\alpha = 0$ for
all $\alpha \in \Pi_\eta \backslash \Pi''$ with $- \breve \alpha
\notin \mathcal C''$.

We rewrite the two expressions for $x$ in a slightly different form
as follows:
\begin{equation} \label{eqn_cones_intersection}
x = \sum \limits_{i = 1}^s a_i \varrho_i - \sum \limits_{\alpha \in
\Pi_\eta} b_\alpha \breve \alpha = \sum \limits_{i = 1}^s k_i
\varrho_i - \sum \limits_{\alpha \in \Pi_\eta} l_\alpha \breve
\alpha.
\end{equation}
Here the coefficients $a_i, b_\alpha, k_i, l_\alpha \in \QQ^+$ are
determined by the following rule:

\begin{itemize}
\item
if the set $\Pi_i \cap \Pi'$ is empty, then $a_i = 0$; otherwise
this set contains exactly one root $\gamma$, in which situation $a_i
= a_\gamma$;

\item
the coefficients $b_\alpha$ are already defined for $\alpha \in
\Pi_\eta \backslash \Pi'$, all the others are zero;

\item
if the set $\Pi_i \cap \Pi''$ is empty, then $k_i = 0$; otherwise
this set contains exactly one root $\gamma$, in which situation $k_i
= k_\gamma$;

\item
the coefficients $l_\alpha$ are already defined for $\alpha \in
\Pi_\eta \backslash \Pi''$, all the others are zero.
\end{itemize}

Consider the images in $Q_s / Q_{s-1}$ of all parts of
equality~(\ref{eqn_cones_intersection}):
$$
x^\diamond = a_s \varrho_s^\diamond - \sum \limits_{\alpha \in
\Pi_s} b_\alpha \breve \alpha^\diamond = k_s \varrho_s^\diamond -
\sum \limits_{\alpha \in \Pi_s} l_\alpha \breve \alpha^\diamond.
$$
Clearly, for each of the cones $\mathcal C', \mathcal C''$ its image
in $Q_s / Q_{s-1}$ is a cone of the fan $\mathcal F_s$, and so the
intersection of these images is again a cone in~$\mathcal F_s$. This
immediately implies $a_s = k_s$ and $b_\alpha = l_\alpha$ for all
$\alpha \in \Pi_s$. Therefore,
$$
\sum \limits_{i = 1}^{s-1} a_i \varrho_i - \sum \limits_{\alpha \in
\Pi_\eta \backslash \Pi_s} b_\alpha \breve \alpha = \sum \limits_{i
= 1}^{s-1} k_i \varrho_i - \sum \limits_{\alpha \in \Pi_\eta
\backslash \Pi_s} l_\alpha \breve \alpha,
$$
with both sides lying in the space~$Q_{s-1}$. Applying induction, we
obtain $a_i = k_i$ for all $i = 1, \ldots, s$ and $b_\alpha =
l_\alpha$ for all $\alpha \in \Pi_\eta$. In particular, $a_i = 0$
for all $i = 1, \ldots, s$ with $\varrho_i \notin \mathcal C''$ and
$b_\alpha = 0$ for all $\alpha \in \Pi_\eta$ with $- \breve \alpha
\notin \mathcal C''$; similarly, $k_i = 0$ for all $i = 1, \ldots,
s$ with $\varrho_i \notin \mathcal C'$ and $l_\alpha = 0$ for all
$\alpha \in \Pi_\eta$ with $- \breve \alpha \notin \mathcal C'$.
Hence $\mathcal C' \cap \mathcal C''$ is a common face of $\mathcal
C', \mathcal C''$.

Now let us show that the fan $\mathcal F_\eta$ is complete. Let $x
\in Q_\eta$ be an arbitrary element. Clearly, the element
$x^\diamond \in Q_s / Q_{s-1}$ turns out to lie in a cone of the fan
$\mathcal F_s$, and so there is a unique expression $x^\diamond =
a_s \varrho_s^\diamond - \sum \limits_{\alpha \in \Pi_s} b_\alpha
\breve \alpha^\diamond$ with $a_s, b_\alpha \in \QQ^+$ and at least
one of the coefficients $a_s, b_\alpha$ being zero. Next, the
element $x - a_s \varrho_s + \sum \limits_{\alpha \in \Pi_s}
b_\alpha \breve \alpha$ lies in the space~$Q_{s-1}$, and by
induction we get an expression $x = \sum \limits_{i = 1}^s a_i
\varrho_i - \sum \limits_{\alpha \in \Pi_\eta} b_\alpha \breve
\alpha$, where $a_i, b_\alpha \in \QQ^+$ and for every $i = 1,
\ldots, s$ at least one of the coefficients $a_i, b_\alpha$ ($\alpha
\in \Pi_i$) is zero. For each $i = 1, \ldots, s$, we look at the
coefficient~$a_i$. If it is nonzero, then we choose any root $\alpha
\in \Pi_i$ with $b_\alpha = 0$ and set $\Pi_i(x) = \lbrace \alpha
\rbrace$. Otherwise we set $\Pi_i(x) = \varnothing$. At last, we set
$\Pi(x) = \Pi_1(x) \cup \ldots \cup \Pi_s(x)$. By construction,
$\Pi(x) \in \widetilde \Pi_\eta$ and $x \in \mathcal C_{\Pi(x)}$.

The proof is completed by observing that the fan $\mathcal F_\eta$
is regular by
Lemma~\ref{lemma_cone_is_regular}(\ref{lemma_cone_is_regular_b}).
\end{proof}

\begin{proposition}[{see~\cite[\S\,2]{Lu93}}]
The triple $(\mathfrak X_\eta, \mathcal F_\eta, \rho_\eta)$ is an
Enriques's $B^-$-system.
\end{proposition}

\begin{proof}
By construction and Lemma~\ref{lemma_complete_fan}, it suffices to
check conditions (\ref{item_a}),~(\ref{item_b}) of
Definition~\ref{def_B-system}. But these follow directly from
properties (\ref{AM1})--(\ref{AM5}) since by~(\ref{eqn_rho_alpha})
one has $\langle \rho_\eta(\alpha), \beta \rangle = \eta(\alpha,
\beta)$ for every $\alpha \in \Pi$ and $\beta \in \Pi \cap \mathfrak
X_Z$.
\end{proof}

For an admissible map~$\eta$, we denote by $Z_\eta$ the smooth
complete spherical $B^-$-variety corresponding to the Enriques's
$B^-$-system $(\mathfrak X_\eta, \mathcal F_\eta, \rho_\eta)$.

Given a wonderful $B^-$-variety~$Z$, we let $(\mathfrak X_Z,
\mathcal F_Z, \rho_Z)$ be the corresponding Enriques's $B^-$-system
and introduce the map $\eta_Z \colon \Pi \times \Pi \to \ZZ$ as
follows:
\begin{equation} \label{eqn_admissible_map_via_EBs}
\eta_Z(\alpha, \beta) =
\begin{cases} \langle \rho_Z(\alpha), \beta \rangle & \text{if }
\beta \in \mathfrak X_Z; \\
0 & otherwise.
\end{cases}
\end{equation}

\begin{proposition}[{see~\cite[Proposition~1]{Lu93}}]
\label{prop_wonderful_B-var} The map $Z \mapsto \eta_Z$ is a
bijection between wonderful $B^-$-varieties \textup(considered up to
$B^-$-equivariant isomorphism\textup) and admissible maps.
\end{proposition}

\begin{proof}
Let $Z$ be a wonderful $B^-$-variety. We first show that the lattice
$\mathfrak X_Z$ is generated by the set $\Pi_Z = \lbrace \alpha \in
\Pi \mid \rho_Z(\alpha) \ne 0 \rbrace$. Observe that $\Pi_Z \subset
\mathfrak X_Z$ by Definition~\ref{def_B-system} and consider the
cone
$$
\widetilde{\mathcal C}_0 = \lbrace q \in Q_Z \mid \langle q, \alpha
\rangle \le 0 \text{ for all } \alpha \in \Pi_Z \rbrace.
$$
By~(\ref{def_wond_B-var2}) there is a unique $B^-$-fixed point $z_0$
in~$Z$. Let $\mathcal C_0 \subset Q_Z$ be the maximal cone
in~$\mathcal F_Z$ corresponding to~$z_0$. By~(\ref{def_wond_B-var3})
every $T$-stable prime divisor in $Z$ containing $z_0$ is
$B^-$-stable. Applying Lemma~\ref{lemma_B-stable_divisor} we obtain
$\langle \varrho, \alpha \rangle \le 0$ for all $\varrho \in
\mathcal C_0^1$ and $\alpha \in \Pi_Z$, which implies $\mathcal C_0
\subset \widetilde{\mathcal C}_0$. On the other hand, let $\mathcal
C \ne \mathcal C_0$ be an arbitrary maximal cone of the fan
$\mathcal F_Z$ and let $z \in Z$ be the corresponding $T$-stable
point. Since $Z$ has only one $B^-$-stable point, $z$ is
$B^-$-unstable. By Lemma~\ref{lemma_B-unstable_point} the latter
implies that there is a root $\alpha \in \Pi_Z$ with $\langle c,
\alpha \rangle \ge 0$ for all $c \in \mathcal C$. Therefore
$\mathcal C \cap (\widetilde{\mathcal C}_0)^\circ = \varnothing$.
Since the fan $\mathcal F_Z$ is complete, we obtain
$(\widetilde{\mathcal C}_0)^\circ \subset \mathcal C_0$, hence
$\widetilde{\mathcal C}_0 = \mathcal C_0$. As the cone $\mathcal
C_0$ is simplicial, we have $\rk \mathfrak X_Z = |\Pi_Z|$. Since
$\langle \rho_Z(\alpha), \alpha \rangle = 1$ for all $\alpha \in
\Pi_0$, it follows that each element of $\Pi_0$ is primitive
in~$\mathfrak X$. Combining the latter with regularity of $\mathcal
C_0$, we obtain $\mathfrak X_Z = \ZZ \Pi_Z$.

The result of the previous paragraph along
with~(\ref{eqn_admissible_map_via_EBs}) implies that for a root
$\alpha \in \Pi$ the following three conditions are equivalent:
\begin{itemize}
\item
$\alpha \in \mathfrak X_Z$;

\item
$\rho_Z(\alpha) \ne 0$;

\item
$\eta_Z(\alpha, \alpha) \ne 0$.
\end{itemize}
Taking this into account and using Definition~\ref{def_B-system}
together with Lemma~\ref{lemma_EBS_property}, one easily checks that
$\eta_Z$ satisfies all the conditions (\ref{AM1})--(\ref{AM5}), so
that $\eta_Z$ is an admissible map.

Now let us prove that $Z = Z_{\eta_Z}$. It follows from the above
considerations that $\Pi_Z = \Pi_{\eta_Z}$ and $\mathfrak X_Z =
\mathfrak X_{\eta_Z}$. Next, for every $\alpha \in \Pi_{\eta_Z}$ we
have
$$
\rho_Z(\alpha) = \sum \limits_{\gamma \in \Pi_{\eta_Z}} \langle
\rho_Z(\alpha), \gamma \rangle \breve \gamma = \sum \limits_{\gamma
\in \Pi_{\eta_Z}} \eta(\alpha, \gamma) \breve \gamma =
\rho_{\eta_Z}(\alpha),
$$
hence $\rho_Z = \rho_{\eta_Z}$ and $\mathcal F_Z^1 \supset \mathcal
F_{\eta_Z}^1$. On the other hand, if $\varrho \in \mathcal F_Z^1
\backslash \mathcal C_0^1$, then the corresponding $T$-stable prime
divisor in $Z$ does not contain~$z_0$, and so it is not $B^-$-stable
and is moved by the subgroup $U_{-\alpha}$ for some $\alpha \in
\Pi_Z$. In view of Lemma~\ref{lemma_Y-unstable_divisor} the latter
means $\varrho = \rho_Z(\alpha) = \rho_{\eta_Z}(\alpha) \in \mathcal
F_{\eta_Z}^1$, hence $\mathcal F_Z^1 = \mathcal F_{\eta_Z}^1$. At
last, let us show that $\mathcal F_Z = \mathcal F_{\eta_Z}$. As in
the paragraph following Lemma~\ref{lemma_eta_nonnegative}, we
introduce the number $s = s(\eta_Z)$ and put $\Pi_i =
\Pi_i(\eta_Z)$, $\varrho_i = \varrho_i(\eta_Z)$ for every $i = 1,
\ldots, s$. Now let $\mathcal C$ be a maximal cone of the
fan~$\mathcal F_Z$ and assume that $\mathcal C \ne \mathcal
C_{\Pi'}$ for every $\Pi' \in \widetilde \Pi_{\eta_Z}$. Taking into
account Remark~\ref{rem_number_of_max_cones}, we conclude that there
is $i \in \lbrace 1, \ldots, s \rbrace$ such that $\mathcal C^1
\supset \lbrace - \breve \alpha \mid \alpha \in \Pi_i \rbrace \cup
\lbrace \varrho_i \rbrace$. Since $\langle \varrho_i, \alpha \rangle
= 1$ for all $\alpha \in \Pi_i$ and $\langle \varrho_i, \alpha
\rangle \le 0$ for all $\alpha \notin \Pi_i$, the element $\varrho_i
- \sum \limits_{\alpha \in \Pi_i} \breve \alpha$ is contained in
$\mathcal C_0 \cap \mathcal C$, the latter being a common face of
$\mathcal C_0$ and~$\mathcal C$. It follows that $\varrho_i \in
\mathcal C_0^1$, which is not the case. Hence $\mathcal F_Z =
\mathcal F_{\eta_Z}$.

Finally, let $\eta$ be an admissible map and let $Z_\eta$ be the
corresponding smooth complete spherical $B^-$-variety. Let $z_0 \in
Z_\eta$ be the point corresponding to the cone
\begin{equation} \label{eqn_cone_C_empty}
\mathcal C_\varnothing = \lbrace q \in Q_\eta \mid \langle q, \alpha
\rangle \le 0 \text{ for all } \alpha \in \Pi_\eta \rbrace.
\end{equation}
Since $\mathcal C_\varnothing^1 \cap \rho_\eta(\Pi_\eta) =
\varnothing$, the point $z_0$ is $B^-$-fixed (see
Lemma~\ref{lemma_B-unstable_point}) and every $T$-stable irreducible
closed subvariety of $Z$ containing $z_0$ is $B^-$-stable (see
Lemma~\ref{lemma_B-stable_divisor}). Let $z \in Z_\eta$ be any
$T$-fixed point different from $z_0$ and let $\mathcal C \in
\mathcal F_\eta$ be the corresponding maximal cone. As $\mathcal C
\ne \mathcal C_\varnothing$, there exists a root $\alpha \in
\Pi_\eta$ such that $\langle \varrho, \alpha \rangle \ge 0$ for all
$\varrho \in \mathcal C^1$. The latter implies $- \breve \alpha
\notin \mathcal C^1$, hence $\rho_\eta(\alpha) \in \mathcal C^1$
(see Remark~\ref{rem_number_of_max_cones}) and $\langle \varrho,
\alpha \rangle = 0$ for all $\varrho \in \mathcal C^1 \backslash
\lbrace \rho_\eta(\alpha) \rbrace$. By
Lemma~\ref{lemma_B-unstable_point} the point $z$ is $B^-$-unstable.
Thus $Z_\eta$ is a wonderful $B^-$-variety and $\eta_{Z_\eta} =
\eta$.
\end{proof}

\subsection{Relationship with Luna's general classification}
\label{subsec_1993&general}

For every admissible map~$\eta$, we choose $H_\eta \subset B^-$ to
be the stabilizer of a point of the open $B^-$-orbit in the
wonderful $B^-$-variety~$Z_\eta$.
Theorem~\ref{thm_bij_wond_B-var_SSWS} and
Proposition~\ref{prop_wonderful_B-var} imply the following result.

\begin{theorem} \label{thm_AM_SSWS}
The map $\eta \mapsto H_\eta$ induces a bijection between admissible
maps and conjugacy classes in~$B^-$ of strongly solvable wonderful
subgroups of~$G$ contained in~$B^-$.
\end{theorem}

Let $\eta$ be an admissible map and let $Z_\eta$ be the
corresponding wonderful $B^-$-variety.
Proposition~\ref{prop_wonderful_B&G}(\ref{prop_wonderful_B&G_b})
yields that $X_\eta = G *_{B^-} Z_\eta$ is a wonderful $G$-variety
whose open $G$-orbit is isomorphic to $G / H_\eta$. The main goal of
this subsection is to compute the spherical system of~$X_\eta$ in
terms of~$\eta$.

Let $\phi \colon X_\eta \to G / B^-$ be the natural $G$-equivariant
morphism. We identify $Z_\eta$ with $\phi^{-1}(o)$. Let $z_0$ denote
the unique $B^-$-fixed point in~$Z_\eta$. We recall the notation
$\Pi_\eta = \lbrace \alpha \in\nobreak \Pi \mid \eta(\alpha, \alpha)
= 1 \rbrace$ and the map $\rho_\eta$ given by~(\ref{eqn_rho_alpha}).
Let $\mathscr S_\eta = (\Pi^p_\eta, \Sigma_\eta, \mathcal D^a_\eta)$
be the spherical system of $X_\eta$ and denote by $\varkappa_\eta$
the corresponding map $\mathcal D^a_\eta \to
\Hom_\ZZ(\ZZ\Sigma_\eta, \ZZ)$.

\begin{proposition} \label{prop_invariants_via_AM}
The spherical system $\mathscr S_\eta$ is determined as follows:
\begin{enumerate}[label=\textup{(\alph*)},ref=\textup{\alph*}]
\item \label{prop_invariants_via_AM_a}
$\Pi^p_\eta = \varnothing$.

\item \label{prop_invariants_via_AM_b}
$\Sigma_\eta = \Pi_\eta$.

\item \label{prop_invariants_via_AM_c}
The set $\mathcal D^a_\eta$ is in bijection with the set $\Pi_\eta
\cup \rho_\eta(\Pi_\eta)$. For every $\alpha \in \Pi_\eta$, let
$D_\alpha^-$ \textup(resp.~$D_\alpha^+$\textup) be the color
corresponding to $\alpha$ \textup(resp. $\rho_\eta(\alpha)$\textup).
Then $\mathcal D(\alpha) = \lbrace D_\alpha^-, D_\alpha^+ \rbrace$
for all $\alpha \in \Pi_\eta$, and one has $\langle
\varkappa_\eta(D_\alpha^+), \beta \rangle = \eta(\alpha, \beta)$,
$\langle \varkappa_\eta(D_\alpha^-), \beta \rangle = \langle
\alpha^\vee, \beta \rangle - \eta(\alpha, \beta)$ for all $\beta \in
\Pi_\eta$.
\end{enumerate}
\end{proposition}

\begin{proof}
(\ref{prop_invariants_via_AM_a}) The open $B$-orbit in $X_\eta$ maps
onto the open $B$-orbit in $G / B^-$, whose stabilizer is well known
to be~$B$. Therefore $\Pi^p_\eta = \varnothing$.

(\ref{prop_invariants_via_AM_b}) Evidently, the closed $G$-orbit in
$X_\eta$ is just $Gz_0$. Let us find the canonical $B$-chart
$(X_\eta)_B$ (see~\S\,\ref{subsec_standard_completions}). By
definition, $(X_\eta)_B$ consists of $B$-orbits in $X_\eta$ whose
closure contains~$Gz_0$. Since $Gz_0$ maps (isomorphically) onto $G
/ B^-$, every such $B$-orbit maps necessarily onto $B o \simeq B /
T$, which is the unique open $B$-orbit in $G / B^-$. Therefore
$(X_\eta)_B \subset \phi^{-1}(B o)$. As we have already seen in the
proof of Proposition~\ref{prop_wonderful_B&G}, there are
$B$-equivariant isomorphisms $\phi^{-1}(B o) \simeq B *_T Z_\eta
\simeq U \times Z_\eta$, where the $B$-action on the latter variety
is given by formula~(\ref{eqn_action_on_X0}). We note that every
$T$-semi-invariant rational function on $Z_\eta$ naturally extends
to a $B$-semi-invariant rational function on $U \times Z_\eta$ of
the same weight, and every $B$-semi-invariant rational function on
$U \times Z_\eta$ is obtained in this way.

Since the set $B z_0$ is open in $Gz_0$, a $B$-orbit $O \subset B
*_T Z_\eta$ is contained in $(X_\eta)_B$ if and only if the
$T$-orbit $O \cap Z_\eta \subset Z_\eta$ contains $z_0$ in its
closure. It follows that $(X_\eta)_B \simeq B *_T Z_0 \simeq U
\times \nobreak Z_0$, where $Z_0 \subset Z_\eta$ is the $B$-stable
affine open subset corresponding to the cone $\mathcal C_\varnothing
\in\nobreak \mathcal F_\eta$ given by~(\ref{eqn_cone_C_empty}). The
explicit description of the cone $\mathcal C_\varnothing$ yields
that the weight semigroup of $T$-semi-invariant functions in
$\CC[Z_0]$, as well as the weight semigroup of $B$-semi-invariant
functions in $\CC[(X_\eta)_B]$, is generated by the set $-
\Pi_\eta$. Applying Proposition~\ref{prop_B-chart_spherical_roots},
we obtain $\Sigma_\eta = \Pi_\eta$.

(\ref{prop_invariants_via_AM_c}) We first note that the open
$G$-orbit $X^0_\eta \subset X_\eta$ is isomorphic to the homogeneous
bundle $G *_{B^-} Z^0_\eta$, where $Z^0_\eta$ is the open
$B^-$-orbit in $Z_\eta$. We also note that the fan $\mathcal
F^0_\eta$ corresponding to $Z^0_\eta$ as a toric $T$-variety
consists of the cones $\mathcal C \in \mathcal F_\eta$ such that
$\mathcal C \cap \mathcal C_\varnothing = \lbrace 0 \rbrace$.

Let $\mathcal D' \subset \mathcal D$ be the distinguished subset of
colors corresponding to $\phi$ (see
Proposition~\ref{prop_wonderful_morphisms}) and let $D \in \mathcal
D'$. Since $\phi(D)$ is dense in $G / B^-$, it follows that $D \cap
\phi^{-1}(B o)$ is a $B$-stable prime divisor in $\phi^{-1}(B o)
\simeq B *_T Z^0_\eta \simeq U \times Z^0_\eta$, where the
$B$-action on the latter variety is given
by~(\ref{eqn_action_on_X0}). It is easy to deduce that the set
$\mathcal D'$ is in bijection with the set of $T$-stable prime
divisors in $Z^0_\eta$ or, equivalently, with the set $(\mathcal
F^0_\eta)^1$ or, equivalently, with the set $\rho_\eta(\Pi_\eta)$.
It follows that the set $\mathcal D'$ is formed by the colors
$D_\alpha^+$ ($\alpha \in \Pi_\eta$) indicated in the hypothesis.
Next, since $\langle \varkappa_\eta(D), \alpha \rangle = 1$ for all
$\alpha \in \Sigma \cap \Pi$ and $D \in \mathcal D(\alpha)$, it is
easily verified that
$$
\mathcal D(\alpha) \cap \lbrace D_\beta^+ \mid \beta \in \Pi_\eta
\rbrace = \lbrace D_\alpha^+ \rbrace
$$
for every $\alpha \in \Pi_\eta$. In view of
part~(\ref{prop_invariants_via_AM_b}) and
Proposition~\ref{prop_alternative}, for every $\alpha \in \Pi_\eta$
the set $\mathcal D(\alpha) \backslash \lbrace D_\alpha^+ \rbrace$
contains the divisor $D_\alpha^-$ indicated in the hypothesis. It
remains to notice that the axioms of an admissible map imply
$\varkappa_\eta(D_\alpha^-) \ne \varkappa_\eta(D_\beta^\pm)$ for any
$\alpha, \beta \in \Pi_\eta$ with $\alpha \ne \beta$.
\end{proof}

\begin{remark} \label{rem_cone_of_B/H}
In fact, the fan $\mathcal F^0_\eta$ considered in the proof of
part~(\ref{prop_invariants_via_AM_c}) contains a unique maximal
cone, which is generated by the set $(\mathcal F^0_\eta)^1 = \lbrace
\rho_\eta(\alpha) \mid \alpha \in \Pi_\eta \rbrace$.
\end{remark}

\begin{remark}
In his preprint~\cite{Lu93} Luna also proved that, for two
admissible maps $\eta, \eta'$, the corresponding strongly solvable
wonderful subgroups $H_\eta, H_{\eta'}$ are conjugate in $G$ if and
only if $\Pi_\eta = \Pi_{\eta'}$ and there is a bijection $i \colon
\mathcal D^a_\eta \to \mathcal D^a_{\eta'}$ such that
$\varkappa_\eta = \varkappa_{\eta'} \circ i$. In other words,
$H_\eta, H_{\eta'}$ are conjugate in $G$ if and only if $\mathscr
S_{\eta} = \mathscr S_{\eta'}$, which is a particular case of
Theorem~\ref{thm_classification}.
\end{remark}

\begin{remark}
Proposition~\ref{prop_invariants_via_AM} provides an independent
proof of Corollary~\ref{crl_strongly_solvable_SS}.
\end{remark}

Let $\mathscr S = (\Pi^p, \Sigma, \mathcal D^a)$ be a strongly
solvable spherical system. Recall that $\Pi^p = \varnothing$ and
$\Sigma \subset \Pi$, see Corollary~\ref{crl_strongly_solvable_SS}
or Proposition~\ref{prop_invariants_via_AM}. Let $H \subset B^-$ be
a strongly solvable wonderful subgroup corresponding to~$\mathscr S$
and let $X$ be the wonderful embedding of~$G/H$. We fix a subset
$\mathcal D' \subset \mathcal D^a$ satisfying the conditions of
Proposition~\ref{prop_strongly_solvable}. This subset gives rise to
a $G$-equivariant morphism $\phi \colon X \to G / B^-$, which
provides a wonderful $B^-$-variety $Z = \phi^{-1}(o)$ (see
Proposition~\ref{prop_wonderful_B&G}). Our last goal in this
subsection is to find the admissible map $\eta$ corresponding
to~$Z$.

Let $\mathcal D$ be the set of colors of~$X$. Recall from the proof
of Proposition~\ref{prop_strongly_solvable} that $|(\mathcal D
\backslash \mathcal D') \cap\nobreak \mathcal D(\alpha)| = 1$ for
every $\alpha \in \Pi$. Then for every $\alpha \in \Sigma$ the set
$$
\mathcal D' \cap \mathcal D(\alpha) = \lbrace D \in \mathcal D' \mid
\langle \varkappa(D), \alpha \rangle = 1 \rbrace
$$
contains exactly one element, we denote it by $D_\alpha^+$. The
following proposition is a direct consequence of
Proposition~\ref{prop_invariants_via_AM}(\ref{prop_invariants_via_AM_c}).

\begin{proposition} \label{prop_AM_via_SSSS}
The admissible map $\eta$ is determined as follows:
\begin{equation} \label{eqn_AM_via_general}
\eta(\alpha, \beta) =
\begin{cases}
\langle \varkappa(D_\alpha^+), \beta \rangle & \text{if }
\alpha, \beta \in \Sigma; \\
0 & \text{otherwise}.
\end{cases}
\end{equation}
\end{proposition}

\begin{remark} \label{rem_AM_axioms_via_gen}
Using the axioms of spherical systems and
Proposition~\ref{prop_strongly_solvable}, one can easily prove that
the map~$\eta$ defined by~(\ref{eqn_AM_via_general}) is indeed an
admissible map; see~\cite[Example~30.23]{Tim}.
\end{remark}

\section{The explicit classification of strongly solvable spherical
subgroups} \label{sect_explicit_classification}

In this section we present the explicit classification of strongly
solvable spherical subgroups and its relationships with Luna's 1993
classification and Luna's general classification. The explicit
classification was obtained in the paper~\cite{Avd_solv} in the case
where $G$ is semisimple and all subgroups under consideration are
connected. However all arguments in~\cite{Avd_solv} hold for any
connected reductive group~$G$ and replacing connected solvable
subgroups by strongly solvable ones causes only minor changes in the
formulation of some statements. All statements in this section taken
from~\cite{Avd_solv} are valid for arbitrary (not necessarily
connected) strongly solvable subgroups of~$G$.

\subsection{Description of the approach}
\label{subsec_solv_description}

Let $H \subset B^-$ be a strongly solvable subgroup of~$G$ and let
$N \subset U^-$ be the unipotent radical of~$H$. We say that $H$ is
\textit{standardly embedded in~$B^-$} (with respect to~$T$) if $S =
H \cap T$ is a Levi subgroup of~$H$, so that $H = S \rightthreetimes
N$. General results on Levi decompositions (see, for instance,
\cite[\S\,6.4]{OV}) yield that every subgroup $H \subset B^-$ is
conjugate in $B^-$ to a subgroup standardly embedded in~$B^-$.

In what follows, we assume that $H$ is standardly embedded in~$B^-$
and keep the decomposition $H = S \rightthreetimes N$. Let $\tau =
\tau_S \colon \mathfrak X(T) \to \mathfrak X(S)$ be the character
restriction map. Consider the natural action of~$S$ on $\mathfrak
u^-$. This action yields a decomposition $\mathfrak u^- = \bigoplus
\limits_{\varphi \in \tau(\Delta^+)} \mathfrak u^-_{-\varphi}$,
where $\mathfrak u^-_{-\varphi}$ is the weight subspace of weight
$-\varphi$ with respect to~$S$. For every $\varphi \in
\tau(\Delta^+)$, we put $\mathfrak n_{-\varphi} = \mathfrak
u^-_{-\varphi} \cap \mathfrak n$ and denote by $c_\varphi$ the
codimension of $\mathfrak n_{-\varphi}$ in~$\mathfrak
u^-_{-\varphi}$.

\begin{proposition}[{\cite[Theorem~1]{Avd_solv}}]
\label{prop_sphericity_criterion} In the above notation, the
following conditions are equivalent:
\begin{enumerate}[label=\textup{(\arabic*)},ref=\textup{\arabic*}]
\item $H$ is spherical in~$G$;

\item $c_\varphi \le 1$ for every $\varphi \in \Phi$, and all weights
$\varphi$ with $c_\varphi = 1$ are linearly independent
in~$\mathfrak X(S)$.
\end{enumerate}
\end{proposition}

Until the end of this subsection we assume in addition that $H$ is
spherical.

According to Proposition~\ref{prop_sphericity_criterion}, we
introduce the set
$$
\Phi = \lbrace \varphi \in \mathfrak X(S) \mid c_\varphi = 1
\rbrace.
$$
For every $\varphi \in \Phi$, we put $\Psi_\varphi = \lbrace \alpha
\in \Delta^+ \mid \tau(\alpha) = \varphi \text{ and } \mathfrak
g_{-\alpha} \not\subset \mathfrak n \rbrace$ and
\begin{equation} \label{eqn_disjoint_union}
\Psi = \bigcup \limits_{\varphi \in \Phi} \Psi_\varphi.
\end{equation}
(Note that the union is disjoint.) Clearly,
$$
\Psi = \lbrace \alpha \in \Delta \mid \mathfrak g_{-\alpha}
\not\subset \mathfrak n \rbrace.
$$

\begin{definition}[{\cite[Definition~1]{Avd_solv}}]
Roots in $\Psi$ are said to be \textit{active}.
\end{definition}

The disjoint union given by~(\ref{eqn_disjoint_union}) naturally
determines an equivalence relation on the set~$\Psi$, which will be
denoted by $\sim$. We note that for every $\alpha, \beta \in \Psi$
one has $\alpha \sim \beta$ if and only if $\tau(\alpha) =
\tau(\beta)$.

Active roots have the following property
(see~\cite[Lemma~4]{Avd_solv}): if $\alpha$ is an active root and
$\alpha = \beta + \gamma$ for some roots $\beta, \gamma \in
\Delta^+$, then exactly one of the roots $\beta, \gamma$ is active.
Taking this property into account, we say that an active root
$\beta$ is \textit{subordinate} to an active root $\alpha$ if
$\alpha = \beta + \gamma$ for some $\gamma \in \Delta^+$. For every
active root $\alpha$, we denote by $F(\alpha)$ the set consisting of
$\alpha$ and all active roots subordinate to~$\alpha$. An active
root $\alpha$ is said to be \textit{maximal} if it is subordinate to
no other active root. We denote by $\mathrm M$ the set of maximal
active roots.

The following proposition plays an important role in the structure
theory of strongly solvable spherical subgroups standardly embedded
in~$B^-$.

\begin{proposition}[{\cite[Proposition~1]{Avd_solv}}]
\label{prop_crucial}
Let $\varphi, \varphi' \in \Phi$. Suppose that roots $\alpha \in
\Psi_\varphi$ and $\beta \in \Psi_{\varphi'}$ are different and
$\gamma = \beta - \alpha \in \Delta^+$. Then $\Psi_\varphi + \gamma
\subset \Psi_{\varphi'}$.
\end{proposition}

\begin{corollary} \label{crl_maximal_active_roots}
For every $\varphi \in \Phi$, one has either $\Psi_\varphi \subset
\mathrm M$ or $\Psi_\varphi \cap \mathrm M = \varnothing$.
\end{corollary}

\begin{corollary}[{\cite[Lemma~3]{Avd_solv}}] \label{crl_non-acute}
The angles between the roots in $\mathrm M$ are pairwise non-acute.
In particular, the roots in $\mathrm M$ are linearly independent.
\end{corollary}

For every $\varphi \in \Phi$, the subspace $\mathfrak n_{-\varphi}
\subset \mathfrak u^-_{-\varphi}$ is the kernel of a linear function
$\xi_\varphi \in (\mathfrak u^-_{-\varphi})^*$, which is uniquely
determined up to proportionality. We note that
$\xi_\varphi(e_{-\alpha}) \ne 0$ for all $\alpha \in \Psi_\varphi$
and $\xi_\varphi(e_{-\alpha}) = 0$ for all $\alpha \in \Delta^+
\backslash \Psi_\varphi$ with $\tau(\alpha) = \varphi$.

\begin{proposition}[{\cite[Proposition~2]{Avd_solv}}]
Suppose that $\varphi, \varphi' \in \Phi$, $\varphi \ne \varphi'$,
and $\Psi_\varphi + \gamma \subset\nobreak \Psi_{\varphi'}$ for some
$\gamma \in \Delta^+$. Then there is a constant $c \ne 0$ such that
$\xi_\varphi(e_{-\alpha}) = c\xi_{\varphi'}([e_{-\alpha},
e_{-\gamma}])$ for all $\alpha \in \Psi_\varphi$. In particular, up
to proportionality, $\xi_\varphi$ is uniquely determined
by~$\xi_{\varphi'}$.
\end{proposition}

\begin{corollary} \label{crl_maximal_linear_functions}
The subspace $\mathfrak n \subset \mathfrak u^-$ is uniquely
determined by linear functions\, $\xi_\varphi$, where $\varphi$ runs
over the set $\tau(\mathrm M) \subset \Phi$.
\end{corollary}

Corollaries~\ref{crl_non-acute}
and~\ref{crl_maximal_linear_functions} yield the following result.

\begin{theorem}[{\cite[Theorem~2]{Avd_solv}}]
\label{thm_first_approx}
Up to conjugation by an element of\,~$T$, a strongly solvable
spherical subgroup $H$ standardly embedded in $B^-$ is uniquely
determined by the pair $(S, \Psi)$. Moreover, $H$ is explicitly
recovered from $(S, \Psi)$.
\end{theorem}

\begin{proposition}[{\cite[Proposition~3]{Avd_solv}}]
\label{prop_associated_root} Let $\alpha$ be an active root. Then
there is a unique simple root $\pi(\alpha) \in \Supp \alpha$ with
the following property: if $\alpha = \beta + \gamma$ for some roots
$\beta, \gamma \in \Delta^+$, then $\beta$
\textup(resp.~$\gamma$\textup) is active if and only if $\pi(\alpha)
\notin \Supp \beta$ \textup(resp. $\pi(\alpha) \notin \Supp
\gamma$\textup).
\end{proposition}

This proposition provides a map $\pi \colon \Psi \to \Pi$.

\begin{proposition} \label{prop_amounts}
The set $(\Psi, \sim)$ amounts to the set $(\mathrm M, \pi, \sim)$
with $\pi$ and $\sim$ restricted to~$\mathrm M$.
\end{proposition}

\begin{proof}
Clearly, $(\mathrm M, \pi, \sim)$ is determined by $(\Psi, \sim)$.
Let us show the converse. Proposition~\ref{prop_associated_root}
implies that for every active root $\alpha$ the set $F(\alpha)$ is
uniquely determined by the simple root $\pi(\alpha)$. Since every
active root is subordinate to a maximal active root
(see~\cite[Corollary 3(b)]{Avd_solv}), it follows that the whole set
$\Psi$ is uniquely determined by~$\mathrm M$ and the restriction of
$\pi$ to~$\mathrm M$. At last, Proposition~\ref{prop_crucial} and
Corollary~\ref{crl_maximal_active_roots} show that for $\alpha,
\beta \in \Psi$ one has $\alpha \sim \beta$ if and only if there is
an element $\gamma \in \Delta^+ \cup \lbrace 0 \rbrace$ such that
$\alpha + \gamma, \beta + \gamma \in \mathrm M$ and $\alpha + \gamma
\sim \beta + \gamma$.
\end{proof}

The following lemma together with its corollary will be useful
in~\S\,\ref{subsec_solvabl_invariants}.

\begin{lemma}
Let $\alpha \in \Delta^+$ and consider the set $I_\alpha = \lbrace
\alpha \rbrace \cup \lbrace \beta \in \Delta^+ \mid \alpha - \beta
\in \Delta^+ \rbrace$. Then the sublattice $\ZZ I_\alpha \subset \ZZ
\Delta$ contains $\Supp \alpha$.
\end{lemma}

\begin{proof}
We use the same idea as in the proof of~\cite[Lemma~6(b)]{Avd_solv}.
The proof is by induction on~$\hgt \alpha$. If $\hgt \alpha = 1$
then the assertion is true. Assume that $\hgt \alpha = k$ and the
assertion is proved for all $\alpha' \in \Delta^+$ with $\hgt
\alpha' < k$. By a well-known lemma from linear algebra
(see~\cite[Lemma~1]{Avd_solv}) there is a root $\alpha_0 \in \Pi$
such that $(\alpha, \alpha_0) > 0$. Then $\alpha - \alpha_0 \in
\Delta^+$, whence $\alpha_0 \in I_\alpha$. Let $r_0$ be the
reflection corresponding to $\alpha_0$. Note that for every $\gamma
\in \Delta$ one has $r_0(\gamma) = \gamma - \langle \alpha_0^\vee,
\gamma \rangle \alpha_0 \in \gamma + \ZZ\alpha_0$. Consider the root
$\beta = r_0(\alpha) \in \Delta^+$. Clearly, $\hgt \beta < \hgt
\alpha$ and $\Supp \alpha \supset \Supp \beta \supset \Supp \alpha
\backslash \lbrace \alpha_0 \rbrace$. By the induction hypothesis,
$\ZZ I_\beta \supset \Supp \beta$. Since for every $\beta' \in
I_\beta$ one has either $r_0(\beta') \in I_\alpha$ or $r_0(\beta')
\in \lbrace - \alpha_0, \alpha + \alpha_0 \rbrace$, we conclude that
$I_\alpha \supset \lbrace \alpha_0 \rbrace \cup (r_0(I_\beta)
\backslash \lbrace - \alpha_0, \alpha + \alpha_0 \rbrace)$. Hence
$$
\ZZ I_\alpha \supset \ZZ (r_0(I_\beta) \cup \lbrace \alpha_0
\rbrace) \supset \ZZ(r_0(\Supp \beta) \cup \lbrace \alpha_0 \rbrace)
\supset \Supp \beta \cup \lbrace \alpha_0 \rbrace = \Supp \alpha,
$$
where the relation $r_0(\Supp \beta) \subset \Supp \beta + \ZZ
\alpha_0$ is taken into account.
\end{proof}

To state the corollary, we need to introduce the set $\Pi_0 =
\bigcup\limits_{\beta \in \mathrm M} \Supp \beta \subset \Pi$. Since
for every $\alpha \in \Psi$ the map $\pi \colon F(\alpha) \to \Supp
\alpha$ is bijective (see~\cite[Corollary~6]{Avd_solv}), we have
$\Pi_0 = \pi(\Psi)$.

\begin{corollary} \label{crl_contains_support}
The following assertions hold:
\begin{enumerate}[label=\textup{(\alph*)},ref=\textup{\alph*}]
\item \label{crl_contains_support_a}
for every $\alpha \in \Psi$, the sublattice $\ZZ F(\alpha) \subset
\mathfrak X(T)$ contains $\Supp \alpha$;

\item \label{crl_contains_support_b}
the sublattice $\ZZ \Psi \subset \mathfrak X(T)$ coincides with $\ZZ
\Pi_0$;

\item \label{crl_contains_support_c}
the sublattice $\ZZ \Phi \subset \mathfrak X(S)$ coincides with
$\tau(\ZZ \Pi_0)$.
\end{enumerate}
\end{corollary}

\subsection{Classification} \label{subsec_explicit_class}

In this subsection we present the classification of strongly
solvable spherical subgroups standardly embedded in $B^-$, up to
conjugation by elements of~$T$, and then explain when two such
subgroups are conjugate in~$G$.

Let $H \subset G$ be a strongly solvable spherical subgroup
standardly embedded in~$B^-$. We retain all the notation introduced
in~\S\,\ref{subsec_solv_description}.

The following proposition lists all possibilities for a pair
$(\alpha, \pi(\alpha))$ with $\alpha \in \Psi$.

\begin{proposition}[{\rm \cite[Theorem~3]{Avd_solv}}] \label{prop_active_roots}
For every active root $\alpha$, the pair $(\alpha, \pi(\alpha))$ is
contained in Table~\textup{\ref{table_active_roots}}.
\end{proposition}

\begin{table}[h]

\caption{Active roots}\label{table_active_roots}

\begin{center}

\begin{tabular}{|c|c|c|c|}

\hline

No. & Type of $\Supp \alpha$ & $\alpha$ & $\pi(\alpha)$ \\

\hline

1 & any of rank $r$ & $\alpha_1 + \alpha_2 + \ldots + \alpha_r$ &
$\alpha_1, \alpha_2, \ldots, \alpha_r$\\

\hline

2 & $\mathsf B_r$ & $\alpha_1 + \alpha_2 + \ldots + \alpha_{r-1} +
2\alpha_r$ & $\alpha_1, \alpha_2, \ldots, \alpha_{r-1}$\\

\hline

3 & $\mathsf C_r$ & $2\alpha_1 + 2\alpha_2 + \ldots + 2\alpha_{r-1}
+ \alpha_r$ & $\alpha_r$\\

\hline

4 & $\mathsf F_4$ & $2\alpha_1 + 2\alpha_2 + \alpha_3 + \alpha_4$ &
$\alpha_3, \alpha_4$\\

\hline

5 & $\mathsf G_2$ & $2\alpha_1 + \alpha_2$ & $\alpha_2$\\

\hline

6 & $\mathsf G_2$ & $3\alpha_1 + \alpha_2$ & $\alpha_2$\\

\hline
\end{tabular}

\end{center}

\end{table}

The notation in Table~\ref{table_active_roots} is as follows. We
denote by $\alpha_i$ the $i$th simple root in $\Supp \alpha$. In the
column ``$\pi(\alpha)$'' we list all possibilities for $\pi(\alpha)$
for a given active root~$\alpha$.

Let $\alpha \in \Delta^+$. A simple root $\delta \in \Supp \alpha$
is said to be \textit{terminal with respect to $\Supp \alpha$} if in
the Dynkin diagram of $\Supp \alpha$ the node corresponding to
$\delta$ is joined by an edge (possibly multiple) with exactly one
node.

We now list several conditions on a pair $\alpha, \beta$ of active
roots. These conditions will be used below in
Definition~\ref{def_ARS}.

\begin{figure}[h]

\begin{picture}(150,115)
\put(70,75){\circle{4}} \put(65,83){$\gamma_0$}

\put(70,73){\line(0,-1){18}} \put(70,53){\circle{4}}
\put(76,51){$\gamma_1$} \put(70,51){\line(0,-1){11}}
\put(70,36){\circle*{0.5}} \put(70,32){\circle*{0.5}}
\put(70,28){\circle*{0.5}} \put(70,25){\line(0,-1){11}}
\put(70,12){\circle{4}} \put(76,10){$\gamma_s$}

\put(72,76){\line(2,1){16}} \put(68,76){\line(-2,1){16}}

\put(90,85){\circle{4}} \put(88,74){$\beta_1$}
\put(92,86){\line(2,1){10}} \put(106,93){\circle*{0.5}}
\put(110,95){\circle*{0.5}} \put(114,97){\circle*{0.5}}
\put(118,99){\line(2,1){10}} \put(130,105){\circle{4}}
\put(128,94){$\beta_q$}

\put(50,85){\circle{4}} \put(45,74){$\alpha_1$}
\put(48,86){\line(-2,1){10}} \put(34,93){\circle*{0.5}}
\put(30,95){\circle*{0.5}} \put(26,97){\circle*{0.5}}
\put(22,99){\line(-2,1){10}} \put(10,105){\circle{4}}
\put(05,94){$\alpha_p$}

\end{picture}

{ \makeatletter

\renewcommand{\@makecaption}[2]{
\vspace{\abovecaptionskip}%
\sbox{\@tempboxa}{#1 #2}%
\global\@minipagefalse \hbox to \hsize {{\scshape \hfil #1 #2\hfil}}
\vspace{\belowcaptionskip}}

\makeatother

\caption{}\label{diagram_difficult}

}
\end{figure}
\begin{enumerate}
\setcounter{enumi}{-1}
\renewcommand{\labelenumi}{(D\arabic{enumi})}
\renewcommand{\theenumi}{D\arabic{enumi}}

\item \label{D0}
$\Supp \alpha \cap \Supp \beta = \varnothing$;

\item \label{D1}
$\Supp \alpha \cap \Supp \beta = \{\delta\}$, where $\pi(\alpha) \ne
\delta$, $\pi(\beta) \ne \delta$, and $\delta$ is terminal with
respect to both $\Supp \alpha$ and $\Supp \beta$;

\setcounter{enumi}{0}
\renewcommand{\labelenumi}{(E\arabic{enumi})}
\renewcommand{\theenumi}{E\arabic{enumi}}

\item \label{E1}
$\Supp \alpha \cap \Supp \beta = \{\delta\}$, where $\delta =
\pi(\alpha) = \pi(\beta)$, $\alpha - \delta \in \Delta^+$, $\beta -
\delta \in \Delta^+$, and $\delta$ is terminal with respect to both
$\Supp \alpha$ and $\Supp \beta$;

\setcounter{enumi}{1}
\renewcommand{\labelenumi}{(D\arabic{enumi})}
\renewcommand{\theenumi}{D\arabic{enumi}}

\item \label{D2}
the Dynkin diagram of $\Supp \alpha \cup \Supp \beta$ has the form
shown in Figure~\ref{diagram_difficult} (for some $p, q, s\ge 1$),
$\alpha = \alpha_1 + \ldots + \alpha_p + \gamma_0 + \gamma_1 +
\ldots + \gamma_s$, $\beta = \beta_1 + \ldots + \beta_q + \gamma_0 +
\gamma_1 + \ldots + \gamma_s$, $\pi(\alpha) \notin \Supp \alpha \cap
\Supp \beta$, and $\pi(\beta) \notin \Supp \alpha \cap \Supp \beta$;

\setcounter{enumi}{1}
\renewcommand{\labelenumi}{(E\arabic{enumi})}
\renewcommand{\theenumi}{E\arabic{enumi}}

\item \label{E2}
the Dynkin diagram of $\Supp \alpha \cup \Supp \beta$ has the form
shown in Figure~\ref{diagram_difficult} (for some $p,q,s\ge 1$),
$\alpha = \alpha_1 + \ldots + \alpha_p + \gamma_0 + \gamma_1 +
\ldots + \gamma_s$, $\beta = \beta_1 + \ldots + \beta_q + \gamma_0 +
\gamma_1 + \ldots + \gamma_s$, and $\pi(\alpha) = \pi(\beta) \in
\Supp \alpha \cap \Supp \beta$.
\end{enumerate}

\begin{definition} \label{def_ARS}
Suppose that $S \subset T$ is a subgroup, $\mathrm M \subset
\Delta^+$ is a subset, $\pi \colon \mathrm M \to \Pi$ is a map, and
$\sim$ is an equivalence relation on~$\mathrm M$. Let $\tau \colon
\mathfrak X(T) \to \mathfrak X(S)$ be the character restriction map
and put $\Pi_0 = \bigcup \limits_{\beta \in \mathrm M} \Supp \beta$.

The triple $(\mathrm M, \pi, \sim)$ is said to be an
\textit{ARS-set}\footnote{``ARS'' is an abbreviation for ``active
root system''.} if it satisfies the following conditions:
\begin{enumerate}
\renewcommand{\labelenumi}{(A)}
\renewcommand{\theenumi}{A}

\item \label{AA}
if $\alpha \in \mathrm M$, then $\pi(\alpha) \in \Supp \alpha$ and
the pair $(\alpha, \pi(\alpha))$ is contained in
Table~\textup{\ref{table_active_roots}};

\renewcommand{\labelenumi}{(D)}
\renewcommand{\theenumi}{D}

\item \label{DD}
if $\alpha, \beta \in \mathrm M$ and $\alpha \nsim \beta$, then one
of conditions (\ref{D0}), (\ref{D1}), (\ref{D2}) holds;

\renewcommand{\labelenumi}{(E)}
\renewcommand{\theenumi}{E}

\item \label{EE}
if $\alpha, \beta \in \mathrm M$ and $\alpha \sim \beta$, then one
of conditions (\ref{D0}), (\ref{D1}), (\ref{E1}), (\ref{D2}),
(\ref{E2}) holds;

\renewcommand{\labelenumi}{(C)}
\renewcommand{\theenumi}{C}

\item \label{CC} if $\alpha \in \mathrm M$, then $\Supp \alpha
\not\subset \bigcup\limits_{\delta \in \mathrm M \backslash
\{\alpha\}}\Supp \delta$.
\end{enumerate}

The quadruple $(S, \mathrm M, \pi, \sim)$ is said to be an
\textit{extended ARS-set} if $(\mathrm M, \pi, \sim)$ is an ARS-set
and condition~(\ref{TT}) below is satisfied:
\begin{enumerate}
\renewcommand{\labelenumi}{(T)}
\renewcommand{\theenumi}{T}
\item \label{TT}
$\Ker \tau \cap \ZZ \Pi_0 = \ZZ \lbrace \alpha - \beta \mid \alpha,
\beta \in \mathrm M, \alpha \sim \beta \rbrace$.
\end{enumerate}
\end{definition}

\begin{remark}
Thanks to
Corollary~\ref{crl_contains_support}(\ref{crl_contains_support_b}),
the original form of condition~(\ref{TT}) from~\cite{Avd_solv} is
equivalent to that presented in the definition.
\end{remark}

We now return to our subgroup~$H$. With $H$ we associate the
invariants $\Upsilon(H) = (S, \mathrm M, \pi, \sim)$ and
$\Upsilon_0(H) = (\mathrm M, \pi, \sim)$, where $\pi$ and $\sim$ are
restricted to~$\mathrm M$.

The following theorem provides a classification of all strongly
solvable spherical subgroups in~$G$ standardly embedded in~$B^-$,
see~\cite[Theorems~4 and~5]{Avd_solv}.

\begin{theorem} \label{thm_classification_ARS}
The map $H \mapsto \Upsilon(H)$ is a bijection between strongly
solvable spherical subgroups standardly embedded in~$B^-$, up to
conjugation by elements of\,~$T$, and extended ARS-sets.
\end{theorem}

To complete the presentation of the explicit classification, in the
remaining part of this subsection we explain when two strongly
solvable spherical subgroups of $G$ standardly embedded in $B^-$ are
conjugate in~$G$.

\begin{definition}[{\cite[Definition~7]{Avd_solv}}]
A root $\delta \in \Psi$ is said to be \textit{regular} if the
projection of the subspace $\mathfrak n \subset \mathfrak u^-$ to
the root subspace $\mathfrak g_{-\delta}$ along the sum of the other
root subspaces is zero.
\end{definition}

Let $\Psi^{\mathrm{reg}}$ denote the set of regular active roots.

\begin{definition}[{\cite[Definition~8]{Avd_solv}}]
Suppose that $\delta \in \Psi^{\mathrm{reg}} \cap \Pi$. An
\textit{elementary transformation with center~$\delta$} (or simply
an \textit{elementary transformation}) is a transformation of the
form $H \mapsto \sigma_\delta H \sigma_\delta^{-1}$, where
$\sigma_\delta \in N_G(T)$ is any element whose image in the Weyl
group $N_G(T) / T$ coincides with the simple reflection associated
with~$\delta$.
\end{definition}

Evidently, in this definition the group $\sigma_\delta H
\sigma_\delta^{-1}$ is also standardly embedded in~$B^-$.

\begin{theorem}[{\cite[Theorem~6]{Avd_solv}}]
Let $H, H' \subset G$ be strongly solvable spherical subgroups
standardly embedded in~$B^-$. The subgroups $H, H'$ are conjugate in
$G$ if and only if there is a chain of elementary transformations
taking $H$ to~$H'$.
\end{theorem}

\begin{remark}
For an elementary transformation $H \mapsto H'$, the extended
ARS-set of $H'$ is explicitly expressed in terms of that of~$H$,
see~\cite[Proposition~14]{Avd_solv}.
\end{remark}

\begin{remark}
In the following subsection we shall compute the homogeneous
spherical datum corresponding to a strongly solvable spherical
subgroup of $G$ standardly embedded in~$B^-$, see
Theorem~\ref{thm_PCI_via_EWS}. By Theorems~\ref{thm_uniqueness}
or~\ref{thm_classification} this will provide an alternative
approach for determining when two strongly solvable spherical
subgroups $H,H' \subset G$ standardly embedded in $B^-$ are
conjugate in~$G$.
\end{remark}

\subsection{Computation of the invariants}
\label{subsec_solvabl_invariants}

Let $H \subset G$ be a strongly solvable spherical subgroup
standardly embedded in~$B^-$. We retain all the notation introduced
in~\S\,\ref{subsec_solv_description}.

Let $\mathscr L \subset \mathfrak X(T)$ be the sublattice generated
by all elements of the form $\alpha - \beta$, where $\alpha, \beta
\in \Psi$ and $\alpha \sim \beta$. Clearly, $\mathscr L \subset \Ker
\tau$.

For every $\alpha \in \Pi$ we introduce the element $\Omega_\alpha =
(\varpi_\alpha, - \tau(\varpi_\alpha)) \in \mathfrak X_+(B) \oplus
\mathfrak X(H)$. For every $\varphi \in \Phi$, we put
$\lambda_\varphi = \sum \limits_{\alpha \in \pi(\Psi_\varphi)}
\varpi_\alpha$ and introduce the element $\Omega_\varphi =
(\lambda_\varphi, - \tau(\lambda_\varphi) + \varphi) \in \mathfrak
X_+(B) \oplus \mathfrak X(H)$.

\begin{proposition} \label{prop_solv_ews}
Suppose that $G$ is simply connected. Then the semigroup $\widehat
\Lambda^+_{G / H}$ is freely generated by all the
elements\,~$\Omega_\alpha$, $\alpha \in \Pi$, and all the
elements\,~$\Omega_\varphi$, $\varphi \in \Phi$.
\end{proposition}

\begin{proof}
In the case of connected $H$ the assertion was proved
in~\cite[Theorem~4]{AG}. (In fact, in~\cite{AG} the group $H$ was
assumed to be standardly embedded in~$B$, so one has to translate
those results into our settings.) But the condition of $H$ being
connected is inessential here. Indeed, it is obvious that for every
$\alpha \in \Pi$ the subspace
$V(\varpi^*_\alpha)^{(H^0)}_{-\tau(\varpi_\alpha)} \subset
V(\varpi^*_\alpha)$ is spanned by a lowest-weight vector
$w_{\varpi^*_\alpha}$, which is even $B^-$-semi-invariant. Further,
it was shown in~\cite{AG} that for every $\varphi \in \Phi$ the
subspace $V(\lambda^*_\varphi)^{(H^0)}_{- \tau(\lambda_\varphi) +
\varphi} \subset V(\lambda^*_\varphi)$ is spanned by a vector of the
form $(\sum \limits_{\beta \in \Psi_\varphi} b_\beta e_{\beta})
\cdot w_{\lambda^*_\varphi}$, where all the coefficients $b_\beta$
are nonzero. For all of these vectors to be $S$-semi-invariant, the
only condition that matters is $\mathscr L \subset \Ker \tau$.
\end{proof}

Taking into account Proposition~\ref{prop_ews_in_general}, we get
the following theorem.

\begin{theorem} \label{thm_solv_ews}
Suppose that $G = C \times G^{ss}$ and $G^{ss}$ is simply connected.
Then $\widehat \Lambda^+_{G / H} = \widehat \Lambda^+_{G^{ss} /
H^{ss}} \oplus \lbrace (\nu, -\nu_H) \mid \nu \in \mathfrak X(C)
\rbrace$, where $\widehat \Lambda^+_{G^{ss} / H^{ss}}$ is freely
generated by all the elements\,~$\Omega_\alpha$, $\alpha \in \Pi$,
and all the elements\,~$\Omega_\varphi$, $\varphi \in \Phi$.
\end{theorem}

For every $\alpha \in \Pi$ (resp. $\varphi \in \Phi$), let
$D_\alpha$ (resp. $D_\varphi$) be the color of $G/H$ corresponding
to the indecomposable element $\Omega_\alpha$ (resp.
$\Omega_\varphi$) of $\widehat \Lambda^+_{G^{ss} / H^{ss}}$.

\begin{remark} \label{rem_division_of_colors}
Under the natural morphism $G / H \to G / B^-$, every color
$D_\alpha$, $\alpha \in \Pi$, is the preimage of a color of $G /
B^-$, and every color $D_\varphi$, $\varphi \in \Phi$, maps
dominantly to $G / B^-$.
\end{remark}

We denote by $\widehat S$ the subgroup of\,~$T$ defined by the
vanishing of all elements in~$\mathscr L$. Clearly, $\widehat S$ is
the largest subgroup of\, $T$ normalizing~$N$. We set $\widehat H =
\widehat S \rightthreetimes N$.

\begin{proposition} \label{prop_sph_clos_str_solv}
The spherical closure $\overline H$ of $H$ coincides with~$\widehat
H$.
\end{proposition}

\begin{proof}
Passing to a finite covering of~$G$, without loss of generality we
may assume that $G = C \times G^{ss}$ and $G^{ss}$ is simply
connected. Then we may apply
Proposition~\ref{prop_spherical_closure}. Repeating the proof of
Proposition~\ref{prop_sph_clos_in_parabolic}, we find that the
common stabilizer in~$G$ of all the lines $\langle
w_{\varpi^*_\alpha} \rangle =
V(\varpi^*_\alpha)^{(H)}_{-\tau(\varpi_\alpha)}$, $\alpha \in \Pi$,
coincides with~$B^-$, whence $\overline H \subset B^-$. On the other
hand, by~\cite[Theorem~3]{Avd_norm} one has $N_G(H) \cap B^- =
N_G(H^0) \cap B^- = \widehat H$. (Here we also use that $N_G(H) =
N_G(H^0)$; see~\cite[\S\,5.2, Corollary]{BriP}.) As was already
mentioned in the proof of Theorem~\ref{prop_solv_ews}, for every
$\varphi \in \Phi$ the line $V(\lambda^*_\varphi)^{(H)}_{-
\tau(\lambda_\varphi) + \varphi}$ is $\widehat H$-semi-invariant.
\end{proof}

\begin{corollary} \label{crl_SSWS_ARS}
The following assertions hold:
\begin{enumerate}[label=\textup{(\alph*)},ref=\textup{\alph*}]
\item \label{crl_SSWS_ARS_a}
a strongly solvable wonderful subgroup of $G$ contained in $B^-$ is
uniquely determined by its unipotent radical;

\item \label{crl_SSWS_ARS_b}
there is a bijection between ARS-sets and conjugacy classes in $B^-$
of strongly solvable wonderful subgroups of $G$ contained in $B^-$.
\end{enumerate}
\end{corollary}

\begin{proof}
Recall from Corollary~\ref{crl_str_solv_wond_sph_clos} that every
strongly solvable wonderful subgroup of~$G$ is spherically closed.

(\ref{crl_SSWS_ARS_a}) (compare with the proof
of~\cite[Lemma~32]{Avd_solv}) Let $H$ be a strongly solvable
wonderful subgroup of $G$ standardly embedded in~$B^-$ and let $N$
be the unipotent radical of~$H$. Then by
Proposition~\ref{prop_sph_clos_str_solv} a Levi subgroup of $H$ is
recovered as a Levi subgroup of $N_{B^-}(N)$.

(\ref{crl_SSWS_ARS_b}) Let $H, H' \subset G$ be two strongly
solvable wonderful subgroups standardly embedded in~$B^-$ and let
$\Upsilon_0(H), \Upsilon_0(H')$ be the corresponding ARS-sets.
Suppose that $H' = bHb^{-1}$ for some $b \in B^-$. Then
by~\cite[Proposition~13]{Avd_solv} one has $H' = b' H b'^{-1}$,
where $b' \in N_G(T) \cap\nobreak B^- =\nobreak T$, whence
$\Upsilon_0(H) = \Upsilon_0(H')$.
\end{proof}

According to
Corollary~\ref{crl_contains_support}(\ref{crl_contains_support_c})
and Proposition~\ref{prop_sphericity_criterion}, for every $\mu \in
\ZZ \Pi_0 + \Ker \tau$ and every $\varphi \in \Phi$ there exists a
unique integer $J(\varphi, \mu)$ such that
\begin{equation} \label{eqn_tau(mu)}
\tau(\mu) = \sum \limits_{\varphi \in \Phi} J(\varphi, \mu)\varphi.
\end{equation}

The following lemma is a direct consequence
of~\cite[Lemma~10]{Avd_solv}.

\begin{lemma} \label{lemma_disjoint_union}
The union $\pi(\Psi) = \bigcup \limits_{\varphi \in \Phi}
\pi(\Psi_\varphi)$ is disjoint.
\end{lemma}

In view of this lemma, every root $\alpha \in \Pi_0$ determines a
unique weight $\varphi[\alpha] \in \Phi$ such that $\alpha \in
\pi(\Phi_{\varphi[\alpha]})$.

\begin{theorem} \label{thm_PCI_via_EWS}
The homogeneous spherical datum of\, $G / H$ is determined as
follows:
\begin{enumerate}[label=\textup{(\alph*)},ref=\textup{\alph*}]
\item \label{thm_PCI_via_EWS_a}
$\Lambda_{G/H} = \ZZ \Pi_0 + \Ker \tau$.

\item \label{thm_PCI_via_EWS_b}
$\Pi^p_{G/H} = \varnothing$.

\item \label{thm_PCI_via_EWS_c}
$\Sigma_{G/H} = \Pi_0$.

\item \label{thm_PCI_via_EWS_d}
The set $\mathcal D^a_{G/H}$ consists of all divisors $D_\alpha$,
$\alpha \in \Pi_0$, and all divisors $D_\varphi$, $\varphi \in
\Phi$. Moreover, one has
\begin{itemize}
\item
$\langle \varkappa(D_\alpha), \mu \rangle = \langle \alpha^\vee, \mu
\rangle - J(\varphi[\alpha], \mu)$ for every $\alpha \in \Pi_0$ and
$\mu \in \ZZ\Pi_0 + \Ker \tau$;

\item
$\langle \varkappa(D_\varphi), \mu \rangle = J(\varphi, \mu)$ for
every $\varphi \in \Phi$ and $\mu \in \ZZ\Pi_0 + \Ker \tau$.
\end{itemize}
\end{enumerate}
\end{theorem}

\begin{proof}
Passing to a finite covering of~$G$, without loss of generality we
may assume that $G = C \times G^{ss}$ and $G^{ss}$ is simply
connected.

(\ref{thm_PCI_via_EWS_a}) Let $\mu \in \mathfrak X(T)$. By
Proposition~\ref{prop_weight_lattice} and formula~(\ref{eqn_(mu,
0)}), $\mu \in \Lambda_{G/H}$ if and only if there is an expression
\begin{equation} \label{eqn_mu_in_weight_lattice}
(\mu,0) = \sum \limits_{\alpha \in \Pi} c_\alpha \Omega_\alpha +
\sum \limits_{\varphi \in \Phi} c_\varphi \Omega_\varphi + (\nu, -
\tau(\nu)),
\end{equation}
where $c_\alpha, c_\varphi \in \ZZ$ and $\nu \in \mathfrak X(C)$.

First suppose that $\mu \in \Lambda_{G/H}$.
Then~(\ref{eqn_mu_in_weight_lattice}) implies the following equality
in the group $\mathfrak X(S)$:
$$
0 = -\sum \limits_{\alpha \in \Pi} c_\alpha \tau(\varpi_\alpha)
-\sum \limits_{\varphi \in \Phi} c_\varphi \tau(\lambda_\varphi) +
\sum \limits_{\varphi \in \Phi} c_\varphi \varphi - \tau(\nu).
$$
For every $\varphi \in \Phi$, fix any root $\alpha_\varphi \in
\Psi_\varphi$. Then
$$
\sum \limits_{\alpha \in \Pi} c_\alpha \varpi_\alpha + \sum
\limits_{\varphi \in \Phi} c_\varphi \lambda_\varphi - \sum
\limits_{\varphi \in \Phi} c_\varphi \alpha_\varphi + \nu \in \Ker
\tau.
$$
Hence in view of~(\ref{eqn_mu_in_weight_lattice}) one has
$$
\mu = \sum \limits_{\alpha \in \Pi} c_\alpha \varpi_\alpha + \sum
\limits_{\varphi \in \Phi} c_\varphi \lambda_\varphi + \nu \in \sum
\limits_{\varphi \in \Phi} c_\varphi \alpha_\varphi + \Ker \tau
\subset \ZZ\Psi + \Ker \tau = \ZZ \Pi_0 + \Ker \tau,
$$
where the latter equality holds by
Corollary~\ref{crl_contains_support}(\ref{crl_contains_support_b}).

Conversely, suppose that $\mu \in \ZZ \Pi_0 + \Ker \tau = \ZZ \Psi +
\Ker \tau$ and consider the corresponding expression of the
form~(\ref{eqn_tau(mu)}). One has
\begin{multline*}
(0, \tau(\mu)) = \sum \limits_{\varphi \in \Phi} J(\varphi, \mu) (0,
\varphi) =\\ \sum \limits_{\varphi \in \Phi} J(\varphi, \mu)
(\Omega_{\varphi} - \sum \limits_{\alpha \in \pi(\Psi_\varphi)}
\Omega_\alpha) = \sum \limits_{\varphi \in \Phi} J(\varphi, \mu)
\Omega_\varphi - \sum \limits_{\varphi \in \Phi} \sum
\limits_{\alpha \in \pi(\Psi_\varphi)} J(\varphi, \mu)
\Omega_\alpha.
\end{multline*}
By Lemma~\ref{lemma_disjoint_union}, the double sum in the latter
expression is just a sum over the set $\pi(\Psi) = \Pi_0$, whence
\begin{equation*}
(0, \tau(\mu)) = \sum \limits_{\varphi \in \Phi} J(\varphi, \mu)
\Omega_\varphi - \sum \limits_{\alpha \in \Pi_0} J(\varphi[\alpha],
\mu) \Omega_\alpha.
\end{equation*}
Recall that there is a unique expression $\mu = \mu^{ss} + \mu^C$,
where $\mu^{ss} \in \mathfrak X(B^{ss})$ and $\mu^C \in \mathfrak
X(C)$. Then $(\mu, - \tau(\mu)) = \sum \limits_{\alpha \in \Pi}
\langle \alpha^\vee, \mu \rangle \Omega_\alpha + (\mu^C, -
\tau(\mu^C))$ and we finally obtain
\begin{equation} \label{eqn_element_of_weight_lattice}
(\mu, 0) = \sum \limits_{\alpha \in \Pi} \langle \alpha^\vee, \mu
\rangle \Omega_\alpha - \sum \limits_{\alpha \in \Pi_0}
J(\varphi[\alpha], \mu) \Omega_\alpha + \sum \limits_{\varphi \in
\Phi} J(\varphi, \mu) \Omega_\varphi + (\mu^C, - \tau(\mu^C)),
\end{equation}
whence $\mu \in \Lambda_{G/H}$.

(\ref{thm_PCI_via_EWS_b})--(\ref{thm_PCI_via_EWS_d}) The knowledge
of the semigroup $\widehat \Lambda^+_{G / H}$ in combination with
Proposition~\ref{prop_Sigma_via_supports} yields $\Pi^p_{G / H} =
\varnothing$ and $\Sigma_{G / H} \cap \Pi = \Pi_0$. As $\Sigma_{G /
H} \subset \Pi$ (see Corollary~\ref{crl_strongly_solvable_HSD}), we
get $\Sigma_{G / H} = \Pi_0$. At last, the set $\mathcal D^a_{G /
H}$ consists of all the divisors $D_\alpha$ with $\alpha \in \Pi_0$
and all the divisors $D_\varphi$, where $\varphi \in \Phi$. By
Proposition~\ref{prop_value_of_color}, the values on $\Lambda_{G /
H}$ of elements in $\varkappa(\mathcal D^a_{G / H})$ are read off
from expression~(\ref{eqn_element_of_weight_lattice}).
\end{proof}

\subsection{Relationship with Luna's 1993 approach}
\label{subsec_1993&explicit}

Theorem~\ref{thm_AM_SSWS} together with
Corollary~\ref{crl_SSWS_ARS}(\ref{crl_SSWS_ARS_b}) imply that there
is a natural bijection between admissible maps and ARS-sets. The
goal of this subsection is to find an explicit description of this
bijection.

Let $H \subset G$ be a strongly solvable wonderful subgroup
standardly embedded in~$B^-$. Let $(\mathrm M, \pi, \sim)$
(resp.~$\eta$) be the ARS-set (resp. admissible map) corresponding
to~$H$. Recall that by Proposition~\ref{prop_amounts} the set
$(\mathrm M, \pi, \sim)$ amounts to the pair $(\Psi, \sim)$.

From \S\,\ref{subsec_wond_B-var} we recall the set $\Pi_\eta =
\lbrace \alpha \in \Pi \mid \eta(\alpha, \alpha) = 1 \rbrace$. One
has $\Pi_\eta = \Pi_0$ since both $\Pi_\eta$ and $\Pi_0$ coincide
with~$\Sigma_{G / H}$ (see
Proposition~\ref{prop_invariants_via_AM}(\ref{prop_invariants_via_AM_b})
and Theorem~\ref{thm_PCI_via_EWS}(\ref{thm_PCI_via_EWS_c})).

We first express $\eta$ in terms of the pair $(\Psi, \sim)$.

Clearly, $B^- / H$ is a smooth affine toric $T$-variety whose weight
lattice $\mathfrak X = \mathfrak X_{B^- / H}$ is generated by the
set~$\Pi_\eta$. Let $\mathcal F = \mathcal F_{B^-/H}$ be the
corresponding fan in $Q_{B^-/H} = \Hom_\ZZ(\mathfrak X, \QQ)$. As
$B^-/H$ is smooth and affine, $\mathcal F$ contains a unique maximal
cone~$\mathcal C_0$, which is regular.

The natural projection $B^- / H \to T / S$ yields a $T$-equivariant
isomorphism $B^- / H \simeq T *_S\nobreak (U^- / N)$. Hence
$T$-stable prime divisors in $B^- / H$ are in natural bijection with
$S$-stable prime divisors in $U^- / N$. By~\cite[Lemma~1.4]{Mon},
there is an $S$-equivariant isomorphism $U^- / N \simeq \mathfrak
u^- / \mathfrak n$. Next,
Proposition~\ref{prop_sphericity_criterion} yields an
$S$-equivariant isomorphism
\begin{equation} \label{eqn_u-/n}
\mathfrak u^- / \mathfrak n \simeq \bigoplus\limits_{\varphi \in
\Phi} \CC_{-\varphi},
\end{equation}
where $\CC_{-\varphi}$ is the one-dimensional $S$-module on which
$S$ acts via the character~$-\varphi$. Clearly, the $S$-stable prime
divisors in $\bigoplus \limits_{\varphi \in \Phi} \CC_{-\varphi}$
are just the coordinate hyperplanes, hence they are in natural
bijection with the set~$\Phi$. For every $\varphi \in \Phi$, let
$E_\varphi$ be the prime divisor in $B^- / H$ corresponding
to~$\varphi$ via the above-mentioned natural bijections. Let also
$q_\varphi$ denote the element of $\mathcal F^1$ corresponding
to~$E_\varphi$.

\begin{lemma} \label{lemma_q_varphi}
For every $\varphi \in \Phi$ and every $\alpha \in \Psi$, one has
\begin{equation} \label{eqn_values_of_q_varphi}
\langle q_\varphi, \alpha \rangle = \begin{cases} 1 & \text{if
$\alpha \in \Psi_\varphi$}; \\ 0 & \text{if $\alpha \notin
\Psi_\varphi$}.
\end{cases}
\end{equation}
In particular, $\langle q_\varphi, \Ker \tau \rangle = 0$ for every
$\varphi \in \Phi$.
\end{lemma}

\begin{proof}
According to the above discussion, there is a $T$-equivariant
isomorphism
\begin{equation} \label{eqn_B-/H}
B^- / H \simeq T *_S (\mathfrak u^- / \mathfrak n).
\end{equation}
By Theorem~\ref{thm_homogeneous_bundles}, the right-hand side
of~(\ref{eqn_B-/H}) is the geometrical quotient of $T \times
(\mathfrak u^- / \mathfrak n)$ by the action of~$S$ given by $s(t,x)
= (ts^{-1}, sx)$, where $s \in S$, $t \in T$, $x \in\nobreak
\mathfrak u^- / \mathfrak n$. Consequently, $\CC[B^- / H] \simeq
(\CC[T] \otimes \CC[\mathfrak u^- / \mathfrak n])^S$, where the
invariants are taken with respect to the induced diagonal action
of~$S$ on functions. Taking into account~(\ref{eqn_u-/n}), one
easily deduces that the weight semigroup of $T$-semi-invariant
regular functions on $B^- / H$ equals $\tau^{-1}(\ZZ^+ \Phi)$.
Corollary~\ref{crl_contains_support}(\ref{crl_contains_support_c})
yields $\tau(\mathfrak X) = \ZZ \Phi$, hence the cone $\mathcal C_0$
is the image in $Q_{B^- / H}$ of the cone $(\QQ^+ \Phi)^\vee \subset
\Hom_\ZZ(\ZZ\Phi, \QQ)$. Clearly, the set $\lbrace q_\varphi \mid
\varphi \in \Phi \rbrace$ is the image in $Q_{B^- / H}$ of the basis
of $\Hom_\ZZ(\ZZ\Phi, \ZZ)$ dual to~$\Phi$, whence the claim.
\end{proof}

Below we shall need the map $\rho_\eta \colon \Pi_\eta \to
\Hom_\ZZ(\ZZ\Pi_\eta, \ZZ)$ given by~(\ref{eqn_rho_alpha}). We also
recall that the fan $\mathcal F$ coincides with the fan $\mathcal
F^0_\eta$ introduced in the proof of
Proposition~\ref{prop_invariants_via_AM}, see also
Remark~\ref{rem_cone_of_B/H}.

\begin{lemma} \label{lemma_fans_coincide}
For every $\alpha \in \Pi_0$, one has $\rho_\eta(\alpha) =
q_{\varphi[\alpha]}$.
\end{lemma}

\begin{proof}
Fix a root $\widehat \alpha \in \Psi_{\varphi[\alpha]}$ such that
$\pi(\widehat \alpha) = \alpha$. By
Corollary~\ref{crl_contains_support}(\ref{crl_contains_support_a}),
one has $\alpha \in \ZZ F(\widehat \alpha)$.
Proposition~\ref{prop_associated_root} yields $\alpha \in \widehat
\alpha + \ZZ \lbrace F(\widehat \alpha) \backslash \lbrace \widehat
\alpha \rbrace \rbrace$. Since $\tau(\beta) \ne \tau(\widehat
\alpha)$ for all $\beta \in F(\widehat \alpha) \backslash \lbrace
\widehat \alpha \rbrace$ (see~\cite[Lemma~5(a)]{Avd_solv}),
by~(\ref{eqn_values_of_q_varphi}) we obtain $\langle
q_{\varphi[\alpha]}, \alpha \rangle = 1$. Let $\gamma \in \Pi_0$ be
such that $q_{\varphi[\alpha]} = \rho_\eta(\gamma)$. Then
formula~(\ref{eqn_rho_alpha}) yields $\eta(\gamma,\alpha) = 1$,
whence $\rho_\eta(\gamma) = \rho_\eta(\alpha)$.
\end{proof}

\begin{theorem}
Suppose that $\alpha, \beta \in \Pi$. Then
$$
\eta(\alpha, \beta) = \begin{cases} J(\varphi[\alpha], \beta) &
\text{if $\alpha, \beta \in \Pi_0$};
\\ 0 & \text{otherwise}.\end{cases}
$$
\end{theorem}

\begin{proof}
Fix $\alpha, \beta \in \Pi_0$ and write $\tau(\beta) = \sum
\limits_{\varphi \in \Phi} J(\varphi, \beta) \varphi$ according
to~(\ref{eqn_tau(mu)}). In view of formula~(\ref{eqn_rho_alpha}),
Lemma~\ref{lemma_fans_coincide}, and Lemma~\ref{lemma_q_varphi} one
has
$$
\eta(\alpha, \beta) = \langle \rho_\eta(\alpha), \beta \rangle =
\langle q_{\varphi[\alpha]}, \beta \rangle = J(\varphi[\alpha],
\beta).
$$
To complete the proof, it remains to notice that $\eta(\alpha,
\beta) = 0$ whenever at least one of the roots $\alpha, \beta$ is
not in $\mathfrak X = \ZZ \Pi_0$.
\end{proof}

Our next goal is to establish the converse part of the relationship
between the two approaches. The following lemma plays a key role in
that.

\begin{lemma} \label{lemma_crucial}
Let $\alpha$ be an active root and let $\beta$ be a linear
combination of simple roots in~$\Pi_0$ with non-negative integer
coefficients. Suppose that $\tau(\beta) = \tau(\alpha)$. Then
$\beta$ is an active root.
\end{lemma}

\begin{proof}
For $\beta = \alpha$ there is nothing to prove, therefore we assume
$\beta \ne \alpha$. By condition~(\ref{TT}) there are maximal active
roots $\alpha_1, \beta_1, \ldots, \alpha_k, \beta_k$ and integers
$p_1, \ldots, p_k$ such that
$$
\tau(\alpha_1) = \tau(\beta_1), \ldots, \tau(\alpha_k) =
\tau(\beta_k)
$$
and
$$
\beta = \alpha + p_1(\alpha_1 - \beta_1) + \ldots + p_k(\alpha_k -
\beta_k).
$$
Without loss of generality we may assume that $p_i > 0$ for all $i$
and $\alpha_i \ne \beta_j$ for all $i,j$. By condition~(\ref{CC}),
for every $i = 1, \ldots, k$ there is a simple root $\gamma_i \in
\Supp \beta_i$ that is contained in none of $\Supp \alpha_j$ and
none of $\Supp \beta_j$ with $\beta_j \ne \beta_i$. This implies
$\gamma_i \in \Supp \alpha$, whence $\alpha \in F(\beta_i)$ for all
$i = 1, \ldots, k$. Further, for $\beta_i \ne \beta_j$ one has
$\gamma_i \in \Supp \alpha \subset \Supp \beta_j$, which contradicts
the condition $\gamma_i \notin \Supp \beta_j$. It follows that
$\beta_1 = \ldots = \beta_k$ and $p_1 + \ldots + p_k = 1$, whence $k
= 1$, $p_1 = 1$, and $\beta = \alpha + \alpha_1 - \beta_1 = \alpha_1
- \beta'_1$, where $\beta'_1 = \beta_1 - \alpha$. Since $\alpha \in
F(\beta_1)$, one has either $\beta'_1 = 0$ or $\beta'_1 \in \Delta^+
\backslash \Psi$. The first case yields $\beta = \alpha_1 \in \Psi$.
In the second case we obtain $\Supp \beta'_1 \subset \Supp
\alpha_1$. Since $\tau(\alpha_1) = \tau(\beta_1)$, one of conditions
(\ref{E1}) or (\ref{E2}) holds for $\alpha_1$ and $\beta_1$. A
simple analysis then yields $\beta = \alpha_1 - \beta'_1 \in \Psi$.
\end{proof}

\begin{remark}
The proof of Lemma~\ref{lemma_crucial} is valid for arbitrary (not
necessarily wonderful) strongly solvable spherical subgroup $H
\subset G$ standardly embedded in~$B^-$.
\end{remark}

We now recover the pair $(\Psi, \sim)$ from the admissible
map~$\eta$.

\begin{theorem} \label{thm_ARS_via_AM}
The following assertions hold:
\begin{enumerate}[label=\textup{(\alph*)},ref=\textup{\alph*}]
\item \label{thm_ARS_via_AM_a}
the map $\Psi_\varphi \mapsto q_{\varphi}$ is a bijection between
the equivalence classes of\, $\Psi$ and the set
$\rho_\eta(\Pi_\eta)$;

\item \label{thm_ARS_via_AM_b}
for a fixed $\varrho_0 \in \rho_\eta(\Pi_\eta)$, an element $\alpha
\in \ZZ^+ \Pi_\eta$ is an active root in the corresponding
equivalence class if and only if it satisfies the following system
of linear equations:
$$
\begin{cases}
\langle \varrho_0, \alpha \rangle = 1; & \\
\langle \varrho, \alpha \rangle = 0 & \text{for all } \varrho \in
\rho_\eta(\Pi_\eta) \backslash \lbrace \varrho_0 \rbrace.
\end{cases}
$$
\end{enumerate}
\end{theorem}

\begin{proof}
(\ref{thm_ARS_via_AM_a}) Obvious.

(\ref{thm_ARS_via_AM_b}) Since $\Pi_0 = \Pi_\eta$, every active root
is contained in $\ZZ^+ \Pi_\eta$. By Lemma~\ref{lemma_q_varphi}, for
every $\varphi \in \Phi$ an element $\alpha \in \ZZ \Pi_\eta$
satisfies $\tau(\alpha) = \varphi$ if and only if it satisfies
equalities~(\ref{eqn_values_of_q_varphi}). By
Lemma~\ref{lemma_crucial}, every element $\alpha \in \ZZ \Pi_\eta$
with $\tau(\alpha) = \varphi$ is in fact an active root.
\end{proof}

\begin{remark}
Proposition~\ref{prop_AM_via_SSSS} combined with
Theorem~\ref{thm_ARS_via_AM} provide a method for determining
explicitly a strongly solvable wonderful subgroup starting from its
spherical system. This method was suggested by D.~Luna as a
conjecture in a private note addressed to the author.
\end{remark}

\begin{example}
Suppose that $\eta(\alpha, \beta) = 0$ for all $\alpha, \beta \in
\Pi$. Then $\Pi_\eta = \varnothing$, hence $\Psi = \varnothing$ and
the corresponding strongly solvable wonderful subgroup of~$G$ is
just~$B^-$.
\end{example}

\begin{example}
Suppose that
$$
\eta(\alpha, \beta) =
\begin{cases}
1 & \text{if } \alpha = \beta; \\
0 & \text{otherwise.}
\end{cases}
$$
Then $\Pi_\eta = \Pi$, $\rho_\eta(\Pi_\eta) = \lbrace \breve \alpha
\mid \alpha \in \Pi \rbrace$, and $\Psi = \Pi$ with pairwise
non-equivalent roots. The corresponding strongly solvable wonderful
subgroup of~$G$ is $T \rightthreetimes (U^-, U^-)$.
\end{example}

Other examples for $G$ of small rank can be found in
Appendix~\ref{sect_lists}.

\subsection{Computation of extended ARS-sets via homogeneous
spherical data} \label{subsec_ext_ARS_via_HSD}

In this subsection we assume that $G = C \times G^{ss}$ and $G^{ss}$
is simply connected.

Let $\mathscr H = (\Lambda, \Pi^p, \Sigma, \mathcal D^a)$ be a
strongly solvable homogeneous spherical datum and let $H \subset G$
be a spherical subgroup with $\mathscr H_{G / H} = \mathscr H$. By
Proposition~\ref{prop_sph_clos_in_parabolic}, the spherical closure
$\overline H$ of~$H$ is strongly solvable as well.
Corollary~\ref{crl_strongly_solvable_HSD} yields $\Pi^p =
\varnothing$ and $\Sigma \subset \Pi$, which implies $\mathscr S_{G
/ \overline H} = (\Pi^p, \Sigma, \mathcal D^a)$ (with $\varkappa$
restricted to $\Hom_\ZZ(\ZZ\Sigma, \ZZ)$) by
Proposition~\ref{prop_sph_syst_of_sph_closure}.

Since $\overline H$ is strongly solvable, there exists a subset
$\mathcal D'_{G/\overline H} \subset \mathcal D^a_{G/\overline H}$
satisfying the conditions of
Proposition~\ref{prop_strongly_solvable}. This subset is the
distinguished subset of colors of a uniquely determined
$G$-equivariant morphism $G / \overline H \to G / B^-$. Hence we
have a natural $G$-equivariant morphism $G / H \to G / B^-$ and may
assume $H \subset B^-$. Moreover, it may be also assumed that $H$ is
standardly embedded in~$B^-$, so that $H = S \rightthreetimes N$,
where $S \subset T$ and $N \subset U^-$. Let $\tau, \Phi, \Psi,
\ldots$ be as in \S\,\ref{subsec_solv_description}.

Let $\mathcal D'_{G/H} \subset \mathcal D^a_{G/H}$ be the subset
corresponding to $\mathcal D'_{G/\overline H}$ under the natural
bijection between $\mathcal D_{G/H}$ and $\mathcal D_{G/\overline
H}$. Then the set $\mathcal D^\circ = \mathcal D^\circ_{G / H} =
\mathcal D_{G/H} \backslash \mathcal D'_{G/H}$ contains exactly
$|\Pi|$ elements. More precisely, for every $\alpha \in \Pi$ the set
$\mathcal D_{G/H}(\alpha) \cap \mathcal D^\circ$ contains exactly
one element, we denote it by~$D_\alpha$. Theorem~\ref{thm_solv_ews}
and Remark~\ref{rem_division_of_colors} imply that for every $\alpha
\in \Pi$ one has $(\lambda_{D_\alpha}, \chi_{D_\alpha}) =
(\varpi_\alpha, - \tau(\varpi_\alpha))$.

We recall that in \S\,\ref{subsec_ews_pci_interrelations} we
introduced the notation $\ZZ^{\mathcal D}$ for the free Abelian
group generated by the set $\mathcal D = \mathcal D_{G/H}$. By
Proposition~\ref{prop_gen_char}, the homomorphism
$$
\psi \colon \ZZ^{\mathcal D} \oplus \mathfrak X(C) \to \mathfrak
X(H) \simeq \mathfrak X(S),
$$
given by $D \mapsto \chi_D$ for every $D \in \mathcal D_{G/H}$ and
$\nu \mapsto \tau(\nu)$ for every $\nu \in \mathfrak X(C)$, is
surjective. Proposition~\ref{prop_kernel} says that $\Ker \psi$ is
generated by the elements $\sum \limits_{D \in \mathcal D} \langle
\varkappa(D), \mu \rangle D - \mu^C$, where $\mu$ runs over a basis
of~$\Lambda_{G/H}$.

Let $\ZZ^{\mathcal D^\circ} \subset \ZZ^{\mathcal D}$ be the
subgroup generated by the set $\mathcal D^\circ$. We identify
$\ZZ^{\mathcal D^\circ}$ with $\mathfrak X(T^{ss})$ via the
isomorphism given by $D_\alpha \mapsto - \varpi_\alpha$. Thus the
group $\mathfrak X(T)$ is identified with $\ZZ^{\mathcal D^\circ}
\oplus \mathfrak X(C)$. Modulo this identification, the subgroup $S
\subset T$ is recovered as follows.

\begin{proposition}
One has $\Ker \tau = (\ZZ^{\mathcal D^\circ} \oplus \mathfrak X(C))
\cap \Ker \psi$.
\end{proposition}

\begin{proof}
Since $\chi_{D_\alpha} = - \tau(\varpi_\alpha)$, the restriction of
the map $\psi$ to $\ZZ^{\mathcal D^\circ} \oplus \mathfrak X(C)$ is
surjective.
\end{proof}

For every $D \in \mathcal D'_{G/H}$, we set $\Pi_D = \lbrace \alpha
\in \Pi \mid \langle \varkappa(D), \alpha \rangle = 1 \rbrace$.

\begin{theorem} \label{thm_H_via_PCI}
The following assertions hold:
\begin{enumerate}[label=\textup{(\alph*)},ref=\textup{\alph*}]
\item \label{thm_H_via_PCI_a}
$\Phi = \lbrace \psi(D - \sum \limits_{\alpha \in \Pi_D} D_\alpha)
\mid D \in \mathcal D'_{G/H} \rbrace \subset \mathfrak X(S)$;

\item \label{thm_H_via_PCI_b}
for a given $\varphi \in \Phi$, one has $\Psi_\varphi = \lbrace
\alpha \in \ZZ^+\Sigma \mid \tau(\alpha) = \varphi \rbrace$.
\end{enumerate}
\end{theorem}

\begin{proof}
The set $\Phi$ is in natural bijection with the set $\mathcal
D'_{G/H}$. Under this bijection, a weight $\varphi \in \Phi$
corresponds to the color $D_\varphi$ such that
$(\lambda_{D_\varphi}, \chi_{D_\varphi}) = (\lambda_\varphi, -
\tau(\lambda_\varphi) + \varphi)$; see Theorem~\ref{thm_solv_ews}.
Evidently, $\varphi = \psi(D_\varphi - \sum \limits_{\alpha \in \Pi
: \varpi_\alpha \in \supp \lambda_\varphi} D_\alpha)$. To prove
part~(\ref{thm_H_via_PCI_a}) it remains to notice that $\supp
\lambda_\varphi = \lbrace \varpi_\alpha \mid \alpha \in
\Pi_{D_\varphi} \rbrace$. Part~(\ref{thm_H_via_PCI_b}) is a direct
consequence of Lemma~\ref{lemma_crucial}.
\end{proof}

\section{Possible generalizations of Luna's 1993 classification
and the explicit classification} \label{sect_generalizations}

A natural question arising in connection with Luna's 1993
classification and the explicit classification is whether it is
possible to generalize them to the case of arbitrary spherical
subgroups. In this section we present some ideas on this question.

Let $H \subset G$ be a subgroup and let $H_u$ be its unipotent
radical. It is well known (see, for instance,~\cite[\S\,30.3]{Hum})
that there exists a parabolic subgroup $P \subset G$ such that $H
\subset P$ and $H_u \subset P_u$, where $P_u$ is the unipotent
radical of~$P$. Accordingly, one can pose a problem of classifying
all spherical subgroups $H \subset G$ contained in a fixed parabolic
subgroup $P \subset G$ so that $H_u \subset P_u$. We note the
following two ``opposite'' particular cases of this problem:
\begin{itemize}
\item
$P = G$, which implies that $H$ is reductive;

\item
$P$ is a Borel subgroup of~$G$, which is equivalent to $H$ being
strongly solvable.
\end{itemize}

In \S\,\ref{subsec_generalization_Luna_1993} (resp.
\S\,\ref{subsec_generalization_expl_class}) we discuss possibilities
of generalizing Luna's 1993 classification (resp. the explicit
classification) in order to solve the indicated problem in the case
where $P$ is any proper parabolic subgroup of~$G$.

\subsection{The case of Luna's 1993 classification}
\label{subsec_generalization_Luna_1993}

Let $P \subset G$ be a proper parabolic subgroup and let $L$ be a
Levi subgroup of~$P$. Fix a Borel subgroup $B_L$ of~$L$.

The two definitions below are direct generalizations of
Definitions~\ref{def_spher_B-var} and~\ref{def_wond_B-var}.

\begin{definition}
A normal irreducible $P$-variety $Z$ is said to be
\textit{spherical} if $B_L$ has an open orbit in~$Z$, that is, $Z$
is a spherical $L$-variety.
\end{definition}

\begin{definition}
A spherical $P$-variety $Z$ is said to be \textit{wonderful} if it
possesses the following properties:
\begin{enumerate}[label=\textup{(WP\arabic*)},ref=\textup{WP\arabic*}]
\item
$Z$ is smooth and complete;

\item
there is exactly one closed $P$-orbit $Z_0$ in~$Z$;

\item
every irreducible $B_L$-stable closed subvariety $Z' \subset Z$
containing $Z_0$ is actually $P$-stable.
\end{enumerate}
\end{definition}

Generalizing the proof of Proposition~\ref{prop_wonderful_B&G} one
can prove the following result.

\begin{proposition}
Let $Z$ be a $P$-variety and consider the $G$-variety $X = G *_P Z$.
\begin{enumerate}[label=\textup{(\alph*)},ref=\textup{\alph*}]
\item $Z$ is a spherical $P$-variety if and only if $X$ is a
spherical $G$-variety;

\item $Z$ is a wonderful $P$-variety if and only if $X$ is a
wonderful $G$-variety.
\end{enumerate}
\end{proposition}

Thus the classification of wonderful subgroups $H \subset G$ such
that $H \subset P$ and $H_u \subset P_u$ reduces to that of
wonderful $P$-varieties. The latter would be possible if there were
a description of connected automorphism groups of smooth complete
spherical $L$-varieties. Unfortunately, at the moment no such
description exists except for the case where $L$ is a maximal torus
of~$G$, which yields a classification of wonderful $B^-$-varieties.

\subsection{The case of the explicit classification}
\label{subsec_generalization_expl_class}

Before we start our discussion, we need to introduce the notion of a
generalized root.

Let $Q \subset G$ be a parabolic subgroup and let $M$ be a Levi
subgroup of~$Q$. We denote by~$C_M$ the center of~$M$. Let
$\Delta_M$ be the set of nonzero weights for the natural action of
$C_M$ on~$\mathfrak g$. In the paper~\cite{Ko} Kostant undertook a
detailed study of the set~$\Delta_M$ and found that it has many
properties in common with the usual set of roots. For this reason we
refer to elements in $\Delta_M$ as ``generalized roots'' (or just
``$C_M$-roots''). It is well known that for every $\nu \in \Delta_M$
the representation of $M$ on the corresponding weight subspace
$\mathfrak g(\nu) \subset \mathfrak g$ is irreducible. (The latter
is a direct consequence of~\cite[Lemma~3.9]{GOV}; see
also~\cite[Theorem~0.1]{Ko}.)

We now take $P$ and $L$ to be as
in~\S\,\ref{subsec_generalization_Luna_1993}, so that $P = L
\rightthreetimes P_u$. Let $H \subset P$ be a subgroup satisfying
$H_u \subset P_u$. Replacing $H$ by a conjugate subgroup, we may
assume that $K = H \cap L$ is a Levi subgroup of~$H$, so that $H = K
\rightthreetimes H_u$. In this situation, we say that $H$ is
\textit{standardly embedded in~$P$} (with respect to~$L$). Let $S
\subset K$ be a generic stabilizer of the natural action of $K$ on
$\mathfrak l / \mathfrak k$. The subgroup $S$ is reductive;
see~\cite[Corollary~8.2]{Kn90} or~\cite[Theorem~3(ii) and
\S\,2.1]{Pa90}. By \cite[Theorem~1.2(i)]{Pa94}, $H$ being spherical
is equivalent to the following two conditions holding
simultaneously:

\begin{itemize}
\item
$K$ is a spherical subgroup of~$L$;

\item
$\mathfrak p_u / \mathfrak h_u$ is a spherical $S$-module (that is,
$\mathfrak p_u / \mathfrak h_u$ is spherical as an $S$-variety).
\end{itemize}
We note that this sphericity criterion is very useful in practice
for checking whether a given subgroup $H$ standardly embedded in~$P$
is spherical in~$G$. Indeed, a computation of the group $S$ for a
given reductive spherical subgroup $K \subset L$ easily reduces to
the cases listed in Tables~4 and~5 of the paper~\cite{KnoV}.
Besides, there is a complete classification of spherical modules of
reductive groups, which was obtained in~\cite{Kac} (the case of a
simple module), \cite{BenR}, and~\cite{Lea} (the latter two papers
dealt independently with the general case).

Assume that a spherical subgroup $H \subset P$ standardly embedded
in~$P$ is fixed. Results of Knop (see~\cite[Corollary~8.2]{Kn90})
and Panyushev (see~\cite[Theorem~1(iii), Theorem~3(ii), and
\S\,2.1]{Pa90}) imply that there is a parabolic subgroup $P_L$ of
$L$ with a Levi subgroup $M$ such that $(M, M) \subset S \subset M$.
Then $Q = P_L \rightthreetimes P_u$ is a parabolic subgroup of $G$
and $M$ is a Levi subgroup of $Q$. Let $C_M$ denote the center
of~$M$ and let $\Delta_M$ be the corresponding set of generalized
roots. We denote by $\Delta^+_M \subset \Delta_M$ the set of weights
for the natural action of $C_M$ on~$\mathfrak p_u$. Put
$$
\Psi_{H, S} = \lbrace \nu \in \Delta^+_M \mid \mathfrak g(\nu)
\not\subset \mathfrak h_u \rbrace.
$$
The set $\Psi_{H, S}$ is a direct analogue of the set of active
roots appearing in the explicit classification. Accordingly,
generalized roots in $\Psi_{H, S}$ will be called
\textit{generalized active roots}.

The following conjecture is a direct generalization of
Theorem~\ref{thm_first_approx}.

\begin{conjecture} \label{conj_first_approx}
Up to conjugation by an element of $N_L(K)$, a spherical subgroup $H
\subset G$ standardly embedded in~$P$ is uniquely determined by the
triple $(K,S, \Psi_{H,S})$.
\end{conjecture}

If this conjecture is true, then it can serve as a first
approximation to an explicit classification (that is, in terms of
Lie algebras) of arbitrary spherical subgroups in reductive groups,
just like Theorem~\ref{thm_first_approx} serves as a first
approximation to the explicit classification of strongly solvable
spherical subgroups. It seems to the author that, in contrast to the
strongly solvable case, a proof of the conjecture in full generality
has to be based on a large number of case-by-case considerations.

The importance of Conjecture~\ref{conj_first_approx} (or perhaps of
a similar result) also consists in the fact that it provides
``normal forms'' for spherical subgroups, which creates a background
for solving problems~(\ref{P1}) and~(\ref{P2}) in a way similar to
that in the strongly solvable case. More precisely, it seems to be
possible to compute the principal combinatorial invariants of a
spherical subgroup given by a normal form of it, as well as to find
a normal form of a spherical subgroup given by its principal
combinatorial invariants (or by its homogeneous spherical datum).

\appendix

\section{Homogeneous bundles} \label{sect_homogeneous_bundles}

In this appendix we recall the construction of a homogeneous bundle.
Let $L$ be a group, $K$ a subgroup of~$L$, and $Z$ an arbitrary
$K$-variety. By definition, the homogeneous bundle $L *_K Z$ over $L
/ K$ associated with $Z$ is the quotient set of $L \times Z$ by the
action of $K$ given by the formula $k (l, z) = (lk^{-1}, kz)$, where
$k \in K$, $l \in L$, $z \in Z$. Clearly, $L *_K Z$ is equipped with
an action of $L$ induced by the natural action on $L \times Z$ by
left translation of the first factor. Moreover, there is a natural
$L$-equivariant map $L *_K Z \to L / K$, which justifies the term
``homogeneous bundle over~$L / K$''.

\begin{theorem}[see {\cite[Corollary~2]{Bia}, \cite[Theorem~4.19]{PV}}]
\label{thm_homogeneous_bundles} If $Z$ is covered by $K$-stable
quasiprojective open subsets, then the set $L *_K Z$ has the
structure of an algebraic variety such that the quotient map $L
\times Z \to L *_K Z$ by the action of $K$ is a geometric quotient.
\end{theorem}

The assumptions of Theorem~\ref{thm_homogeneous_bundles} are
fulfilled in case of connected $K$ and normal $Z$
(see~\cite[Lemma~8]{Sum}) or in case of quasiprojective~$Z$. This
turns out to be enough for all homogeneous bundles considered in
this paper to be algebraic varieties in the indicated sense.

Below we list a few properties of homogeneous bundles.

\begin{proposition}[see {\cite[\S\,2.1]{Tim}}] \label{prop_morphism_to_L/K}
If an $L$-variety $X$ admits an $L$-equivariant morphism $\phi
\colon X \to L / K$, then $X \simeq L *_K Z$, where $Z =
\phi^{-1}(o)$.
\end{proposition}

\begin{proposition}[see {\cite[Proposition~4.22]{PV} or
\cite[\S\,2.1]{Tim}}]
\label{prop_smoothness} The variety $L *_K Z$ is smooth
\textup(resp. normal\textup) if and only if $Z$ is smooth
\textup(resp. normal\textup).
\end{proposition}

\begin{proposition} \label{prop_completeness}
Let $P \subset G$ be a parabolic subgroup and let $Z$ be a
$P$-variety. The variety $G *_P Z$ is complete if and only if $Z$ is
complete.
\end{proposition}

The proof of this proposition provided below was communicated to the
author by D.\,A.~Timashev.

\begin{proof}
We may assume that $P \supset B$. Let $P^-$ be a parabolic subgroup
opposite to $P$ (that is, $P \cap P^-$ is a Levi subgroup of
both~$P$ and~$P^-$) and let $P_u^-$ be the unipotent radical
of~$P^-$. Consider the natural morphism $\phi \colon G *_P Z \to G /
P$. Since $Z \simeq \phi^{-1}(o)$, $Z$ is complete whenever $G *_P
Z$ is so. Conversely, assume that $Z$ is complete. To prove that $G
*_P Z$ is complete, by definition one has to show that for every
algebraic variety $W$ the projection morphism $(G *_P Z) \times W
\to W$ is closed. Since $G / P$ is complete, it suffices to show
that the morphism
$$
(G *_P Z) \times W \to (G / P) \times W
$$
extending $\phi$ is closed. As $G / P$ is covered by a finite number
of shifts of the open subset $P^-o = P_u^-o \simeq P_u^-$, the
problem reduces to showing that the morphism
$$
\phi^{-1}(P_u^-o) \times W \to P_u^-o \times W
$$
is closed. But $\phi^{-1}(P_u^-o) \simeq P_u^- \times Z$, and so the
required result follows from the completeness of~$Z$.
\end{proof}

For a more detailed discussion of the construction of a homogeneous
bundle see~\cite{Bia}, \cite[\S\,4.8]{PV}, or~\cite[\S\,2.1]{Tim}.

\section{Strongly solvable wonderful subgroups in small rank}
\label{sect_lists}

This appendix is illustrative. Here we list all strongly solvable
spherical systems for all semisimple groups $G$ of rank at most~$2$
and also for any simple group of type~$\mathsf A_3$. Since every
strongly solvable wonderful subgroup $H \subset G$ is spherically
closed (see Corollary~\ref{crl_str_solv_wond_sph_clos}), $H$
contains the center of~$G$. Therefore our lists depend only on the
Dynkin diagram of~$G$. We recall that by
Corollary~\ref{crl_strongly_solvable_SS} every strongly solvable
spherical system has the form $(\varnothing, \Pi', \mathcal D^a)$
for some subset $\Pi' \subset \Pi$.

For every strongly solvable spherical system, we indicate all
distinguished subsets of colors satisfying the conditions of
Proposition~\ref{prop_strongly_solvable}. For every such subset of
colors, we also indicate the corresponding admissible map and the
corresponding set of active roots divided into equivalence classes.

It turns out that in all cases it is enough to list all strongly
solvable spherical systems $(\varnothing, \Pi', \mathcal D^a)$ with
$\Pi' = \Pi$ (according to the general theory, these are called
\textit{cuspidal}; see~\cite[\S\,3.4]{Lu01}), since all other
strongly solvable spherical systems (together with all possible
admissible maps and sets of active roots) naturally come from
strongly solvable spherical systems corresponding to proper
subdiagrams of the Dynkin diagram of~$\Pi$. We note that cuspidal
strongly solvable spherical systems are uniquely determined by the
set~$\mathcal D^a$, and in this case the set $\mathcal D$ of colors
of the corresponding wonderful $G$-variety coincides with~$\mathcal
D^a$.

For every~$G$, the only strongly solvable spherical system
$(\varnothing, \Pi', \mathcal D^a)$ with $\Pi' = \varnothing$
coincides with $(\varnothing, \varnothing, \varnothing)$ and
corresponds to the wonderful variety $G / B^-$ (see
Example~\ref{example_hom_wond}). In this case, the admissible map
vanishes and the set of active roots is empty.

If $\rk G = 1$, then the only cuspidal strongly solvable spherical
system has the form $(\varnothing, \lbrace \alpha \rbrace, \lbrace
D^+, D^-\rbrace)$, where $\alpha$ is the unique simple root of~$G$
and $\langle \varkappa(D^+), \alpha \rangle = \langle
\varkappa(D^-), \alpha \rangle = 1$. The subsets $\lbrace D^+
\rbrace, \lbrace D^- \rbrace \subset \mathcal D^a$ are the only
distinguished subsets of colors satisfying the conditions of
Proposition~\ref{prop_strongly_solvable}, both determine the same
admissible map $\eta$ given by $\eta(\alpha, \alpha) = 1$ and the
same set of active roots, which is just $\lbrace \alpha \rbrace$.

If $\rk G = 2$, then, depending on the type of the Dynkin diagram
of~$\Pi$, the information about all cuspidal strongly solvable
spherical systems is presented in Table~\ref{table_A1timesA1}
(type~$\mathsf A_1 \times \mathsf A_1$), Table~\ref{table_A2}
(type~$\mathsf A_2$), Table~\ref{table_B2} (type~$\mathsf B_2$), and
Table~\ref{table_G2} (type~$\mathsf G_2$).

At last, the information about all cuspidal strongly solvable
spherical systems in the case where $G$ is of type~$\mathsf A_3$ is
presented in Table~\ref{table_A3}.

The notation used in Tables~\ref{table_A1timesA1}--\ref{table_A3} is
as follows. If $G$ is simple, then $\alpha_i$ denotes the $i$th
simple root of~$G$. If $G$ is of type $\mathsf A_1 \times \mathsf
A_1$, then $\alpha_1, \alpha_2$ denote the two different simple
roots of~$G$. Each matrix in the column ``$\mathcal D^a$''
represents the set $\mathcal D^a$ of a cuspidal strongly solvable
spherical system. Elements of $\mathcal D^a$ are in bijection with
rows of the matrix. For every $D \in \mathcal D^a$, the $i$th
element in the corresponding row is the value $\langle \varkappa(D),
\alpha_i \rangle$. In the column ``DSC'' (the heading is an
abbreviation for ``distinguished subsets of colors'') we list all
subsets of $\mathcal D^a$ satisfying the conditions of
Proposition~\ref{prop_strongly_solvable}. Each subset is given by
numbers of the corresponding rows of the matrix in the column
``$\mathcal D^a$''. For every such subset $\mathcal D'$, in the
column ``Admissible map'' (resp.~``Active roots'') we indicate the
matrix of the corresponding admissible map (resp. the equivalence
classes of active roots) determined by~$\mathcal D'$.

\renewcommand{\tabcolsep}{0pt}

\newcommand{\wl}{2cm}
\newcommand{\wc}{4cm}
\renewcommand{\wr}{4cm}

\vspace{2em}

\begin{table}[h]

\caption{Cuspidal strongly solvable spherical systems in
type~$\mathsf A_1 \times \mathsf A_1$} \label{table_A1timesA1}

\begin{tabular}{|c|c|c|}
\hline No. & $\mathcal D^a$ &
\begin{tabular}{p{\wl}|p{\wc}|c}
\centering{DSC} & \centering{Admissible map} &
\begin{tabular}{p{\wr}}\centering{Active roots}\end{tabular} \\
\end{tabular}\\

\hline

1 & $\begin{bmatrix} 1 & 0 \\ 1 & 0 \\ 0 & 1 \\ 0 & 1\end{bmatrix}$
&

\begin{tabular}{p{\wl}|p{\wc}|c}

\centering{\begin{tabular}{c}$1,3\lefteqn{\text{ or}}$ \\
$1,4\lefteqn{\text{ or}}$ \\ $2,3\lefteqn{\text{ or}}$ \\
$2,4$ \end{tabular}} & \centering{$\begin{bmatrix} 1 & 0 \\ 0 & 1
\end{bmatrix}$} &
\begin{tabular}{p{\wr}}\centering{$\lbrace \alpha_1 \rbrace,
\lbrace \alpha_2 \rbrace$} \end{tabular} \\

\end{tabular}\\

\hline

2 & $\begin{bmatrix} 1 & 0 \\ 0 & 1 \\ 1 & 1 \end{bmatrix}$ &

\begin{tabular}{p{\wl}|p{\wc}|c}
\centering{$\vphantom{\begin{bmatrix} 1 & 0 \\ 0 & 1
\\ 1 & 1 \end{bmatrix}}3$} & \centering{$\begin{bmatrix} 1 & 1 \\
1 & 1
\end{bmatrix}$} &
\begin{tabular}{p{\wr}}
\centering{$\lbrace \alpha_1, \alpha_2 \rbrace$}\end{tabular}\\

\end{tabular}\\

\hline
\end{tabular}
\end{table}



\begin{table}[!h]

\caption{Cuspidal strongly solvable spherical systems in
type~$\mathsf A_2$} \label{table_A2}

\begin{tabular}{|c|c|c|}
\hline No. & $\mathcal D^a$ &
\begin{tabular}{p{\wl}|p{\wc}|c}
\centering{DSC} & \centering{Admissible map} &
\begin{tabular}{p{\wr}}\centering{Active roots}\end{tabular} \\
\end{tabular}\\

\hline

1 & $\begin{bmatrix} 1 & 0 \\ 0 & 1 \\ 1 & -1 \\ -1 &
1\end{bmatrix}$ &

\begin{tabular}{p{\wl}|p{\wc}|c}

\centering{$1,2$} & \centering{$\begin{bmatrix} 1 & 0 \\ 0 & 1
\end{bmatrix}$} &
\begin{tabular}{p{\wr}}\centering{$\lbrace \alpha_1 \rbrace,
\lbrace \alpha_2 \rbrace$} \end{tabular} \\

\hline

\centering{$1,4$} & \centering{$\begin{bmatrix} 1 & 0 \\
-1 & 1 \end{bmatrix}$} &
\begin{tabular}{p{\wr}} \centering{$\lbrace \alpha_1 {+} \alpha_2
\rbrace, \lbrace \alpha_2 \rbrace$} \end{tabular}\\

\hline

\centering{$2,3$} & \centering{$\begin{bmatrix} 1 & -1 \\ 0 & 1
\end{bmatrix}$} &
\begin{tabular}{p{\wr}} \centering{$\lbrace \alpha_1 \rbrace,
\lbrace \alpha_1 {+} \alpha_2 \rbrace$} \end{tabular}\\

\end{tabular}\\

\hline

2 & $\begin{bmatrix} 1 & 1 \\ 1 & -2 \\ -2 & 1 \end{bmatrix}$ &

\begin{tabular}{p{\wl}|p{\wc}|c}
\centering{$\vphantom{\begin{bmatrix} 1 & 1 \\ 1 & -2
\\ -2 & 1 \end{bmatrix}}1$} & \centering{$\begin{bmatrix} 1 & 1 \\
1 & 1
\end{bmatrix}$} &
\begin{tabular}{p{\wr}}
\centering{$\lbrace \alpha_1, \alpha_2 \rbrace$}\end{tabular}\\

\end{tabular}\\

\hline
\end{tabular}
\end{table}



\begin{table}[!h]

\caption{Cuspidal strongly solvable spherical systems in
type~$\mathsf B_2$} \label{table_B2}

\begin{tabular}{|c|c|c|}
\hline No. & $\mathcal D^a$ &
\begin{tabular}{p{\wl}|p{\wc}|c}
\centering{DSC} & \centering{Admissible map} &
\begin{tabular}{p{\wr}}\centering{Active roots}\end{tabular} \\
\end{tabular}\\

\hline

1 & $\begin{bmatrix} 1 & 0 \\ 0 & 1 \\ 1 & -1 \\ -2 & 1
\end{bmatrix}$ &

\begin{tabular}{p{\wl}|p{\wc}|c}

\centering{$1,2$} & \centering{$\begin{bmatrix} 1 & 0 \\ 0 & 1
\end{bmatrix}$} &
\begin{tabular}{p{\wr}}\centering{$\lbrace \alpha_1 \rbrace,
\lbrace \alpha_2 \rbrace$} \end{tabular} \\

\hline

\centering{$1,4$} & \centering{$\begin{bmatrix} 1 & 0 \\
-2 & 1 \end{bmatrix}$} &
\begin{tabular}{p{\wr}}\centering{$\lbrace \alpha_1 {+}
2\alpha_2 \rbrace, \lbrace \alpha_2 \rbrace$} \end{tabular} \\

\hline

\centering{$2,3$} & \centering{$\begin{bmatrix} 1 & -1 \\ 0 & 1
\end{bmatrix}$} &
\begin{tabular}{p{\wr}}
\centering{$\lbrace \alpha_1 \rbrace, \lbrace \alpha_1 {+}
\alpha_2 \rbrace$} \end{tabular}\\

\end{tabular}\\

\hline

2 & $\begin{bmatrix} 1 & 0 \\ -1 & 1 \\ 1 & -1 \\ -1 & 1
\end{bmatrix}$ &

\begin{tabular}{p{\wl}|p{\wc}|c}

\centering{\begin{tabular}{c}$1,2$\\ or \\ $1,4$ \end{tabular}} &
\centering{$\vphantom{\begin{bmatrix} 1 & 0 \\ -1 & 1 \\ 1 & -1 \\
-1 & 1 \end{bmatrix}}
\begin{bmatrix} 1 & 0 \\
-1 & 1
\end{bmatrix}$} &
\begin{tabular}{p{\wr}}\centering{$\lbrace \alpha_1 {+} \alpha_2 \rbrace,
\lbrace \alpha_2 \rbrace$} \end{tabular} \\

\end{tabular}\\

\hline

3 & $\begin{bmatrix} 1 & 1 \\ 1 & -2 \\ -3 & 1 \end{bmatrix}$ &

\begin{tabular}{p{\wl}|p{\wc}|c}
\centering{$\vphantom{\begin{bmatrix} 1 & 1 \\ 1 & -2
\\ -2 & 1 \end{bmatrix}}1$} & \centering{$\begin{bmatrix} 1 & 1 \\
1 & 1 \end{bmatrix}$} &
\begin{tabular}{p{\wr}}
\centering{$\lbrace \alpha_1, \alpha_2 \rbrace$}\end{tabular}\\

\end{tabular}\\

\hline
\end{tabular}
\end{table}


\begin{table}[!h]

\caption{Cuspidal strongly solvable spherical systems in
type~$\mathsf G_2$} \label{table_G2}

\begin{tabular}{|c|c|c|}
\hline No. & $\mathcal D^a$ &
\begin{tabular}{p{\wl}|p{\wc}|c}
\centering{DSC} & \centering{Admissible map} &
\begin{tabular}{p{\wr}}\centering{Active roots}\end{tabular} \\
\end{tabular}\\

\hline

1 & $\begin{bmatrix} 1 & 0 \\ 0 & 1 \\ 1 & -3 \\ -1 & 1
\end{bmatrix}$ &

\begin{tabular}{p{\wl}|p{\wc}|c}

\centering{$1,2$} & \centering{$\begin{bmatrix} 1 & 0 \\ 0 & 1
\end{bmatrix}$} &
\begin{tabular}{p{\wr}}\centering{$\lbrace \alpha_1 \rbrace,
\lbrace \alpha_2 \rbrace$} \end{tabular} \\

\hline

\centering{$1,4$} & \centering{$\begin{bmatrix} 1 & 0 \\ -1 & 1
\end{bmatrix}$} &
\begin{tabular}{p{4cm}}
\centering{$\lbrace \alpha_1 {+} \alpha_2 \rbrace, \lbrace \alpha_2 \rbrace$} \end{tabular}\\

\hline

\centering{$2,3$} & \centering{$\begin{bmatrix} 1 & -3 \\
0 & 1 \end{bmatrix}$} &
\begin{tabular}{p{\wr}}\centering{$\lbrace \alpha_1 \rbrace,
\lbrace 3\alpha_1 {+} \alpha_2 \rbrace$} \end{tabular} \\

\end{tabular}\\

\hline

2 & $\begin{bmatrix} 1 & -1 \\ 0 & 1 \\ 1 & -2 \\ -1 & 1
\end{bmatrix}$ &

\begin{tabular}{p{\wl}|p{\wc}|c}

\centering{$1,2$} & \centering{$\begin{bmatrix} 1 & -1 \\
0 & 1
\end{bmatrix}$} &
\begin{tabular}{p{\wr}}\centering{$\lbrace \alpha_1 \rbrace,
\lbrace \alpha_1 {+} \alpha_2 \rbrace$} \end{tabular} \\

\hline

\centering{$2,3$} & \centering{$\begin{bmatrix} 1 & -2 \\ 0 & 1
\end{bmatrix}$} &
\begin{tabular}{p{\wr}}\centering{$\lbrace \alpha_1 \rbrace,
\lbrace 2\alpha_1 {+} \alpha_2 \rbrace$} \end{tabular} \\

\end{tabular}\\

\hline

3 & $\begin{bmatrix} 1 & 1 \\ 1 & -4 \\ -2 & 1 \end{bmatrix}$ &

\begin{tabular}{p{\wl}|p{\wc}|c}
\centering{$\vphantom{\begin{bmatrix} 1 & 1 \\ 1 & -2
\\ -2 & 1 \end{bmatrix}}1$} & \centering{$\begin{bmatrix} 1 & 1 \\
1 & 1 \end{bmatrix}$} &
\begin{tabular}{p{\wr}}
\centering{$\lbrace \alpha_1, \alpha_2 \rbrace$}\end{tabular}\\

\end{tabular}\\

\hline
\end{tabular}
\end{table}



\begin{longtable}{|c|c|c|}

\caption{Cuspidal strongly solvable spherical systems in
type~$\mathsf A_3$} \label{table_A3} \\

\hline No. & $\mathcal D^a$ &
\begin{tabular}{p{\wl}|p{\wc}|c}
\centering{DSC} & \centering{Admissible map} &
\begin{tabular}{p{\wr}}\centering{Active roots}\end{tabular} \\
\end{tabular}\\ \endfirsthead

\hline No. & $\mathcal D^a$ &
\begin{tabular}{p{\wl}|p{\wc}|c}
\centering{DSC} & \centering{Admissible map} &
\begin{tabular}{p{\wr}}\centering{Active roots}\end{tabular} \\
\end{tabular}\\ \endhead

\hline

1 & $\begin{bmatrix} 1 & 0 & 0 \\ 0 & 1 & 0 \\ 0 & 0 & 1 \\ 1 & -1 &
0 \\ -1 & 1 & -1 \\ 0 & -1 & 1
\end{bmatrix}$ &

\begin{tabular}{p{\wl}|p{\wc}|c}

\centering{$1,2,3$} & \centering{$\begin{bmatrix} 1 & 0 & 0 \\ 0 & 1
& 0 \\ 0 & 0 & 1 \end{bmatrix}$} &
\begin{tabular}{p{\wr}}\centering{$\lbrace \alpha_1 \rbrace$,
$\lbrace \alpha_2 \rbrace$, $\lbrace \alpha_3 \rbrace$} \end{tabular} \\

\hline

\centering{$1,2,6$} & \centering{$\begin{bmatrix} 1 & 0 & 0 \\
0 & 1 & 0 \\ 0 & -1 & 1 \end{bmatrix}$} &
\begin{tabular}{p{\wr}}\centering{$\lbrace \alpha_1 \rbrace$,
$\lbrace \alpha_2 {+} \alpha_3 \rbrace$, $\lbrace
\alpha_3 \rbrace$} \end{tabular} \\

\hline

\centering{$1,3,5$} & \centering{$\begin{bmatrix} 1 & 0 & 0 \\ -1 &
1 & -1 \\ 0 & 0 & 1 \end{bmatrix}$} &
\begin{tabular}{p{4cm}} \centering{$\lbrace \alpha_1 {+} \alpha_2
\rbrace$, $\lbrace \alpha_2 \rbrace$, $\lbrace \alpha_2 {+}
\alpha_3 \rbrace$} \end{tabular}\\

\hline

\centering{$2,3,4$} & \centering{$\begin{bmatrix} 1 & -1 & 0 \\ 0 &
1 & 0 \\ 0 & 0 & 1 \end{bmatrix}$}
&\begin{tabular}{p{\wr}}\centering{$\lbrace \alpha_1 \rbrace$,
$\lbrace \alpha_1 + \alpha_2 \rbrace$,
$\lbrace \alpha_3 \rbrace$} \end{tabular}\\

\hline

\centering{$2,4,6$} & \centering{$\begin{bmatrix} 1 & -1 & 0 \\ 0 &
1 & 0 \\ 0 & -1 & 1 \end{bmatrix}$} &
\begin{tabular}{p{\wr}} \centering{$\lbrace \alpha_1 \rbrace$,
$\lbrace \alpha_1 {+} \alpha_2 {+} \alpha_3 \rbrace$,
$\lbrace \alpha_3 \rbrace$} \end{tabular}\\

\end{tabular}\\

\hline

2 & $\begin{bmatrix} 1 & 0 & 0 \\ -1 & 1 & 0 \\ 0 & -1 & 1 \\ 1 & -1
& 0 \\ 0 & 1 & -1 \\ 0 & 0 & 1 \end{bmatrix}$ &

\begin{tabular}{p{\wl}|p{\wc}|c}

\centering{$1,2,3$} & \centering{$\begin{bmatrix} 1 & 0 & 0 \\ -1 &
1 & 0 \\ 0 & -1 & 1 \end{bmatrix}$} &
\begin{tabular}{p{\wr}}\centering{$\lbrace \alpha_1 {+} \alpha_2 {+}
\alpha_3 \rbrace$, $\lbrace \alpha_2 {+} \alpha_3 \rbrace$,
$\lbrace \alpha_3 \rbrace$} \end{tabular} \\

\hline

\centering{$1,2,6$} & \centering{$\begin{bmatrix} 1 & 0 & 0 \\
-1 & 1 & 0 \\ 0 & 0 & 1 \end{bmatrix}$} &
\begin{tabular}{p{\wr}}\centering{$\lbrace \alpha_1 {+} \alpha_2
\rbrace$, $\lbrace \alpha_2 \rbrace$, $\lbrace \alpha_3
\rbrace$} \end{tabular}\\

\hline

\centering{$1,5,6$} & \centering{$\begin{bmatrix} 1 & 0 & 0 \\ 0 & 1
& -1 \\ 0 & 0 & 1 \end{bmatrix}$} &
\begin{tabular}{p{\wr}} \centering{$\lbrace \alpha_1 \rbrace$,
$\lbrace \alpha_2 \rbrace$, $\lbrace \alpha_2 {+} \alpha_3 \rbrace$}
\end{tabular}\\

\hline

\centering{$4,5,6$} & \centering{$\begin{bmatrix} 1 & -1 & 0 \\ 0 &
1 & -1 \\ 0 & 0 & 1 \end{bmatrix}$} &
\begin{tabular}{p{\wr}} \centering{$\lbrace \alpha_1 \rbrace$,
$\lbrace \alpha_1 {+} \alpha_2 \rbrace$, $\lbrace \alpha_1 {+}
\alpha_2 {+} \alpha_3 \rbrace$} \end{tabular}\\

\end{tabular}\\

\hline

3 & $\begin{bmatrix} 1 & 0 & 0 \\ 0 & 1 & 1 \\ 1 & -1 & 0 \\ -1 & 1
& -2 \\ 0 & -2 & 1 \end{bmatrix}$ &

\begin{tabular}{p{\wl}|p{\wc}|c}

\centering{$1,2$} & \centering{$\begin{bmatrix} 1 & 0 & 0 \\ 0 & 1 &
1 \\ 0 & 1 & 1 \end{bmatrix}$} &
\begin{tabular}{p{\wr}}\centering{$\lbrace \alpha_1 \rbrace$,
$\lbrace \alpha_2, \alpha_3 \rbrace$} \end{tabular} \\

\hline

\centering{$2,3$} & \centering{$\begin{bmatrix} 1 & -1 & 0 \\
0 & 1 & 1 \\ 0 & 1 & 1 \end{bmatrix}$} &
\begin{tabular}{p{\wr}}\centering{$\lbrace \alpha_1 \rbrace$,
$\lbrace \alpha_1 {+} \alpha_2, \alpha_3 \rbrace$} \end{tabular}\\

\end{tabular}\\

\hline

4 & $\begin{bmatrix} 1 & 0 & 1 \\ 0 & 1 & 0 \\ 1 & -1 & -1 \\ -1 & 1
& -1 \\ -1 & -1 & 1 \end{bmatrix}$ &

\begin{tabular}{p{\wl}|p{\wc}|c}

\centering{$1,2$} & \centering{$\begin{bmatrix} 1 & 0 & 1 \\ 0 & 1 &
0 \\ 1 & 0 & 1 \end{bmatrix}$} &
\begin{tabular}{p{\wr}}\centering{$\lbrace\alpha_1, \alpha_3 \rbrace$,
$\lbrace \alpha_2 \rbrace$} \end{tabular} \\

\hline

\centering{$1,4$} & \centering{$\begin{bmatrix} 1 & 0 & 1 \\
-1 & 1 & -1 \\ 1 & 0 & 1 \end{bmatrix}$} &
\begin{tabular}{p{\wr}}\centering{$\lbrace \alpha_1 {+} \alpha_2,
\alpha_2 {+} \alpha_3 \rbrace$,
$\lbrace \alpha_2 \rbrace$} \end{tabular}\\

\end{tabular}\\

\hline

5 & $\begin{bmatrix} 1 & 1 & 0 \\ 0 & 0 & 1 \\ 1 & -2 & 0 \\ -2 & 1
& -1
\\ 0 & -1 & 1 \end{bmatrix}$ &

\begin{tabular}{p{\wl}|p{\wc}|c}

\centering{$1,2$} & \centering{$\begin{bmatrix} 1 & 1 & 0 \\ 1 & 1 &
0 \\ 0 & 0 & 1 \end{bmatrix}$} &
\begin{tabular}{p{\wr}}\centering{$\lbrace \alpha_1, \alpha_2 \rbrace$,
$\lbrace \alpha_3 \rbrace$} \end{tabular} \\

\hline

\centering{$1,5$} & \centering{$\begin{bmatrix} 1 & 1 & 0 \\
1 & 1 & 0 \\ 0 & -1 & 1 \end{bmatrix}$} &
\begin{tabular}{p{\wr}}\centering{$\lbrace \alpha_1, \alpha_2 {+}
\alpha_3 \rbrace$, $\lbrace \alpha_3 \rbrace$} \end{tabular}\\

\end{tabular}\\

\hline

6 & $\begin{bmatrix} 1 & 0 & 1 \\ -1 & 1 & 0 \\ 1 & -1 & -1 \\ 0 & 1
& -1 \\ -1 & -1 & 1 \end{bmatrix}$ &

\begin{tabular}{p{\wl}|p{\wc}|c}

\centering{$1,2$} & \centering{$\begin{bmatrix} 1 & 0 & 1 \\ -1 & 1
& 0 \\ 1 & 0 & 1 \end{bmatrix}$} &
\begin{tabular}{p{\wr}}\centering{$\lbrace \alpha_1 {+} \alpha_2,
\alpha_3 \rbrace$, $\lbrace \alpha_2 \rbrace$} \end{tabular} \\

\hline

\centering{$1,4$} & \centering{$\begin{bmatrix} 1 & 0 & 1 \\
0 & 1 & -1 \\ 1 & 0 & 1 \end{bmatrix}$} &
\begin{tabular}{p{\wr}}\centering{$\lbrace \alpha_1, \alpha_2 {+}
\alpha_3 \rbrace$, $\lbrace \alpha_2 \rbrace$} \end{tabular}\\

\end{tabular}\\

\hline

7 & $\begin{bmatrix} 1 & -1 & 1 \\ 0 & 1 & 0
\\ 1 & 0 & -1 \\ -1 & 1 & -1 \\ -1 & 0 & 1 \end{bmatrix}$ &

\begin{tabular}{p{\wl}|p{\wc}|c}
\centering{$\vphantom{\begin{bmatrix} 1 & -1 & 1 \\ 0 & 1 & 0
\\ 1 & 0 & -1 \\ -1 & 1 & -1 \\ -1 & 0 & 1 \end{bmatrix}}1,2$} &
\centering{$\begin{bmatrix} 1 & -1 & 1 \\ 0 & 1 & 0 \\
1 & -1 & 1 \\ \end{bmatrix}$} &
\begin{tabular}{p{\wr}} \centering{$\lbrace \alpha_1,
\alpha_3 \rbrace$, $\lbrace \alpha_1 {+} \alpha_2, \alpha_2 {+}
\alpha_3 \rbrace$}\end{tabular}\\

\end{tabular}\\

\hline

8 & $\begin{bmatrix} 1 & 1 & 1 \\ 1 & -2 & -1
\\ -2 & 1 & -2 \\ -1 & -2 & 1 \end{bmatrix}$ &

\begin{tabular}{p{\wl}|p{\wc}|c}
\centering{$\vphantom{\begin{bmatrix} 1 & 1 & 1 \\ 1 & -2 & -1
\\ -2 & 1 & -2 \\ -1 & -2 & 1 \end{bmatrix}}1$} &
\centering{$\begin{bmatrix} 1 & 1 & 1 \\ 1 & 1 & 1 \\
1 & 1 & 1 \\ \end{bmatrix}$} &
\begin{tabular}{p{\wr}} \centering{$\lbrace \alpha_1,
\alpha_2, \alpha_3 \rbrace$}\end{tabular}\\

\end{tabular}\\

\hline

\end{longtable}

\end{document}